\documentclass[a4paper]{amsart}
\usepackage[utf8]{inputenc}
\usepackage[T1]{fontenc}
\usepackage{lmodern}
\usepackage{amssymb}
\usepackage{mathtools}
\usepackage{enumitem}
\usepackage{amsthm}
\usepackage{dsfont}
\usepackage{extpfeil}
\usepackage{xargs}
\usepackage{MnSymbol}

\usepackage{microtype}

\usepackage{tikz,tikz-cd}
\usetikzlibrary{graphs,graphs.standard,quotes}

\usetikzlibrary{matrix,arrows}
\tikzset{cd/.style=matrix of math nodes,row sep=2em,column sep=2em, text height=1.5ex, text depth=0.5ex}
\tikzset{cdar/.style=->,auto}

\tikzset{dar/.style={double,double equal sign distance,-implies}}
\tikzset{mid/.style={anchor=mid}} 

\tikzset{triar/.style={anchor=mid,->}}
\tikzset{tridar/.style={anchor=mid,double,double equal sign distance,-implies}}
\tikzset{narrowfill/.style={inner sep=0pt, fill=white}}

\setlist[enumerate,1]{label=\textup{(\arabic*)}}
\setlist[enumerate,2]{label=\textup{(\alph*)}}
\numberwithin{equation}{section}

\usepackage[
pdftitle={Bicategorical perspective on Steinberg algebras},
pdfauthor={Ralf Meyer, Taufik Yusof, Fabian Rodatz},
pdfsubject={Mathematics}
]{hyperref}

\usepackage{todonotes}
\usepackage{soul}

\usepackage[lite]{amsrefs}

\renewcommand*{\PrintDOI}[1]{\href{http://dx.doi.org/\detokenize{#1}}{doi:
\detokenize{#1}}}

\BibSpec{book}{%
  +{}  {\PrintPrimary}                {transition}
  +{,} { \textit}                     {title}
  +{.} { }                            {part}
  +{:} { \textit}                     {subtitle}
  +{,} { \PrintEdition}               {edition}
  +{}  { \PrintEditorsB}              {editor}
  +{,} { \PrintTranslatorsC}          {translator}
  +{,} { \PrintContributions}         {contribution}
  +{,} { }                            {series}
  +{,} { \voltext}                    {volume}
  +{,} { }                            {publisher}
  +{,} { }                            {organization}
  +{,} { }                            {address}
  +{,} { \PrintDateB}                 {date}
  +{,} { }                            {status}
  +{}  { \parenthesize}               {language}
  +{}  { \PrintTranslation}           {translation}
  +{;} { \PrintReprint}               {reprint}
  +{.} { }                            {note}
  +{.} {}                             {transition}
  +{} { \PrintDOI}                   {doi}
  +{} { available at \url}            {eprint}
  +{}  {\SentenceSpace \PrintReviews} {review}
}

\BibSpec{thesis}{%
    +{}  {\PrintAuthors}                {author}
    +{,} { \textit}                     {title}
    +{:} { \textit}                     {subtitle}
    +{,} { \PrintThesisType}            {type}
    +{,} { }                            {organization}
    +{,} { }                            {address}
    +{,} { \PrintDateB}                 {date}
    +{,} { \eprint}                     {eprint}
    +{,} { }                            {status}
    +{}  { \parenthesize}               {language}
    +{}  { \PrintTranslation}           {translation}
    +{;} { \PrintReprint}               {reprint}
    +{.} { }                            {note}
    +{.} {}                             {transition}
    +{} { \PrintDOI}                   {doi}
    +{}  {\SentenceSpace \PrintReviews} {review}
}

\theoremstyle{plain}
\newtheorem{theorem}[equation]{Theorem}
\newtheorem{lemma}[equation]{Lemma}
\newtheorem{proposition}[equation]{Proposition}
\newtheorem{propdef}[equation]{Proposition and Definition}

\newtheorem{corollary}[equation]{Corollary}

\theoremstyle{definition}
\newtheorem{definition}[equation]{Definition}

\theoremstyle{remark}
\newtheorem{remark}[equation]{Remark}
\newtheorem{example}[equation]{Example}

\newcommand*{\alb}{\hspace{0pt}} 
\newcommand*{\nb}{\nobreakdash}

\newcommand{\Bound}{\mathbb B}
\newcommand{\Mat}{\mathbb M}

\newcommand*{\s}{s} 
\newcommand*{\rg}{r}

\newcommand*{\Cat}[1][C]{\mathcal #1}
\newcommand*{\Bisp}[1][X]{\mathcal #1}
\newcommand{\Gr}[1][G]{\mathcal #1}
\newcommand{\Qu}{\Pi}
\newcommand{\Gtimes}[1][\s,\Gr^0,\rg]{\times_{#1}}

\newcommand*{\op}{\mathrm{op}}
\newcommand*{\prop}{\mathrm{prop}}
\newcommand{\Corr}{\mathfrak{Corr}}

\newcommand*{\ContS}{\mathfrak S}
\newcommand*{\Cont}{\mathrm C}
\newcommand*{\Contc}{\mathrm{C_c}}

\newcommand*{\Hils}[1][H]{\mathcal #1}

\newcommand*{\Cst}{\textup C^*}
\newcommand*{\Star}{$^*$\nobreakdash-}

\newcommand*{\defeq}{\mathrel{\vcentcolon=}}
\newcommand*{\congto}{\xrightarrow\sim}

\newcommand*{\id}{\mathrm{id}}

\newcommand{\C}{\mathbb{C}}
\newcommand{\N}{\mathbb{N}}
\newcommand{\Z}{\mathbb{Z}}

\DeclarePairedDelimiter{\abs}{\lvert}{\rvert}
\DeclarePairedDelimiter{\norm}{\lVert}{\rVert}
\DeclarePairedDelimiterX{\braket}[2]{\langle}{\rangle}{#1\,\delimsize\vert\,\mathopen{}#2}
\DeclarePairedDelimiterX{\setgiven}[2]{\{}{\}}{#1\,{:}\,\mathopen{}#2}

\newcommand*{\conj}[1]{\overline{#1}}

\DeclareMathOperator{\Cone}{Cone}

\DeclareMathOperator{\Hom}{Hom}
\DeclareMathOperator{\Endo}{End}
\DeclareMathOperator\supp{supp}
\DeclareMathOperator{\CovariantRep}{CovRep}

\newcommand{\charmap}[1]{\mathds{1}_{#1}}
\newcommand*{\A}{\mathfrak A}
\newcommand*{\B}{\mathcal B}
\newcommand*{\BXY}{\B_{\Bisp\circ\Bisp[Y]}}

\newcommand*{\CPg}[1][g]{\mathcal{C}_P^{#1}}

\newcommand*{\URing}{\operatorname{\mathtt{Ring}}}

\newcommand*{\Set}{\operatorname{\mathtt{Set}}}
\newcommand*{\Cats}{\operatorname{\mathfrak{Cat}}}
\newcommand*{\Top}{\operatorname{\mathtt{Top}}}

\newcommand*{\F}{\mathcal F} 
\newcommand*{\Hg}[1][g]{H_{{\mathcal{F},#1}}}
\newcommand*{\OO}[1][e]{\mathcal{O}_{#1}}
\newcommand*{\OOO}{\mathcal{O}}

\newcommand*{\w}{w}

\newcommand{\CovRep}[2][D]{\CovariantRep(#1,#2)}
\newcommand{\LaxCovRep}[2][D]{\CovariantRep_{\mathrm{lax}}(#1,#2)}

\newcommand*{\X}{\mathfrak X}

\newcommand*{\HXg}{H_{{\mathcal{\X},g}}}
\newcommand{\GMH}{\mathcal{H}}

\newcommand{\bigslant}[2]{{\raisebox{.2em}{$#1$}\left/\raisebox{-.2em}{$#2$}\right.}}

\newcommand*{\Grcat}{\mathfrak{Gr}} 

\newcommand*{\Rings}{\mathfrak{Rings}}
\newcommand*{\Ringspr}{\mathfrak{Rings}_\prop} 

\begin{document}
\title[A bicategorical perspective on Steinberg algebras]{A
  bicategorical perspective\\on Steinberg algebras}
\author{Ralf Meyer}
\email{rmeyer2@uni-goettingen.de}
\author{Fabian Rodatz}
\email{fabian.rodatz@hhu.de}
\author{Taufik Yusof}
\email{muhammad-taufik.bin-mohd-yusof@mathematik.uni-goettingen.de}
\address{Mathematisches Institut\\
  Universität Göttingen\\
  Bunsenstraße 3--5\\
  37073 Göttingen\\
  Germany}


\begin{abstract}
  We show that the Steinberg algebra construction for ample groupoids
  is part of a pseudofunctor from the bicategory of ample groupoids and
  groupoid correspondences to the bicategory of rings with local units
  and nondegenerate bimodules.
  We define a covariance ring for diagrams in this bicategory of rings
  and show that it is a bicategorical limit.
  We compute the covariance ring for a diagram of ``proper'' bimodules
  over an Ore monoid.
  For diagrams coming from groupoid correspondences, we identify the
  covariance ring with the Steinberg algebra of its groupoid model.
\end{abstract}

\thanks{This work is part of the project Graph Algebras partially
  supported by EU grant HORIZON-MSCA-SE-2021 Project 101086394.
  We thank the mathematical research institute MATRIX and Western
  Sydney University in Australia, where part of this research was
  performed.
  The second author was supported by the DFG Research Training Group
  2240: Algebro-Geometric Methods in Algebra, Arithmetic and
  Topology.
  The third author was supported by a DAAD scholarship.}

\maketitle

\setcounter{tocdepth}{1} 
\tableofcontents

\section{Introduction}
\label{sec:intro}

Several important \(\Cst\)\nb-algebras have purely algebraic analogues.
For instance, a graph \(\Cst\)\nb-algebra contains the Leavitt path algebra of the
graph and a groupoid \(\Cst\)\nb-algebra of an ample groupoid contains the
Steinberg algebra of the groupoid.
The situation for a graph is a special case because the Leavitt path
algebra and the graph \(\Cst\)\nb-algebra of a directed graph are the
Steinberg algebra and the groupoid \(\Cst\)\nb-algebra of the boundary
path groupoid associated to the graph.
Both the boundary path groupoid and the graph \(\Cst\)\nb-algebra enjoy
certain universal properties, which become particularly clear when
phrased in terms of certain bicategories of groupoids and \(\Cst\)\nb-algebras
(see \cites{Albandik:Thesis, Albandik-Meyer:Colimits,
  Albandik-Meyer:Product, Meyer:Diagrams_models,
  Meyer:Groupoid_models_relative, Castro-Meyer:Graph_actors}).
This article adapts that theory to Leavitt path algebras and Steinberg
algebras instead of graph \(\Cst\)\nb-algebras and groupoid \(\Cst\)\nb-algebras.
Some of our main results are the following.
The Steinberg algebra construction is part of a pseudofunctor of
bicategories from a suitable bicategory of ample groupoids and
groupoid correspondences to a bicategory of rings and bimodules.
For a diagram of proper groupoid correspondences over an Ore monoid,
the Steinberg algebra of the groupoid model is isomorphic to the
covariance ring of the associated diagram of Steinberg algebras and
bimodules.
The covariance ring here is characterised by a universal property that
shows that it is a bicategorical limit.
For a diagram of proper bimodules over an Ore monoid, the covariance
ring may be described explicitly using inductive limits and a grading
by the group completion of the Ore monoid.
The structure is the same as in the \(\Cst\)\nb-algebraic case in
\cite{Albandik-Meyer:Product}*{Theorem 3.16}.
A similar structure is also found for the algebraic Cuntz--Pimsner
rings of Carlsen and Ortega~\cite{Carlsen-Ortega:Algebraic_CP} and for
Kumjian--Pask algebras of higher-rank graphs.
Our results apply in the same way to higher-rank self-similar graphs
or ample topological graphs.
In fact, diagrams of proper ample groupoid correspondences are the
right definition for higher-rank self-similar ample groupoids.
Besides higher-rank graphs, we also treat a class of higher-rank
self-similarities of groups to illustrate our theory.

First, in Section~\ref{sec:Rings}, we define the relevant bicategory
of rings, which plays the role of the \(\Cst\)\nb-correspondence bicategory.
The bicategory of unital rings and bimodules introduced by
B\'enabou~\cite{Benabou:Bicategories} is not suitable because we also
need certain rings without unit in order to treat all Leavitt path
algebras and the Steinberg algebras of all ample groupoids.  We cannot
take all nonunital rings as objects either because, at least, we need
the multiplication map to induce an isomorphism
\(R\otimes_R R \cong R\) in order to have unit arrows in our
bicategory.
While it is conceivable that this property suffices, we assume more,
namely, that our rings have local units.
This simplifies several proofs and suffices to treat Steinberg
algebras of ample groupoids.

In Section~\ref{sec:steinmod}, we recall how to define the bicategory
of ample groupoid correspondences, following
\cites{Antunes-Ko-Meyer:Groupoid_correspondences,
  Meyer:Diagrams_models} and we enrich the map that takes an ample
groupoid~\(\Gr\) to its Steinberg algebra \(A_R(\Gr)\) to a
pseudofunctor of bicategories.
This is an algebraic analogue of the pseudofunctor from the bicategory
of \'etale groupoid correspondences to the bicategory of
\(\Cst\)\nb-correspondences defined in
\cite{Antunes-Ko-Meyer:Groupoid_correspondences}*{Section~7}.
We changed notation compared
to~\cite{Antunes-Ko-Meyer:Groupoid_correspondences}, where
pseudofunctors are called ``homomorphisms'' of bicategories,
following~\cite{Benabou:Bicategories}.
We did this because the term pseudofunctor is more common in current
bicategory literature such as~\cite{Johnson-Yau:2-Dim}.
We also prove that proper groupoid correspondences are mapped to
proper bimodules.

In Section~\ref{sec:diagrams}, we study diagrams of rings and
bimodules and their covariance rings.  Diagrams of rings and bimodules
are the purely algebraic analogues of product systems in the realm of
\(\Cst\)\nb-algebras, and the covariance ring plays the role of the
Cuntz--Pimsner algebra of a product system, at least for diagrams of
proper bimodules.
The covariance ring is characterised by a universal property that
involves covariant representations of the diagram.
It is also a bicategorical limit in the bicategory of rings and
bimodules.
We prove that a covariance ring always exists and construct it using a
Cohn localisation.
We also describe it through generators and relations for diagrams of
groupoid correspondences.

Section~\ref{sec:cov_ring_construction} describes the covariance ring
of a proper diagram over an Ore monoid~\(P\) in a way that is
analogous to the description of the Cuntz--Pimsner algebra of a
product system over~\(P\) in \cite{Albandik-Meyer:Product}*{Theorem
  3.16}.
This explicit computation shows that the covariance ring is still a
rather concrete object in this special case.

In Section~\ref{sec:diagrams_in_grcat}, we recall the construction of
the groupoid model of a proper diagram of groupoid correspondences
over an Ore monoid in \cites{Albandik:Thesis, Meyer:Diagrams_models}.
The groupoid model realises a limit of the diagram in the bicategory
of groupoid correspondences.

Section~\ref{sec:steinalg_of_grpmodel} contains the main result of
this article, which identifies the Steinberg algebra of the groupoid
model of a diagram of proper ample groupoid correspondences over an
Ore monoid with the covariance ring of the associated diagram of
proper bimodules of Steinberg algebras over the groupoids in
the diagram.
In particular, the Steinberg algebra pseudofunctor on the bicategory
of proper ample groupoid correspondences preserves limits over Ore
monoids.
Here the rather explicit descriptions of the covariance ring and the
groupoid model for proper diagrams over Ore monoids are used to
identify the two.
We know counterexamples of diagrams where the \(\Cst\)\nb-algebra of the
groupoid model is quite different from the Cuntz--Pimsner algebra of
the corresponding product system (see
\cite{Albandik-Meyer:Colimits}*{Example~3.7} and
\cite{Meyer:Diagrams_models}*{Section~4.4}).
It is clear that this difference remains when we take Steinberg
algebras instead of groupoid \(\Cst\)\nb-algebras.

In Section~\ref{sec:Cstar-consequence}, we use our main result to
prove an analogous result for groupoid \(\Cst\)\nb-algebras.
This reproves the main result of Albandik~\cite{Albandik:Thesis} in
the ample case.
The key idea is that the groupoid \(\Cst\)\nb-algebra of an ample
groupoid is the maximal \(\Cst\)\nb-completion of its Steinberg
algebra.
It is rather easy to characterise when a representation of the
Steinberg algebra becomes a \Star{}representation, and this allows us
to compare its universal property with that of the Cuntz--Pimsner
algebra of the product system associated to the diagram.

Finally, in Section~\ref{sec:eg}, we apply our theory to two classes
of examples.
First, we consider the special case when the ample groupoids are just
ample topological spaces.
Then a groupoid correspondence is the same as a topological
correspondence as defined in~\cite{Albandik-Meyer:Product}.
Diagrams of them are the same as discrete Conduch\'e
fibrations~\cite{Brown-Yetter:Conduche} when the space is discrete.
Discrete Conduch\'e fibrations generalise higher-rank graphs to any
monoid instead of~\(\N^k\) for some \(k\in\N\).
Our theory describes covariance rings and \(\Cst\)\nb-algebras
associated to this data.
For instance, it covers the Kumjian--Pask algebras of regular
higher-rank graphs.
The second class of examples are some higher-rank self-similarities of
discrete groups, coming from injective group endomorphisms with
finite-index range.
These were studied by Stammeier~\cite{Stammeier:Irreversible} using a
different language, and we show that they are examples of covariance
rings of diagrams of groupoid correspondences naturally associated to
his original data.
This was first worked out in the Master's thesis of Daniel Jentsch.

This article shows how to adapt the successful bicategorical
approach to groupoids and \(\Cst\)\nb-algebras in order to also treat
Steinberg algebras and their relatives.
It is remarkable that we can prove these results using only bimodules
as arrows without invoking an involution.
In the usual definition of, say, a Leavitt path algebra, the number of
generators is doubled to take into account the involution.
The extra generators enter differently in our theory.
Some of our ``relations'' stipulate that a module map should be an
isomorphism, so that an inverse is required implicitly.
These inverses provide the missing generators.
This mechanism is clarified in Section~\ref{sec:Cohn_localisation},
where we show that our covariance rings are Cohn localisations of much
simpler rings.
Our results are limited, however, to proper groupoid correspondences
and proper bimodules.
In more general cases, we need the algebraic correspondences of
Carlsen--Ortega~\cite{Carlsen-Ortega:Algebraic_CP} in order to define
suitable generalisations of Leavitt path algebras because the
involution or the extra generators no longer come for free.
This generalisation is pursued in the recent
dissertation~\cite{Taufik:Thesis}.
This article incorporates results obtained by the second author in his
Master's thesis~\cite{Rodatz:Master}.

\section{Rings with local units and smooth bimodules}
\label{sec:Rings}

In this section, we define the bicategory of rings with local units
and smooth bimodules between them, and a subbicategory of proper
bimodules, and we prove some bimodule isomorphisms needed later.

\begin{definition}
  \label{def:local_unit}
  A ring~\(R\) has \emph{local units} if for any finite set
  \(F=\{r_1,\dotsc,r_n\}\subseteq R\), there is an idempotent \(e\in R\)
  such that \(r_i e=r_i=e r_i\) for all \(i=1,\dotsc,n\).  The idempotent
  element~\(e\) is called a \emph{local unit} for~\(F\).
\end{definition}

This definition goes back to \cites{Abrams:Morita_local_units,
  Anh-Marki:Morita_without_identity}.
A prominent example of a ring with local units is the ring
\(\Mat_\infty(R)\) of finite matrices over a unital ring~\(R\).

\begin{lemma}
  \label{lem:local_units_inductive_limit}
  Let~\(R\) be a ring with local unit.
  Then \(R = \varinjlim R\cdot e\) as a left \(R\)\nb-module, where
  the inductive system is indexed by the directed set of idempotents
  in~\(R\).
  Similarly, \(R = \varinjlim e\cdot R\) as a right \(R\)\nb-module
  and \(R = \varinjlim e\cdot R\cdot e\) as a ring.
\end{lemma}

\begin{proof}
  We define an order relation on idempotents in~\(R\) by \(e_1 \le
  e_2\) if \(e_1 \cdot e_2 = e_1 = e_2\cdot e_1\).
  If \(e_1,e_2\in R\) are idempotent, then there is another idempotent
  \(e\in R\) with \(e\cdot e_j = e_j = e_j\cdot e\) for \(j=1,2\)
  because~\(R\) is a ring with local units.
  This makes the set of idempotents in~\(R\) directed.
  By assumption, any element of~\(R\) belongs to~\(e R e\) for some
  idempotent \(e\in R\).
  This implies that \(R = \varinjlim e R e\).
  Here~\(e R e\) are subrings in~\(R\).
  They happen to be unital, but the inclusion homomorphism \(e_1 R e_1
  \to e_2 R e_2\) for \(e_1 \le e_2\) is not unital.
  The statements about \(R e\) and \(e R\) are proven similarly.
\end{proof}

\begin{proposition}
  \label{prop:loc_then_smooth}
  Let~\(R\) be a ring with local units and~\(M\) a left \(R\)\nb-module.
  Let \(\mu_M\colon R\otimes_R M\to M\) be induced by the
  multiplication map.  The following are equivalent:
  \begin{enumerate}
  \item \label{en:loc_then_smooth_1}%
    \(M\) is \emph{smooth}: \(\mu_M\) is an isomorphism
    \(R\otimes_R M\congto M\);
  \item \label{en:loc_then_smooth_2}%
    \(M\) is \emph{nondegenerate}: \(\mu_M\) is surjective;
  \item \label{en:loc_then_smooth_3}%
    for each \(m\in M\) there is \(e\in \mathrm{Idem}(R)\) such that
    \(e\cdot m=m\);
  \item \label{en:loc_then_smooth_4}%
    \(M = \varinjlim e\cdot M\) as an Abelian group, where the
    inductive system is indexed by the directed set of idempotents
    in~\(R\).
  \end{enumerate}
  Analogous equivalences hold for right \(R\)\nb-modules.
\end{proposition}

\begin{proof}
  \ref{en:loc_then_smooth_1}\(\implies\)\ref{en:loc_then_smooth_2} is
  trivial.
  The conditions \ref{en:loc_then_smooth_3}
  and~\ref{en:loc_then_smooth_4} are equivalent because \(m\in M\)
  belong to \(e\cdot M\subseteq M\) if and only if \(e\cdot m = m\).

  We show
  \ref{en:loc_then_smooth_2}\(\implies\)\ref{en:loc_then_smooth_3}.  Let
  \(m\in M\).  Then we may write \(m=\sum_{i=1}^n r_i\cdot m_i\) for some
  \(r_i\in R\), \(m_i\in M\) by~\ref{en:loc_then_smooth_2}.  Let \(e\in R\)
  be a local unit for the finite set \(\{r_1,r_2,\dotsc,r_n\}\).  Then
  \(e\cdot m= m\), verifying~\ref{en:loc_then_smooth_3}.
   
  We show
  \ref{en:loc_then_smooth_3}\(\implies\)\ref{en:loc_then_smooth_1}.
  The hypothesis implies immediately that~\(\mu_M\) is surjective.
  We prove that it is also injective.  Let \(\sum_{i=1}^n r_i\otimes m_i\)
  be in the kernel of~\(\mu_M\), that is,
  \(\sum_{i=1}^n r_i\cdot m_i=0\) in~\(M\).  Let \(e\in R\) be a local unit
  for \(\{r_i\}_{i=1}^n\).  Then
  \[
    \sum_{i=1}^n r_i\otimes m_i
    = \sum_{i=1}^n e r_i\otimes m_i
    = e\otimes\biggl(\sum_{i=1}^n r_i\cdot m_i\biggr)=0
  \]
  in \(R\otimes_R M\).  So \(\ker \mu_M = 0\) as needed.

  The analogous result for right modules is proven in the same way.
\end{proof}

\begin{proposition}
  \label{pro:Rings_bicat}
  There is a bicategory \(\Rings\), which is defined by the following
  data:
  \begin{itemize}
  \item the objects are rings with local units;
  \item the arrows \(R\leftarrow S\) are the smooth \(R ,S\)-bimodules;
  \item the \(2\)\nb-arrows \(M\Rightarrow N\) for two smooth
    \(R ,S\)-bimodules \(M,N\) are the \(R ,S\)-bimodule homomorphisms
    \(f\colon M\to N\);
  \item the vertical product is the composition of \(R ,S\)-bimodule
    homomorphisms;
  \item the composition of arrows is the balanced tensor product; its
    functoriality gives the horizontal product of \(2\)\nb-arrows;
  \item the unit arrow on~\(R\) is~\(R\) with multiplication as
    \(R,R\)-bimodule structure;
  \item the \emph{associators} are the canonical bimodule
    isomorphisms
    \[
      (M_1\otimes_R M_2)\otimes_S M_3 \to M_1\otimes_R (M_2\otimes_S M_3),\qquad
      (m_1\otimes m_2)\otimes m_3 \mapsto m_1 \otimes (m_2\otimes
      m_3);
    \]
  \item the \emph{uniters} are the canonical multiplication maps
    \[
      R\otimes_R M \to M,\quad r\otimes m \mapsto r\cdot m,\qquad
      M\otimes_S S \to M,\quad m\otimes t \mapsto m\cdot t;
    \]
    these are isomorphisms because~\(M\) is a smooth bimodule.
  \end{itemize}
\end{proposition}

\begin{proof}
  This is straightforward to check.
\end{proof}

Many of our results require bimodules with extra properties that make
them analogues of \emph{proper} \(\Cst\)\nb-correspondences between
\(\Cst\)\nb-algebras.  To formulate these, we first need the following
notation:

\begin{definition}
  \label{def:fgp}
  Let~\(R\) be a ring with local units.
  A smooth right \(R\)\nb-module~\(M\) is
  \begin{itemize}
  \item \emph{finitely generated} if there are finitely many
    \(m_1,m_2,\dotsc,m_k\in M\) such that~\(M\) is spanned by them as a
    right \(R\)-module;
  \item \emph{projective} if any surjective module homomorphism
    onto~\(M\) splits by a module homomorphism;
  \item \emph{fgp} if it is both finitely generated and projective.
  \end{itemize}
\end{definition}

The following lemma generalises a well-known result for unital rings:

\begin{lemma}
  \label{lem:fgp}
  Let~\(R\) be a ring with local units.
  A right \(R\)\nb-module is fgp if and only if it is isomorphic to
  \(e\cdot R^k\) for some idempotent \(e\in \Mat_k(R)\) and \(k\in\N\);
  here we treat elements of~\(R^k\) as column vectors with
  \(\Mat_k(R)\) acting on the left by matrix--vector multiplication.
\end{lemma}

\begin{proof}
  Let~\(M\) be a right \(R\)\nb-module.
  First assume~\(M\) to be fgp.
  Since~\(M\) is finitely generated, there are elements
  \(m_1,\dotsc,m_k\) so that the map
  \[
    \pi\colon R^k \to M,\qquad
    (r_1,\dotsc,r_k)\mapsto
    \sum_{j=1}^k m_j r_j,
  \]
  is surjective.
  Since~\(M\) is projective, \(\pi\) splits by a module homomorphism
  \(\sigma\colon M \to R^k\).
  The map \(\sigma\circ\pi\colon R^k \to R^k\) multiplies by the
  \(k\times k\)-matrix~\(e\) with columns \(\sigma(m_j)\in R^k\),
  \(j=1,\dotsc,k\).
  Since \(\pi\circ\sigma = \id_M\), the map~\(\sigma\pi\) is
  idempotent, and then~\(e\) is idempotent as a matrix.
  The map~\(\pi\) restricts to an isomorphism from the image \(e\cdot
  R^k\) of~\(\sigma\pi\) onto~\(M\).

  Conversely, let \(M\cong e\cdot R^k\) for an idempotent
  \(e\in\Mat_k(R)\).
  Let \(e_1,\dotsc,e_k\in R^k\) be the columns of~\(e\).
  These belong to~\(M\) because \(e\cdot e_j = e_j\), and they form a
  finite generating set for~\(M\).
  Let~\(N\) be another right \(R\)\nb-module and let \(q\colon N\to
  M\) be a surjective \(R\)\nb-module map.
  Then there are \(\hat{e}_1,\dotsc,\hat{e}_k\in N\) with
  \(q(\hat{e}_j) = e_j\).
  The module map \(R^k \to N\), \((r_1,\dotsc,r_k) \mapsto
  \sum_{j=1}^k \hat{e}_j r_j\), restricts to a map \(\sigma\colon M
  =e\cdot R^k \to N\).
  The composite \(q\circ\sigma\colon M\to M\) maps \(x\mapsto e\cdot
  x\) for all \(x\in M\subseteq R^k\), so that it is the identity
  on~\(M\).
  Thus~\(\sigma\) is a section for~\(q\).
\end{proof}

\begin{definition}
  \label{def:properbimod}
  Let \(R\) and~\(S\) be rings with local units.
  A smooth \(R,S\)-bimodule~\(P\) is called \emph{proper} if \(e\cdot
  P\) is fgp for each idempotent \(e\in R\).
\end{definition}

\begin{example}
  \label{ex:ring_is_fgp}
  The identity bimodule~\(S\) over a ring with local unit is proper by
  Lemma~\ref{lem:fgp} (with \(k=1\)).
\end{example}

\begin{example}
  \label{ex:graphmod}
  Let \(G=(E,V)\) be a directed graph, with range and source maps
  \(\rg,\s\colon E\rightrightarrows V\).
  Let~\(\mathbb K\) be a commutative unital ring.
  Let \(R\defeq \bigoplus_{v\in V}\mathbb K\), viewed as functions
  \(V\to \mathbb K\) with finite support.
  This is a ring with pointwise multiplication.
  It has local units even if~\(V\) is infinite.
  Let \(M\defeq \bigoplus_{e\in E}\mathbb K\) as a free \(\mathbb
  K\)\nb-module.
  We write \(\delta_v\in R\) and \(\delta_e \in M\) for the
  characteristic function of \(v\in V\) and \(e\in E\), respectively.
  Let \(\delta_{v=\rg(e)}\) be~\(1\) if \(v=\rg(e)\) and~\(0\)
  otherwise (Kronecker~\(\delta\)), and similarly for the condition
  \(v=\s(e)\) and other statements.
  We give~\(M\) the unique smooth \(R\)\nb-bimodule structure with
  \[
    \delta_v \delta_e = \delta_{v=\rg(e)} \cdot \delta_e,\qquad
    \delta_e \delta_v = \delta_{v=\s(e)} \cdot \delta_e.
  \]
  By the way, any nondegenerate \(R\)\nb-bimodule~\(M\) may be brought
  into this form by choosing bases for \(\delta_v M \delta_w\) for all
  \(v,w\in R\).
  Since any idempotent in~\(R\) is a finite sum of~\(\delta_v\) for
  \(v\in V\), this bimodule is proper if and only if~\(\delta_v M\) is
  fgp for all \(v\in V\).
  Here~\(\delta_v M\) has \(\delta_e\) with \(e\in \rg^{-1}(v)\) as a
  basis.
  As a consequence, \(M\) is a proper \(R\)\nb-bimodule if and only
  if~\(G\) is \emph{row-finite}, that is, \(\rg\colon E \to V \) is
  finite-to-one.
\end{example}

The following theorem is the main reason why we need fgp bimodules.
Recall that if~\(M\) is an \(R,T\)\nb-bimodule and~\(P\) is an
\(S,T\)-bimodule, then there is a canonical \(S,R\)\nb-bimodule structure
on \(\Hom_{-,T}(M,P)\) defined by
\((s\cdot f\cdot r)(m)\defeq s\cdot \bigl(f(r\cdot m)\bigr)\).  This is
usually not smooth, even if \(M\) and~\(P\) are smooth bimodules.

\begin{theorem}
  \label{thm:fgp_is_nice}
  Let \(R,S,T,U\) be rings with local units, let~\(M\) be a smooth
  \(R,T\)-bimodule, let~\(N\) be a smooth \(U,S\)-bimodule, and let~\(P\) be a
  smooth \(S,T\)-bimodule.  Give \(\Hom_{-,T}(M,P)\) and
  \(\Hom_{-,T}(M,N \otimes_S P)\) the canonical \(S,R\)- and
  \(U,R\)-bimodule structures.
  There is a canonical \(U,R\)\nb-bimodule homomorphism
  \[
    N\otimes_S \Hom_{-, T}(M,P) \to \Hom_{-,T}(M,N \otimes_S P),
    \qquad
    n\otimes f \mapsto \bigl[m\mapsto n\cdot f(m)\bigr].
  \]
  It is an isomorphism if~\(M\) is fgp as a \(T\)\nb-module or~\(N\) is
  fgp as an \(S\)\nb-module.
\end{theorem}

\begin{proof}
  We first assume~\(M\) to be fgp as a \(T\)\nb-module.  We use
  Lemma~\ref{lem:fgp} to identify
  \(M\cong e \cdot T^k = e\cdot \Mat_{k,1}(T)\) for some idempotent
  \(e\in \Mat_k(T)\).  Then any \(T\)\nb-module map
  \(M\to N \otimes_S P\) extends canonically to~\(T^k\) by composing
  with the projection \(T^k \to M\), \(x\mapsto e\cdot x\).
  Conversely, a \(T\)\nb-module map \(f\colon T^k\to N \otimes_S P\)
  is such a composite if and only if \(f\circ e = f\), where we
  view~\(e\) as an endomorphism of~\(T^k\).  The same applies to maps
  \(M\to P\).  Here a right \(T\)\nb-module map \(f\colon T^k\to P\)
  with \(f\circ e = f\) must be of the form
  \((t_j) \mapsto \sum_{j=1}^k t_j p_j\) with some~\(p_j \in P\);
  namely, \(p_j\) is the image of the \(j\)th column of~\(e\).  Thus
  \(\Hom_{-,T}(M, P) \cong \Mat_{1,k}(P)\cdot e\), where we use the
  canonical right \(\Mat_k(T)\)-module structure on \(\Mat_{1,k}(P)\).
  Then
  \[
    N\otimes_S \Hom_{-,T}(M, P)
    \cong N\otimes_S \Mat_{1,k}(P)\cdot e
    \cong \Mat_{1,k}(N \otimes_S P)\cdot e.
  \]
  While a general \(T\)\nb-module map \(T\to N \otimes_S P\) need not
  be of the form \(t\mapsto x\cdot t\) for some \(x\in N \otimes_S
  P\), this becomes so after multiplication with an element of~\(T\).
  Therefore, any \(T\)\nb-module map \(f\colon T^k\to N \otimes_S P\)
  with \(f\circ e = f\) must be of the form \((t_i)\mapsto
  \sum_{i=1}^k t_i x_i\) for unique \(x_1,\dotsc,x_k\in N \otimes_S
  P\).
  Then
  \[
    \Hom_{-,T}(M,N \otimes_S P)
    \cong \Mat_{1,k}(N \otimes_S P)\cdot e
    \cong N\otimes_S \Hom_{-,T}(M, P).
  \]

  Next, we assume~\(N\) to be fgp, that is, \(N\cong f \cdot S^k\) for
  some idempotent \(f\in\Mat_k(S)\).  As above, we identify
  \(N\otimes_S P \cong f\cdot P^k\) and
  \begin{align*}
    N\otimes_S \Hom_{-, T}(M,P)
    &\cong f\cdot \Hom_{-,T}(M,P)^k
    \\&\cong \Hom_{-,T}(M,f\cdot P^k)
    \cong \Hom_{-,T}(M,N \otimes_S P).\qedhere
  \end{align*}
\end{proof}

In the situation of Theorem~\ref{thm:fgp_is_nice}, let
\(\Hom_{-,T}(M,P)R\) be the isomorphic image of
\(\Hom_{-,T}(M,P)\otimes_RR\) under the scalar multiplication map.
Equivalently, it is the largest right \(R\)-submodule of
\(\Hom_{-,T}(M,P)\) that is nondegenerate (or smooth) as a right
\(R\)\nb-module.
Define \(\Hom_{-,T}(M,N \otimes_S P) R\) and \(U\Hom_{-,T}(M,N
\otimes_S P)\) similarly.

\begin{theorem}
  \label{thm:proper_is_nice}
  Let \(R,S,T,U\) be rings with local units, let~\(M\) be a smooth
  \(R,T\)-bimodule, let~\(N\) be a smooth \(U,S\)-bimodule, and let~\(P\) be a
  smooth \(S,T\)-bimodule.  Give \(\Hom_{-,T}(M,P)\) and
  \(\Hom_{-,T}(M,N \otimes_S P)\) the canonical \(S,R\)- and \(U,R\)-bimodule
  structures.  If~\(M\) is proper, then the canonical map in
  Theorem~\textup{\ref{thm:fgp_is_nice}} restricts to an isomorphism
  \[
    N\otimes_S \Hom_{-, T}(M,P) R \congto \Hom_{-,T}(M,N \otimes_S P) R.
  \]
  If~\(N\) is proper, then the canonical map in
  Theorem~\textup{\ref{thm:fgp_is_nice}} restricts to an isomorphism
  \[
    N\otimes_S \Hom_{-, T}(M,P)\congto U\Hom_{-,T}(M,N \otimes_S P).
  \]
\end{theorem}

\begin{proof}
  We assume~\(M\) to be proper and prove the first isomorphism.
  Lemma~\ref{lem:local_units_inductive_limit} implies
  \[
    \Hom_{-,T}(M,N \otimes_S P)R
    \cong \varinjlim \Hom_{-,T}(M,N \otimes_S P)\cdot e
    \cong \varinjlim \Hom_{-,T}(e\cdot M,N \otimes_S P),
  \]
  where~\(e\) runs through a local unit in~\(R\).
  Similarly, \(\Hom_{-, T}(M,P) R\) is the inductive limit of
  \(\Hom_{-,T}(e\cdot M,P)\).
  For fixed~\(e\), Theorem~\ref{thm:fgp_is_nice} implies an
  isomorphism
  \[
    N\otimes_S \Hom_{-,T}(e\cdot M,P)
    \cong \Hom_{-,T}(e\cdot M,N \otimes_S P).
  \]
  These isomorphisms for all~\(e\) are natural and thus induce the
  desired isomorphism on the inductive limits.

  The second isomorphism is proven similarly.  Writing both sides as
  inductive limits over idempotents~\(f\) in~\(U\), we see that it
  suffices to prove that the canonical map is an isomorphism
  \(f N\otimes_S \Hom_{-, T}(M,P)\congto \Hom_{-,T}(M,f N \otimes_S
  P)\), and this follows from Theorem~\ref{thm:fgp_is_nice}
  because~\(f N\) is fgp as an \(S\)\nb-module.
\end{proof}

We want to prove that the proper smooth bimodules form a subbicategory
of the bicategory~\(\Rings\).  We already noted that the identity
bimodules are proper.  The remaining thing to check is the following
lemma:

\begin{lemma}
  \label{lem:tensor=proper}
  A tensor product of proper smooth bimodules is also proper.
\end{lemma}

\begin{proof}
  Let \(R,S,T\) be rings with local units and let \(M={_RM}_S\) and
  \(N={_SN}_T\) be proper smooth bimodules.  We claim that \(M\otimes_SN\)
  is proper, that is, \(e\cdot(M\otimes_SN)\) is fgp as a right
  \(T\)-module for every idempotent \(e\in R\).  Since \(e\cdot M\) is an fgp
  right \(S\)-module, \(e\cdot M \cong f\cdot S^k\) for some idempotent
  \(f\in \Mat_k(S)\).  Then
  \[
    e\cdot (M\otimes_SN) \cong (e\cdot M)\otimes_S N
    \cong f\cdot S^k \otimes_S N
    \cong f\cdot N^k,
  \]
  where~\(f\) acts on~\(N^k\) by matrix-vector multiplication and the
  left \(S\)\nb-module structure on~\(N\).  Since~\(S\) is a ring with
  local units, there is an idempotent \(g\in S\) that is a local unit
  for each entry of~\(f\).  Therefore, \(f\cdot N^k\) is contained in
  \((g\cdot N)^k\).  Even more, it is the direct summand of
  \((g\cdot N)^k\) given by left multiplication with~\(f\).
  Since~\(N\) is proper, \(g\cdot N\) is an fgp right \(T\)\nb-module.
  Then \((g\cdot N)^k\) is fgp as well, and so is any direct summand
  of it.  Thus \(e\cdot (M\otimes_SN)\) is an fgp right
  \(T\)\nb-module as asserted.
\end{proof}

\begin{definition}
  Let \(\Ringspr \subseteq \Rings\) be the subbicategory with the
  same objects and \(2\)\nb-arrows and with only the proper bimodules
  as arrows.
\end{definition}

\section{Steinberg bimodules and the pseudofunctor to rings}
\label{sec:steinmod}

In this section, we define a pseudofunctor from the bicategory of
(ample) groupoid correspondences to the bicategory of rings and
bimodules, which maps an ample groupoid~\(\Gr\) to its Steinberg
algebra \(A_R(\Gr)\).
This is an algebraic analogue of the pseudofunctor from the bicategory
of (\'etale) groupoid correspondences to the bicategory of
\(\Cst\)\nb-correspondences defined in
\cite{Antunes-Ko-Meyer:Groupoid_correspondences}*{Section~7}.
We also prove that our pseudofunctor  maps proper groupoid
correspondences to proper bimodules.
As a prerequisite, we recall the variant of the bicategory of groupoid
correspondences that we need for this construction.
We are particularly interested in certain families of compact open
Hausdorff subsets that we call ample bases because these will be used
to describe Steinberg algebras and bimodules.

\subsection{The bicategory of ample groupoid correspondences}
\label{sec:ample_groupoid_corr}

Slightly different variants of the bicategory of groupoid
correspondences have been defined in
\cites{Antunes-Ko-Meyer:Groupoid_correspondences,
  Meyer:Diagrams_models}.
Here, we need yet another variant because the Steinberg algebra only
works for groupoids with totally disconnected unit space.

All topological groupoids in this article are assumed to be
\emph{ample}, that is, their source and range maps \(\s,\rg\colon \Gr
\rightrightarrows \Gr^0\) are local homeomorphisms, \(\Gr^0\) is
Hausdorff, and each point in~\(\Gr^0\) has a compact open
neighbourhood.
Hence, the Hausdorff, compact open subsets form a base for the
topologies both on \(\Gr^0\) and~\(\Gr\).
We do not require~\(\Gr\) to be Hausdorff.
A \emph{slice} or bisection is an open subset~\(U\) of~\(\Gr\) such
that \(\s|_U\) and~\(\rg|_U\) are injective.
The set of compact slices in~\(\Gr\) is a base for the topology
of~\(\Gr\) consisting of compact, open, Hausdorff subsets.
We multiply compact slices as subsets of~\(\Gr\).
This gives an inverse semigroup with unit and zero.
For the construction of the Steinberg algebra of~\(\Gr\), it is
important that the family of compact slices is a base for the topology
with the following extra property:

\begin{definition}
  \label{def:ample_base}
  An \emph{ample base} for a space~\(X\) is a base~\(\B\)
  of the topology consisting of compact, Hausdorff, open subsets
  of~\(X\) such that \(U\setminus V \in \B\) whenever \(U,V\in\B\).
\end{definition}

Notice that \(U\setminus V\) is again compact, Hausdorff, open if
\(U\) and~\(V\) are so.
The base~\(\B_{\Gr}\) of all compact slices in~\(\Gr\) and most bases
that we shall use in the following have the stronger property that any
compact open subset~\(U\) with \(U\subseteq V\) for some
\(V\in\B_{\Gr}\) belongs to~\(\B_{\Gr}\).
One use of a base is to produce a presentation of the Steinberg
algebra.
In this context, it is useful to have as few sets in~\(\B\) as
possible and to define bases assuming the weaker property in
Definition~\ref{def:ample_base}.

We define left and right actions of groupoids as usual.
By convention, we write~\(\rg\) for the anchor map of a left action
and~\(\s\) for the anchor map of a right action.
This ensures that a product \(x\cdot y\) of any kind is defined if and
only if \(\s(x) = \rg(y)\).
We may write \(\s_X\) or~\(\rg_X\) for the anchor maps of an action
on~\(X\) to avoid ambiguity.

Let \(\Gr\) and~\(\Gr[H]\) be ample groupoids.
An \emph{ample groupoid correspondence} \(\Bisp\colon \Gr[H]\leftarrow
\Gr\) is a space~\(\Bisp\) with commuting actions of~\(\Gr[H]\) on the
left and~\(\Gr\) on the right, such that the right \(\Gr\)\nb-action
is free and proper and its anchor map \(\s\colon \Bisp \to \Gr^0\) is
a local homeomorphism.

This is the same definition as
in~\cite{Antunes-Ko-Meyer:Groupoid_correspondences}.
As shown in~\cite{Antunes-Ko-Meyer:Groupoid_correspondences}, it
follows that the orbit space \(\Bisp/\Gr\) is Hausdorff and that the
orbit space projection \(\Qu\colon \Bisp \to \Bisp/\Gr\) is a local
homeomorphism.
The left anchor map of a groupoid correspondence descends to a
continuous map \(\rg_*\colon \Bisp/\Gr \to \Gr[H]^0\).
As in~\cite{Antunes-Ko-Meyer:Groupoid_correspondences}, we call the
correspondence \emph{proper} if~\(\rg_*\) is proper, and \emph{tight}
if~\(\rg_*\) is a homeomorphism.
Any tight groupoid correspondence is proper, and both conditions
define subbicategories, that is, the identity groupoid correspondences
are tight and the composite of two tight or proper groupoid
correspondences is again tight or proper, respectively.

An open subset \(U\subseteq\Bisp\) is called a \emph{slice} or
bisection if both \(\Qu|_U\) and~\(\s|_U\) are injective.
Then \(\s|_U\) and~\(\Qu|_U\) are homeomorphisms onto open subsets of
\(\Gr^0\) and~\(\Bisp/\Gr\), respectively, and~\(U\) is Hausdorff.
It is shown in \cites{Antunes-Ko-Meyer:Groupoid_correspondences,
  Meyer:Diagrams_models} that these slices form a base for the
topology on~\(\Bisp\).
In addition, since any neighbourhood in~\(\Gr^0\) contains a compact
open neighbourhood, any neighbourhood of \(x\in\Bisp\) contains a
neighbourhood of~\(x\) that is a compact open slice.
Thus the set of compact open slices is an ample base
for~\(\Bisp\).
Since~\(\Qu\) is a local homeomorphism, the images~\(\Qu(U)\) of compact
open slices \(U\subseteq \Bisp/\Gr\) form an ample base for~\(\Bisp/\Gr\).

Let \(\Bisp\colon \Gr[H] \leftleftarrows \Gr\) be a groupoid
correspondence and let \(\Qu\colon \Bisp\to\Bisp/\Gr\) be the orbit
space projection.
Recall that \(\braket{x}{y}\) for \(x,y\in \Bisp\) with \(\Qu(x)=\Qu(y)\)
is the unique \(g\in\Gr\) with \(\rg(g)= \s(x)\) and \(x\cdot g = y\).
If \(U_1,U_2\subseteq \Bisp\) are slices, let
\[
  \braket{U_1}{U_2} \defeq
  \setgiven{ \braket{x}{y}}{u_1\in U_1,\ u_2\in U_2,\ \Qu(u_1)=\Qu(u_2)}.
\]
If  \(V\subseteq \Gr[H]\), \(U\subseteq \Bisp\), \(W\subseteq \Gr\)
are slices, define the products \(V U, U W\subseteq \Bisp\) similarly
as the sets of all products \(v u\) or~\(u w\), respectively, with
\(v\in V\), \(u\in U\), \(w\in W\) and \(\s(v) = \rg(u)\) or \(\s(u) =
\rg(w)\).

\begin{lemma}
  \label{lem:operations_on_compact_slices}
  Let \(\Bisp\colon \Gr[H] \leftleftarrows \Gr\) be a groupoid
  correspondence.
  Let \(V\subseteq \Gr[H]\), \(U,U_1,U_2\subseteq \Bisp\),
  \(W\subseteq \Gr\) be compact slices.
  Then \(V U, U W \subseteq \Bisp\) and \(\braket{U_1}{U_2} \subseteq
  \Gr\) are again compact slices.
  Let \(\Bisp[Y]\) be any \(\Gr\)\nb-space.
  If \(W\subseteq \Gr\) is a compact slice and \(Z\subseteq \Bisp[Y]\)
  is a compact, Hausdorff, open subset, then~\(W Z\) is again a
  compact, Hausdorff, open subset of~\(\Bisp[Y]\).
\end{lemma}

\begin{proof}
  The subsets \(V U\), \(U W\) and \(\braket{U_1}{U_2}\)  are slices by
  \cite{Antunes-Ko-Meyer:Groupoid_correspondences}*{Lemma~7.4}.
  In particular, they are Hausdorff and open.
  In addition, since \(\Gr^0\) and~\(\Bisp/\Gr\) are Hausdorff, the
  sets of pairs \((u_1,u_2)\), \((u,v)\) or \((v,w)\) for which
  \(\braket{u_1}{u_2}\), \(u\cdot v\) or \(v\cdot w\) are defined are
  closed subsets of the product spaces \(U_1 \times U_2\), \(U \times
  V\), and \(V\times W\), respectively.
  These products of compact spaces are again compact, and so are their
  closed subsets.
  Since the multiplication and the bracket map are continuous, it
  follows that \(V U\), \(U W\), and \(\braket{U_1}{U_2}\) are again
  compact.

  Next we prove that \(W Z\subseteq \Bisp[Y]\) is open, Hausdorff and
  compact.
  The multiplication map \(\mu\colon \Gr\times_{\Gr^0} \Bisp[Y] \to
  \Bisp[Y]\) of any groupoid action is a local homeomorphism by
  \cite{Antunes-Ko-Meyer:Groupoid_correspondences}*{Lemma~2.9}.
  Therefore, \(W Z = \mu(W\times_{\Gr^0} Z)\) is open in~\(\Bisp[Y]\).
  Since~\(\s|_W\) is a homeomorphism onto \(\s(W)\), the projection
  \(W\times_{\Gr^0} Z \to \Bisp[Y]\), \((w,y)\mapsto y\), is a
  homeomorphism onto \(\rg_{\Bisp[Y]}^{-1}(\s(W))\).
  So \(W\times_{\Gr^0} Z\) is a compact Hausdorff space.
  The restriction of~\(\mu\) to \(W\times_{\Gr^0} Z\) is an injective
  local homeomorphism, hence a homeomorphism.
  So \(W Z\) is also compact and Hausdorff.
\end{proof}

For two ample groupoid correspondences \(\Bisp,\Bisp[Y]\colon \Gr[H]
\leftleftarrows \Gr\), a \emph{\(2\)\nb-arrow}
\(\Bisp\Rightarrow\Bisp[Y]\) is a continuous
\(\Gr[H],\Gr\)-equivariant map \(\varphi\colon \Bisp\to\Bisp[Y]\).
In~\cite{Antunes-Ko-Meyer:Groupoid_correspondences}, the
map~\(\varphi\) is assumed to be injective.
This is needed for it to induce an \emph{isometry} between the
\(\Cst\)\nb-correspondences induced by \(\Bisp\) and~\(\Bisp[Y]\).
Since our bicategory of rings allows all bimodule maps as
\(2\)\nb-arrows, we may allow more \(2\)\nb-arrows of groupoid
correspondences.
This change does not matter much because our constructions only need
invertible \(2\)\nb-arrows.

The composition \(\Bisp \circ \Bisp[Y]\) of two groupoid
correspondences \(\Bisp\colon \Gr[H] \leftarrow \Gr\) and
\(\Bisp[Y]\colon \Gr\leftarrow \Gr[K]\) is defined as
in~\cite{Antunes-Ko-Meyer:Groupoid_correspondences}, as the orbit
space of the canonical diagonal action of~\(\Gr\) on the fibre product
space \(\Bisp \Gtimes \Bisp[Y]\).
This is a groupoid correspondence \(\Gr[H] \leftarrow \Gr[K]\).
It is automatically ample because \(\Gr[H]\) and~\(\Gr[K]\) are ample.
The same arguments as in
\cites{Antunes-Ko-Meyer:Groupoid_correspondences,
  Meyer:Diagrams_models} show that ample groupoids with ample groupoid
correspondences and the \(2\)\nb-arrows above form a
bicategory~\(\Grcat\).
We sometimes write \(\Bisp\circ_{\Gr} \Bisp[Y]\) if we need to
mention~\(\Gr\).

The construction of \(\Bisp\circ \Bisp[Y]\) still works in the same
way if~\(\Bisp[Y]\) is replaced by a topological space with a left
action of~\(\Gr\), without any other groupoid acting on~\(\Bisp[Y]\)
on the right.
In this generality, it is true that the \(\Gr\)\nb-action on \(\Bisp
\Gtimes \Bisp[Y]\) is basic, so that the orbit space projection
\(\Bisp \Gtimes \Bisp[Y] \to \Bisp \circ \Bisp[Y]\) is a local
homeomorphism.
(This action need not be proper, however, even if~\(\Bisp[Y]\) is also
a groupoid correspondence.
Equivalently, \(\Bisp \circ \Bisp[Y]\) need not be Hausdorff.)
We will use the spaces \(\Bisp\circ \Bisp[Y]\) later to construct the
groupoid model.
The following lemma yields a standard ample base
in~\(\Bisp\circ\Bisp[Y]\), given ample bases in \(\Bisp\),
\(\Bisp[Y]\) and~\(\Gr\) that are suitably compatible:

\begin{lemma}
  \label{lem:base_of_comp}
  Let \(\Bisp\colon\Gr[H]\leftarrow\Gr\) be an ample correspondence
  and let~\(\Bisp[Y]\) be a left \(\Gr\)\nb-space.
  Let \(\pi\colon\Bisp\Gtimes\Bisp[Y]\to \Bisp\circ\Bisp[Y]\)
  denote the orbit space projection of the diagonal action.
  Let \(\B_{\Gr}\), \(\B_{\Bisp}\), and \(\B_{\Bisp[Y]}\) be ample
  bases in \(\Gr\), \(\Bisp\), and \(\Bisp[Y]\), respectively.
  Assume that all elements of \(\B_{\Gr}\) and~\(\B_{\Bisp}\) are
  compact slices and that \(U W\in\B_{\Bisp}\), \(W
  V\in\B_{\Bisp[Y]}\), and \(\braket{U_1}{U_2}\in \B_{\Gr}\) if \(U,
  U_1,U_2 \in\B_{\Bisp}\), \(W\in\B_{\Gr}\), and
  \(V\in\B_{\Bisp[Y]}\).
  Then the restriction of~\(\pi\) to \(U\Gtimes V\) is a homeomorphism
  onto \(U V\defeq\pi(U\Gtimes V) \subseteq
  \Bisp\circ\Bisp[Y]\) for all \(U\in\B_{\Bisp}\),
  \(V\in\B_{\Bisp[Y]}\).
  These sets \(U V\) are open, Hausdorf, and compact, and form an
  ample base in \(\Bisp\circ\Bisp[Y]\).
  In addition,
  \[
    \BXY \defeq \setgiven{U V}{U\in\B_{\Bisp},\ V\in\B_{\Bisp[Y]}}
    = \setgiven{U V}{U\in\B_{\Bisp},\
      V\in\B_{\Bisp[Y]},\ \s(U)\supseteq \rg(V)}.
  \]
\end{lemma}

\begin{proof}
  Let \(U\in\B_{\Bisp}\) and \(V\in\B_{\Bisp[Y]}\).
  We first prove that \(U V\subseteq \Bisp\circ\Bisp[Y]\) is open,
  Hausdorff, and compact.
  The subset \(U\times_{\Gr^0} V\subseteq\Bisp\Gtimes\Bisp[Y]\)
  is open.
  It is Hausdorff and compact as a closed subset of the Hausdorff
  compact space \(U\times V\).
  We claim that the restriction of~\(\pi\) to \(U\Gtimes V\) is
  injective.
  Indeed, let \(\pi(u_1,v_1) = \pi(u_2,v_2)\) for some \(u_1,u_2\in
  U\), \(v_1,v_2\in V\) with \(\s(u_j) = \rg(v_j)\) for \(j=1,2\).
  Then there is \(g\in\Gr\) with \(u_2 = u_1 g\), \(v_1 = g v_2\).
  Since the orbit space projection on~\(U\) is injective, \(u_2 = u_1
  g\) implies \(u_1 = u_2\) and \(g=1_{\s(u_1)}\).
  Then \(v_1 = v_2\) follows.
  Since~\(\pi\) is injective on the open set \(U\times_{\Gr^0} V\) and
  a local homeomorphism, it restricts to a homeomorphism from
  \(U\times_{\Gr^0} V\) onto an open subset~\(U V\) of
  \(\Bisp\circ\Bisp[Y]\).
  So~\(U V\) is open, Hausdorff and compact.

  If \(U\in\B_{\Bisp}\), \(V\in\B_{\Bisp[Y]}\), then \(\s(U) =
  \braket{U}{U}\in \B_{\Gr}\) and \(V' \defeq \braket{U}{U} V\in
  \B_{\Bisp[Y]}\).
  We compute  \(V' = \rg_{\Bisp[Y]}^{-1}(\s(U)) \cap V\), so that \(U
  V = U V'\) and \(\rg(V') \subseteq \s(U)\).
  So all elements of~\(\BXY\) are of the form~\(U V\) for
  \(U\in\B_{\Bisp}\), \(V\in\B_{\Bisp[Y]}\), with \(\s(U) \supseteq
  \rg(V)\) as claimed.

  Let \(W\subseteq \Bisp\circ \Bisp[Y]\) be an open subset and
  let \(\pi(x,y)\in W\).
  We claim that~\(W\) contains a neighbourhood of~\(\pi(x,y)\) of the
  form~\(U V\).
  The subset \(\pi^{-1}(W) \subseteq\Bisp\Gtimes \Bisp[Y]\) is open.
  So there is an open subset \(\tilde{W}\subseteq
  \Bisp\times\Bisp[Y]\) with \(\pi^{-1}(W) = \tilde{W}\cap
  \Bisp\Gtimes \Bisp[Y]\).
  Since \((x,y)\in \tilde{W}\) and \(\B_{\Bisp}\) and
  \(\B_{\Bisp[Y]}\) are bases, there are \(U\in\B_{\Bisp}\) and
  \(V\in\B_{\Bisp[Y]}\) with \((x,y)\in U\times V\subseteq\tilde{W}\).
  Then \(\pi(x,y)\in U V\subseteq W\).

  It remains to show that the family of subsets~\(\BXY\) is stable
  under intersections and set differences.
  Let \(U_j\in\B_{\Bisp}\) and \(V_j\in\B_{\Bisp[Y]}\) for \(j=1,2\).
  We must show that \(U_1 V_1 \cap U_2 V_2\) and \(U_1 V_1 \setminus
  U_2 V_2\) are in \(\BXY\).
  First, we claim
  \begin{equation}
    \label{eq:intersect_in_BXY}
    U_1 V_1 \cap U_2 V_2 = U_1 V_3,\qquad
    \text{with} \quad V_3 \defeq V_1 \cap \braket{U_1}{U_2} V_2.
  \end{equation}
  Assume \(u_1 v_1 = u_2 v_2\) in \(\Bisp\circ\Bisp[Y]\) for some
  \(u_j\in U_j\), \(v_j\in V_j\) for \(j=1,2\).
  This means that there is \(g\in \Gr\) with \(u_1 g = u_2\) and
  \(v_1 = g v_2\).
  This implies \(g = \braket{u_1}{u_2}\) and \(v_1 = g v_2 =
  \braket{u_1}{u_2} v_2 \in \braket{U_1}{U_2} V_2\).
  So \(u_1 v_1 \in U_1 V_3\) with \(V_3 = V_1 \cap
  \braket{U_1}{U_2} V_2\).
  Conversely, any element in \(U_1 V_3\) belongs to both \(U_1 V_1\)
  and \(U_1 \braket{U_1}{U_2} V_2 \subseteq U_2 V_2\).
  This proves the claim.
  So \(U_1 V_1 \cap U_2 V_2\in \BXY\).

  Next, we claim that
  \[
    U_1 V_1 \setminus U_2 V_2
    = U_1 V_1 \setminus (U_1 V_1 \cap U_2 V_2)
    = U_1 V_1 \setminus U_1 V_3
    \overset{!}= U_1 (V_1 \setminus V_3) \in \BXY.
  \]
  Only the equality marked with~\(!\) requires justification.
  Here we use that an element \(u_1 v_1 \in U_1 V_1\) is determined
  uniquely by \(v_1 \in V_1\) because \(\s(u_1) = \rg(v_1)\)
  determines~\(u_1\) in the slice~\(U_1\).
  Therefore, \(u_1 v_1 \in U_1 V_1\) belongs to \(U_1 V_3\) if and
  only if \(v_1 \in V_3\), as needed.
  We have now verified all claims in the lemma.
\end{proof}

The assumptions in Lemma~\ref{lem:base_of_comp} are satisfied
by Lemma~\ref{lem:operations_on_compact_slices} if
\(\B_{\Gr}\), \(\B_{\Bisp}\) consist of all compact slices 
and~\(\B_{\Bisp[Y]}\) consists of all compact, open Hausdorff subsets.
If~\(\Bisp[Y]\) is a groupoid correspondence, we may also
let~\(\B_{\Bisp[Y]}\) be the set of all compact slices.

\begin{lemma}
  \label{lem:base_of_composition_equality}
  Let \(U_1,U_2\in\B_{\Bisp}\), \(V_1,V_2\in\B_{\Bisp[Y]}\).
  Then \(U_1 V_1 = U_2 V_2\) holds in \(\B_{\Bisp\circ\Bisp[Y]}\) if
  and only if there are \(W_1,W_2\in\Bisp_{\Gr}\) with
  \begin{equation}
    \label{eq:base_of_composition_equality}
    \rg(W_j) \supseteq \rg(V_j) \cap \s(U_j) \text{ for
    }j=1,2 \text{ and }
    (U_1 W_1, W_1^{-1} V_1) = (U_2 W_2, W_2^{-1} V_2).
  \end{equation}
\end{lemma}

\begin{proof}
  First assume that there are \(W_1\) and~\(W_2\)
  satisfying~\eqref{eq:base_of_composition_equality}.
  Any element of \((u_j,v_j)\in U_j \times_{\s,\rg} V_j\) satisfies
  \(\s(u_j) = \rg(v_j) \in \rg(W_j)\), so that there is \(g\in W_j\)
  with \(\rg(g) = \s(u_j) = \rg(v_j)\).
  Then \(\pi(u_j,v_j) = \pi(u_j g,g^{-1} v_j) \in (U_j W_j) (W_j^{-1}
  V_j)\).
  It is clear that any element of \((U_j W_j) (W_j^{-1} V_j)\) also
  belongs to~\(U_j V_j\).
  Therefore, our assumptions imply that \(U_1 V_1 = (U_1 W_1)
  (W_1^{-1} V_1) = (U_2 W_2) (W_2^{-1} V_2) = U_2 V_2\).

  Conversely, assume that \(U_1 V_1 = U_2 V_2\).
  Then
  \begin{equation}
    \label{eq:UV_intersect}
    U_1 V_1
    = U_1 V_1 \cap U_2 V_2
    = U_1 (V_1 \cap \braket{U_1}{U_2} V_2)
  \end{equation}
  by~\eqref{eq:intersect_in_BXY}.
  Since~\(U_1\) is a slice, the only way for two elements of the form
  \(\pi(u,v)\) and \(\pi(u',v')\) with \(u,u'\in U_1\), \(v,v'\in\Bisp[H]\) to
  have the same class in \(\Bisp\circ \Bisp[Y]\) is if \(u= u'\) and
  \(v=v'\).
  Therefore, \eqref{eq:UV_intersect} implies that any element \(v_1\in
  V_1\) with \(\rg(v_1)\in\s(U_1)\) is also contained in
  \(\braket{U_1}{U_2}V_2\).
  Since \(\s(\braket{U_1}{U_2}) \subseteq \s(U_2)\), we may rewrite
  this further as
  \(\s(U_1) V_1 \subseteq \braket{U_1}{U_2} \s(U_2) V_2\).
  The same argument with \(1,2\) exchanged gives
  \(\s(U_2) V_2 \subseteq \braket{U_2}{U_1} \s(U_1)V_1\).
  These two equations together imply \(\s(U_2) V_2 = \braket{U_2}{U_1}
  \s(U_1)V_1\) and \(\s(U_1) V_1 = \braket{U_1}{U_2} \s(U_2)V_2\).
  Then it follows that \(\rg(V_1) \cap \s(U_1) \subseteq
  \rg(\braket{U_1}{U_2})\) because \(\rg(V_1) \cap \s(U_1) = \rg(\s(U_1)
  V_1)\).
  Exchanging \(1\) and~\(2\), the same argument gives \(\rg(V_2) \cap
  \s(U_2) \subseteq \rg(\braket{U_2}{U_1})\).
  Now let
  \[
    W_1\defeq \braket{U_1}{U_2} \in \B_{\Gr},\qquad
    W_2\defeq \rg(\braket{U_2}{U_1}) \in \B_{\Gr^0} \subseteq
    \B_{\Gr}.
  \]
  We have seen above that \(\rg(W_j) \supseteq \rg(V_j) \cap \s(U_j)\)
  for \(j=1,2\).
  We claim that \((U_1 W_1, W_1^{-1} V_1) = (U_2 W_2, W_2^{-1} V_2)\)
  holds as well.
  To prove this, we use that \(U_1 \braket{U_1}{U_2} \subseteq U_2\)
  and \(\s(U_1\braket{U_1}{U_2}) = \rg(\braket{U_2}{U_1})\) to
  conclude that \(U_1 W_1 = U_2 W_2\).
  The computations above show that \(W_1^{-1} V_1 = \braket{U_2}{U_1}
  V_1= \braket{U_2}{U_1} \s(U_1)V_1 = \s(U_2) V_2 = \braket{U_2}{U_1}
  \s(U_1)V_1 = \braket{U_2}{U_1}\braket{U_1}{U_2} \s(U_2) V_2 = W_2^{-1}
  V_2\).
\end{proof}

\subsection{The Steinberg module of a topological space}

Let~\(R\) be a unital ring.  We give it the discrete topology, so that a
function to~\(R\) is continuous if and only if it is locally constant.
Let~\(X\) be a topological space.  The set of all maps
\(R^X\defeq \{\xi\colon X\to R\}\) is an \(R\)-module by pointwise
addition and scalar multiplication.  For a subset \(F\subseteq R^X\),
let \(\langle F\rangle_R \subseteq R^X\) be the \emph{smallest
  \(R\)\nb-submodule of~\(R^X\) containing~\(F\)}.  For a subset
\(A\subseteq X\), its \emph{characteristic function}~\(\charmap{A}\) is
defined by
\begin{align*}
    \charmap{A}\colon X\to R,\qquad
    x\mapsto
        \begin{cases}
            1 &\text{if } x\in A,\\
            0 &\text{if } x\not\in A.
        \end{cases}
\end{align*}
This map is locally constant if and only if \(A\subseteq X\) is closed
and open.
For a subset \(U\subseteq X\) and a map \(\xi\colon U\to R\), its
\emph{extension by zero} \(\tilde{\xi}\in R^X\) is defined by
\(\tilde{\xi}\vert_U \defeq \xi\) and \(\tilde{\xi}\vert_{X\setminus
  U}\defeq 0\).
We do not require~\(X\) to be Hausdorff.
Therefore, \(\tilde{\xi}\) may fail to be locally constant even
if~\(\xi\) is.
Let \(\Contc(U,R) \subseteq R^X\) denote the space of
all~\(\tilde{\xi}\) where \(\xi\colon U\to R\) is locally constant and
the support is compact.
Here
\[
  \supp(\xi)\defeq \xi^{-1}\bigl(R\setminus \{0\}\bigr)\subseteq U.
\]

\begin{propdef}
  \label{lem:steinalg}
  The \emph{Steinberg module} of a topological space~\(X\) is the
  \(R\)\nb-submodule \(A_R(X)\) of~\(R^X\) described in the following
  equivalent ways:
  \begin{enumerate}[label=\textup{(\ref*{lem:steinalg}.\arabic*)},
    leftmargin=*,labelindent=0em]
  \item\label{enum:steinalg1}%
    \(\bigl\langle \tilde{\xi} \bigm\mid \xi\in \Contc(U,R)\text{ for }
    U\subseteq X\text{ a Hausdorff open subset}\bigr\rangle_R\);
  \item\label{enum:steinalg2}%
    \(\bigl\langle \tilde\xi\in R^X \bigm\mid \supp(\xi)\text{ compact
      Hausdorff open and } \xi\vert_{\supp(\xi)}\text{ locally
      constant} \bigr\rangle\);
  \item\label{enum:steinalg3}%
    \(\bigl\langle \charmap{U} \bigm\mid U\subseteq X\text{ a compact Hausdorff
      open subset}\bigr\rangle_R\).
  \end{enumerate}
\end{propdef}

\begin{proof}
  We prove that
  \(\ref{enum:steinalg1}\subseteq \ref{enum:steinalg2}\subseteq
  \ref{enum:steinalg3}\subseteq\ref{enum:steinalg1}\).  The third
  inclusion is immediate, so it remains to prove the first two.

  For the first inclusion, take a Hausdorff open subset \(U\subseteq X\)
  and a map \(\xi\in \Contc(U,R)\).  Since subsets of Hausdorff spaces
  are Hausdorff, \(\supp(\tilde{\xi})=\supp(\xi)\) is Hausdorff.
  Since~\(\xi\) is locally constant, \(\supp(\xi)\) is open in~\(U\) and
  as~\(U\) is open in~\(X\), it is also open in~\(X\).  Thus
  \(\supp(\tilde{\xi})\) is a compact, Hausdorff open subset of~\(X\).
  The restriction of~\(\tilde{\xi}\) to its support remains continuous
  as a function on \(\supp(\tilde{\xi})\subseteq U\).

  For the second inclusion, we take a map \(\xi\colon X\to R\) such that
  \(U\defeq \supp(\xi)\subseteq X\) is compact Hausdorff open and
  \(\xi\vert_U\) is locally constant.  Then \(\xi (U)\) is a finite subset
  \(\xi(U)=\{r_1,\dotsc,r_n\}\subseteq R\).  The subsets
  \(U_i\defeq \xi^{-1}(r_i)\subseteq \supp(\xi)\) for \(i=1,\dotsc, n\)
  are closed and open in~\(U\) because~\(\xi\) is locally constant.
  Since~\(U\) is compact and Hausdorff, they are compact and Hausdorff
  as well.  The subsets~\(U_i\) are disjoint and
  \(X=\xi^{-1}(0)\sqcup U_1\sqcup\dotsb\sqcup U_n\).  Then
  \(\xi=\sum_{i=1}^n r_i\charmap{U_i}\) follows.
\end{proof}

If~\(X\) is a Hausdorff space, then \(A_R(X)=\Contc(X,R)\) is
exactly the space of all locally constant maps \(X\to R\) with compact
support.  In general, however, functions in \(A_R(X)\) need not be
locally constant because~\(\charmap{U}\) for a compact Hausdorff open
subset \(U\subseteq X\) fails to be locally constant when~\(U\) is not
closed.  We will only consider \(A_R(X)\) when~\(X\) has a base
of compact, Hausdorff, open subsets.  If~\(X\) has no such subsets,
then \(A_R(X)=0\) and so our construction is useless.

\begin{lemma}
  \label{lem:dis_union_is_direkt_sum}
  Let \(X=\bigsqcup_{i\in I} X_i\) be a disjoint union of spaces.  Then
  \[
    A_R(X)\cong \bigoplus_{i\in I}  A_R(X_i).
  \]
\end{lemma}

\begin{proof}
  We use Definition~\ref{enum:steinalg3} of the Steinberg module.  Let
  \[
    \iota_i\colon A_R(X_i) \to A_R(X),\qquad
    \xi \mapsto \tilde{\xi},
  \]
  be the extension by zero map.
  It is a well-defined \(R\)-module homomorphism because
  \(X_i\subseteq X\) is open.
  These maps induce a map \(\iota\colon \bigoplus_{i\in I} A_R(X_i)
  \to A_R(X)\).
  Since the sets~\(X_i\) are disjoint, the map~\(\iota\) is clearly
  injective.
  It remains to prove that it is also surjective.
  Let \(U\subseteq X\) be compact, Hausdorff, and open.
  Then \(U_i \defeq U\cap X_i\) is compact, Hausdorff, and open for
  \(i \in I\), and only finitely many of them are nonempty.
  So \(\charmap{U} = \sum_{i\in F} \iota_i(\charmap{U_i})\) for some
  finite subset \(F\subseteq I\).
\end{proof}

Next, we describe the Steinberg module through generators and
relations.

\begin{proposition}
  \label{pro:ample_base_unions}
  Let~\(\B\) be an ample base for~\(X\).
  A subset \(U\subseteq X\) is Hausdorff, compact and open if and only
  if it is a finite disjoint union of subsets in~\(\B\).
\end{proposition}

\begin{proof}
  A finite disjoint union of compact, Hausdorff spaces is again
  compact and Hausdorff, and a finite union of open subsets is again
  open.
  Thus any finite disjoint union of subsets in~\(\B\) is Hausdorff,
  compact and open in~\(X\).
  Conversely, let \(U\subseteq X\) be compact, Hausdorff, and open.
  Since~\(\B\) is a base, for any \(x\in U\) there is \(U_x\in\B\)
  with \(x\in U_x\) and \(U_x \subseteq U\).
  These sets cover~\(U\).
  Since~\(U\) is compact, there are finitely many \(U_i\in \B\) for
  \(i=1,\dotsc, n\) with \(U=\bigcup_{i=1}^n U_i\).
  As~\(U\) is Hausdorff and each~\(U_i\) is compact, \(U_i\) is
  relatively closed in~\(U\).
  For \(i=1,\dotsc,n\), define \(W_i\defeq U_i\setminus
  (\bigcup_{j=1}^{i-1} U_j)\).
  These sets are disjoint, closed and open in~\(U\), and satisfy
  \(U=\bigsqcup_{i=1}^n W_i\).
  They belong to~\(\B\) because we may get them by repeating the set
  difference operation.
\end{proof}

\begin{theorem}[\cite{Li:Semigroup_amenability}*{Lemma 2.2}]
  \label{the:steinmod_is_quot_of_sum}
  Let~\(X\) be a topological space with an ample base~\(\B\).
  Then the \(R\)\nb-module homomorphism
  \[
    \pi\colon \bigoplus_{B\in\B} R\to A_R(X),\qquad
    (r_B) \mapsto \sum_{B\in\B} r_B\cdot\charmap{B},
  \]
  is surjective and its kernel is the \(R\)\nb-submodule generated by
  \(\delta_{U\sqcup V} - \delta_U - \delta_V\) for disjoint
  \(U,V\in\B\) with \(U \sqcup V \in\B\).
\end{theorem}

\begin{proof}
  The map~\(\pi\) is well-defined because \(\charmap{B}\in A_R(X)\)
  for \(B\in\B\).
  Proposition~\ref{pro:ample_base_unions} implies that the
  \(R\)\nb-module generated by~\(\charmap{B}\) for \(B\in\B\)
  contains~\(\charmap{U}\) for any Hausdorff, compact, open subset
  \(U\subseteq X\).
  The latter functions generate~\(A_R(X)\) by~\ref{enum:steinalg3}.
  Thus~\(\pi\) is surjective.

  Let~\(T\) denote the \(R\)\nb-submodule in \(\bigoplus_{B\in\B} R\)
  generated by \(\delta_{U\sqcup V} - \delta_U - \delta_V\) for
  disjoint \(U,V\in\B\) with \(U \sqcup V \in\B\).
  If \(U,V\in\B\) are disjoint, then
  \(\charmap{U\sqcup V}=\charmap{U}+\charmap{V}\).
  This implies \(\pi|_T=0\).
  Let \(x = \sum_{j=1}^N a_j \delta_{U_j}\) be an element of \(\ker
  \pi\).
  We want to subtract an element of~\(T\) so that the result is equal
  to such a sum where all~\(U_j\) are disjoint.
  Recall that \(U_1,U_2\in\B\) implies \(U_1 \cap U_2\in\B\) and \(U_1
  \setminus U_2 \in \B\) and our relations give
  \[
    \delta_{U_1}
    \equiv \delta_{U_1\cap U_2} + \delta_{U_1 \setminus U_2} \bmod T.
  \]
  Iterating this, an induction argument on \(\abs{I \cup J}\) shows
  that~\(\B\) contains
  \[
    U_{I,J} \defeq \bigcap_{i\in I} U_i \setminus \bigcup_{j\in J} U_j
  \]
  for any disjoint finite subsets \(I,J\subseteq \{1,\dotsc,N\}\) and
  that the sum of~\(\delta_{U_{I,J}}\) for all disjoint \(I,J\) with
  \(i \in I\) and fixed \(I \cup J\) is equivalent to~\(\delta_i\) modulo~\(T\).
  By construction, if \(I_k,J_k\) are disjoint pairs with \(I_k \cup
  J_k = \{1,\dotsc,N\}\) for \(k=1,2\), then \(U_{I_1,J_1} \cap
  U_{I_2,J_2} = \emptyset\) unless \(I_1=I_2\) and \(J_1 = J_2\).
  The argument above shows that \(x \equiv \sum_{I,J} a_{I,J}
  \delta_{U_{I,J}} \bmod T\) for some \(a_{I,J}\in R\), where the sum
  runs over all pairs \(I,J\) of disjoint sets with \(I \cup J =
  \{1,\dotsc,N\}\).
  Now \(\pi(x)=0\) implies that \(a_{I,J}=0\) or \(U_{I,J}=\emptyset\)
  for all \(I,J\).
  Since \(\delta_\emptyset = -(\delta_\emptyset + \delta_\emptyset -
  \delta_{\emptyset \cup \emptyset})\in T\), this implies \(x\in T\) as
  desired.
\end{proof}

Results like the following functoriality result are already known in
special cases, but we have not found a reference in the required
generality.

\begin{proposition}
  \label{pro:Steinberg_functorial}
  Let \(X\) and~\(Y\) be spaces with ample bases \(\B_X\) and~\(\B_Y\).
  Let \(f\colon X\to Y\) be a local homeomorphism and let \(g\colon
  Y\to X\) be a proper, continuous map; that is, \(g\times \id_Z\) is
  closed for all spaces~\(Z\).
  Then there are well-defined \(R\)\nb-module maps
  \begin{alignat*}{2}
    f_*\colon A_R(X) &\to A_R(Y),&\qquad
    f_*(h)(y) &= \sum_{x\in X, f(x)=y} h(x),\\
    g^*\colon A_R(X) &\to A_R(Y),&\qquad
    g^*(h)(y) &= h\bigl(g(x)\bigr).
  \end{alignat*}
  Let \(U\in \B_X\).
  Then \(g^*(\charmap{U}) = \charmap{g^{-1}(U)}\).
  If \(f|_U\) is injective, then \(f_*(\charmap{U}) =
  \charmap{f(U)}\).
\end{proposition}

\begin{proof}
  Let~\(\B_X'\) be the set of compact open subsets \(U\in\B_X\) such
  that~\(f|_U\) is injective.
  This subset of~\(\B_X\) is an ample base for~\(X\) as well
  because~\(f\) is a local homeomorphism.
  If \(U\in \B_X'\), then \(f(U)\subseteq Y\) is compact, Hausdorff
  and open because~\(f\) is a local homeomorphism, and
  \(f_*(\charmap{U}) = \charmap{f(U)}\).
  Thus \(f_*(\charmap{U})\in A_R(Y)\).
  This implies that~\(f_*\) maps \(A_R(X)\) to \(A_R(Y)\) because
  functions of the form~\(\charmap{U}\) for \(U\in \B_X'\) generate
  \(A_R(X)\) as an \(R\)\nb-module by
  Theorem~\ref{the:steinmod_is_quot_of_sum}.
  If \(U\in\B_X\), then \(g^{-1}(U)\subseteq Y\) is open because~\(g\)
  is continuous, and Hausdorff and compact because~\(g\) is proper.
  It is clear that \(g^*(\charmap{U}) = \charmap{g^{-1}(U)}\).
  Then~\(g^*\) maps \(A_R(X)\) to \(A_R(Y)\) by
  Theorem~\ref{the:steinmod_is_quot_of_sum}.
\end{proof}

\begin{definition}[{compare \cite{Steinberg:Groupoid_approach}*{Definition~4.4}}]
  \label{def:conv}
  Let~\(\Gr\) be an ample groupoid.
  Define a multiplication on \(A_R(\Gr)\) by the \emph{convolution}
  \((\xi * \eta)(g) = \sum_{h \in \Gr^{\rg(g)}} \xi(h)\eta(h^{-1}g)\)
  for \(\xi, \eta \in A_R(\Gr)\) and \(g\in\Gr\).
  Here \(\Gr^x = \setgiven{g\in \Gr}{\rg(g)=x}\).
\end{definition}

If \(U,V\subseteq \Gr\) are compact slices, then so is~\(U V\) and
\(\charmap{U}*\charmap{V}=\charmap{UV}\) (see
\cite{Steinberg:Groupoid_approach}*{Proposition~4.5}).
Since these characteristic functions generate
\(A_R(\Gr)\) as an \(R\)\nb-module by
Theorem~\ref{the:steinmod_is_quot_of_sum},
it follows that \(\xi*\eta\in A_R(\Gr)\) for all \(\xi, \eta \in
A_R(\Gr)\).
The convolution of functions is associative because the multiplication
of slices is associative.
The Steinberg algebra \(A_R(\Gr)\) is unital if and only if~\(\Gr^0\)
is compact (see
\cite{Steinberg:Groupoid_approach}*{Proposition~4.11}).
In general, it is a ring with local units, that is, an object in
\(\Rings\):

\begin{proposition}
  \label{prop:steinalg_loc_units}
  For an ample groupoid~\(\Gr\), the following is a local unit of
  \(A_R(\Gr)\):
  \[
    E\defeq \setgiven{\charmap{U}}{U\subseteq\Gr^0\text{ compact, open}}.
  \]
\end{proposition}

\begin{proof}
  We omit the proof of this well known result because we will prove a
  more general statement in Lemma~\ref{lem:steinalg_is_smooth}.
\end{proof}

\subsection{Steinberg bimodules of ample correspondences}

For some time, we consider any space~\(\Bisp\) with commuting actions
of~\(\Gr[H]\) on the left and of~\(\Gr\) on the right.
We will assume~\(\Bisp\) to be a groupoid correspondence later, when it
becomes necessary.
We turn \(A_R(\Bisp)\) into a smooth
\(A_R(\Gr[H]),A_R(\Gr)\)-bimodule.

\begin{definition}
  \label{def:conv_module}
  For a space with commuting actions of two groupoids \(\Gr[H]\) on
  the left and~\(\Gr\) on the right, define an \(A_R(\Gr[H]),
  A_R(\Gr)\)\nb-bimodule structure on \(A_R(\Bisp)\) by
  \begin{align*}
    (\alpha * \xi) (x)
    &\defeq  \sum_{g\in \Gr_{\s(x)}} \alpha(x g^{-1}) \cdot\xi(g),\\
    (\zeta * \alpha)(x)
    &\defeq  \sum_{h \in \Gr[H]^{\rg(x)}} \zeta(h)\cdot\alpha(h^{-1} x)
  \end{align*}
  for \(\zeta \in A_R(\Gr[H])\), \(\alpha\in A_R(\Bisp)\),
  \(\xi \in A_R(\Gr)\) and \(x\in\Bisp\).
\end{definition}

We check that this bimodule structure is well-defined, that is,
\(\alpha*\xi,\zeta*\alpha \in A_R(\Bisp)\).
If \(U\subseteq\Gr[H]\) and \(W\subseteq\Gr\) are
compact slices and \(V\subseteq\Bisp\) is compact, open and Hausdorff,
then \(U V\) and~\(V W\) are compact open and Hausdorff
by Lemma~\ref{lem:base_of_comp}, and the same computation as for
\cite{Antunes-Ko-Meyer:Groupoid_correspondences}*{Lemma~7.4}) shows
that  \(\charmap{U}*\charmap{V}=\charmap{UV}\) and
\(\charmap{V}*\charmap{W}=\charmap{VW}\).
Since the characteristic functions generate \(A_R(\Gr[H])\),
\(A_R(\Bisp)\) and \(A_R(\Gr)\) as \(R\)\nb-modules, this implies
\(\alpha*\xi,\zeta*\alpha \in A_R(\Bisp)\).
Since \((U V) W = U (V W)\), the convolutions above make
\(A_R(\Bisp)\) an \(A_R(\Gr[H]),A_R(\Gr)\)-bimodule.

\begin{lemma}
  \label{lem:steinalg_is_smooth}
  The bimodule structure on \(A_R(\Bisp)\) defined above is
  smooth.
\end{lemma}

\begin{proof}
  Let \(\xi_1,\dotsc,\xi_n\in A_R(\Bisp)\).
  Then \(\xi_j\) are \(R\)\nb-linear combinations of characteristic
  functions of Hausdorff, compact, open subsets of~\(\Bisp\).
  Their images under the anchor maps in \(\Gr[H]^0\) and~\(\Gr^0\) are
  compact.
  Since the latter spaces are ample and Hausdorff, the unions of these
  compact subsets are again compact and admit compact open
  neighbourhoods \(U\subseteq \Gr[H]^0\), \(W\subseteq \Gr^0\).
  Then \(\charmap{U}* \xi_j = \xi_j = \xi_j * \charmap{W}\) for \(j=1,\dotsc,n\).
\end{proof}

The following results require~\(\Bisp\) to be a proper groupoid
correspondence.

\begin{proposition}
  \label{pro:Steinberg_proper_with_bases}
  Let \(\Bisp\colon \Gr[H]\leftarrow\Gr\) be a proper ample groupoid
  correspondence.
  Let \(\B_{\Gr}\), \(\B_{\Bisp}\), and \(\B_{\Bisp[Y]}\) be ample
  bases in \(\Gr\), \(\Bisp\), and \(\Bisp[Y]\), respectively,
  satisfying the assumptions in Lemma~\textup{\ref{lem:base_of_comp}}.
  Let \(K\subseteq \Gr[H]^0\).
  Then there are \(U_1,\dotsc,U_n\in\B_{\Bisp[X]}\) such that the
  following map is a right \(A_R(\Gr)\)-module isomorphism:
  \[
    \bigoplus_{j=1}^n \charmap{\s(U_j)} * A_R(\Gr)
    \to \charmap{K} * A_R(\Bisp),\qquad
    \sum f_j \mapsto \charmap{U_j} * f_j.
  \]
  Here each summand \(\charmap{\s(U_j)} * A_R(\Gr)\subseteq A_R(\Gr)\)
  is fgp, so that \(\charmap{K} * A_R(\Bisp)\) is fgp as well, and
  \(\s(U_j) \in\B_{\Gr}\) for \(j=1,\dotsc,n\).
\end{proposition}

\begin{proof}
  Since \(\charmap{K}*\charmap{V}=\charmap{K V} = \charmap{V\cap
  \rg^{-1}(K)}\) for any compact open slice~\(V\), it follows that
  \(\charmap{K}* A_R(\Bisp) = A_R(\rg_{\Bisp}^{-1}(K))\).
  The map \(\Bisp/\Gr\to\Gr[H]^0\) induced by~\(\rg_{\Bisp}\) is
  proper by assumption.
  So the right \(\Gr\)\nb-space \(\rg_{\Bisp}^{-1}(K)\) is cocompact,
  that is, \(\Bar{\rg}^{-1}_{\Bisp}(K)=\Qu(\rg_{\Bisp}^{-1}(K))\) is a
  compact subset of~\(\Bisp/\Gr\), where \(\Qu\colon \Bisp\to\Bisp/\Gr\)
  is the quotient map.
  Since~\(K\) is also open, so is \(\rg_{\Bisp}^{-1}(K) \subseteq
  \Bisp\).
  Therefore, if \(x\in \Qu(\rg_{\Bisp}^{-1}(K))\), then there is
  \(U_x\in\B_{\Bisp}\) with \(x \in \Qu(U_x)\) and \(U_x \subseteq
  \rg_{\Bisp}^{-1}(K)\).
  The sets \(\Qu(U_x)\) form an open cover of
  \(\Qu(\rg_{\Bisp}^{-1}(K))\).
  Since the latter is compact, there are finitely many elements so
  that \(\bigcup_{i=1}^n U_{x_i}\cdot\Gr = \rg_{\Bisp}^{-1}(K)\).
  The image \(\Qu(U_{x_i})\) is still compact open in~\(\Bisp/\Gr\).
  Now we replace these subsets by ones with disjoint images
  in~\(\Bisp/\Gr\), letting
  \[
    U_j \defeq U_{x_j} \setminus \bigcup_{i=1}^{j-1}
    \Qu^{-1}(\Qu(U_{x_i}))
    = U_{x_j} \setminus \bigcup_{i=1}^{j-1} U_{x_i} \cdot \Gr
  \]
  for \(j=1,\dotsc,n\).
  Then \(\Qu(U_j) = \Qu(U_{x_j}) \setminus \bigcup_{i=1}^{j-1}
  \Qu(U_{x_i})\).
  So the subsets \(\Qu(U_j) \subseteq \Bisp/\Gr\) are disjoint and their
  union is still \(\Qu(\rg_{\Bisp}^{-1}(K))\).
  Thus
  \[
    \rg_{\Bisp}^{-1}(K) = \bigsqcup U_j \cdot \Gr,
    \qquad
    A_R(\rg_{\Bisp}^{-1}(K)) \cong \bigoplus A_R(U_j \cdot \Gr).
  \]
  We claim that \(U_j\in \B_{\Bisp}\) for \(j=1,\dotsc,n\).
  We may construct~\(U_j\) by repeated set differences.
  So it suffices to prove that \(U \setminus (V\cdot \Gr)
  \in\B_{\Bisp}\) if \(U,V\in\B_{\Bisp}\).
  Since \(U\) and~\(V\) are compact and the right \(\Gr\)\nb-action is
  proper, the set of \(g\in \Gr\) for which there are \(u\in U\),
  \(v\in V\) with \(u = v g\) is compact.
  By Proposition~\ref{pro:ample_base_unions}, we may write this set as
  a finite union \(\bigcup_{j=1}^\ell W_j\) of sets in~\(\B_{\Gr}\).
  So \(U \setminus (V\cdot \Gr) = U\setminus \bigcup_{j=1}^\ell V\cdot
  W_j\).
  The assumptions in Lemma~\ref{lem:base_of_comp} imply \(V W_j\in
  \B_{\Bisp}\) for all~\(j\), and then repeated set difference shows
  that \(U \setminus (V\cdot \Gr) \in \B_{\Bisp}\) as needed.

  Since~\(U_i\) is a slice, \(\s|_{U_i}\) is a homeomorphism \(U_i
  \congto \s(U_i)\), and \(\s(U_i)\) is a compact open subset
  of~\(\Gr^0\).
  The map \(x\cdot g\mapsto \s(x)\cdot g\) for \(x\in U_i\),
  \(g\in\Gr\) with \(\s(x) = \rg(g)\) is a homeomorphism from
  \(U_i\cdot \Gr \subseteq \Bisp\) onto~\(\Gr_{\s(U_i)}\) because the
  right \(\Gr\)\nb-action on~\(\Bisp\) is free and proper.
  Thus \(A_R(U_i\cdot \Gr) \cong A_R(\s(U_i)\cdot \Gr) =
  A_R(\Gr_{\s(U_i)}) = \charmap{\s(U_i)} * A_R(\Gr)\).
  Now the asserted direct sum decomposition follows.
  The assumptions in Lemma~\ref{lem:base_of_comp} also imply \(\s(U_j)
  = \braket{U_j}{U_j} \in \B_{\Gr}\) as asserted.
\end{proof}

\begin{proposition}
  \label{prop:Steinpro}
  Let \(\Bisp\colon \Gr[H]\leftarrow\Gr\) be a proper ample groupoid
  correspondence.
  Then \(A_R(\Bisp)\) is a proper \(A_{R}(\Gr[H]), A_R(\Gr)\)-bimodule
  as in Definition~\textup{\ref{def:properbimod}}.
\end{proposition}

\begin{proof}
  Proposition~\ref{pro:Steinberg_proper_with_bases} implies that
  \(\charmap{K} * A_R(\Bisp)\) for a compact open subset \(K\subseteq
  \Gr[H]^0\) is fgp as a right \(A_R(\Gr)\)-module.
  This implies the claim because~\(\charmap{K}\) for such~\(K\) is a
  local unit in \(A_R(\Gr)\) by
  Proposition~\ref{prop:steinalg_loc_units}.
\end{proof}

\subsection{Steinberg bimodule as a pseudofunctor}

Now we show that there is a normal pseudofunctor from the
bicategory of ample groupoid correspondences to \(\Rings\) that maps
an ample groupoid~\(\Gr\) to \(A_R(\Gr)\) and an ample groupoid
correspondence~\(\Bisp\) to the bimodule~\(A_R(\Bisp)\).
This is a purely algebraic analogue of the construction in
\cite{Antunes-Ko-Meyer:Groupoid_correspondences}*{Section~7}.
First, we define the pseudofunctor on \(2\)-arrows.
Then we define the bimodule isomorphisms making our pseudofunctor
multiplicative on arrows.

We refer to \cite{Johnson-Yau:2-Dim}*{Definition~4.1.2} for the
definition of a pseudofunctor.
What we call normal is called ``strictly unitary''
in~\cite{Johnson-Yau:2-Dim}.
A pseudofunctor is also called a \emph{homomorphism} of bicategories.

\begin{lemma}
  \label{lem:A(f)_bimod_hom}
  Let \(\Gr\) and~\(\Gr[H]\) be ample groupoids and let
  \(\Bisp,\Bisp[Y]\colon \Gr[H] \leftarrow \Gr\) be ample
  correspondences.
  Let \(f\colon\Bisp\Rightarrow\Bisp[Y]\) be a continuous
  \(\Gr[H],\Gr\)-equivariant map or, equivalently, a \(2\)\nb-arrow.
  Then the \(R\)-bimodule map
  \[
    A_R(f)= f_*\colon A_R(\Bisp) \to A_R(\Bisp[Y]),\qquad
    \alpha \mapsto \left[ y\mapsto \sum_{x\in f^{-1}(y)} \alpha(x) \right],
  \]
  is an \(A_R(\Gr[H]) , A_R(\Gr)\)-bimodule homomorphism.
  It maps~\(\charmap{U}\) to \(\charmap{f(U)}\) for all compact slices
  \(U\subseteq\Bisp\).
  The map \(f\mapsto A_R(f)\) is functorial.
\end{lemma}

\begin{proof}
  Recall that the compact slices form an ample base for
  the topology on~\(\Bisp\).
  Proposition~\ref{pro:Steinberg_functorial} implies that \(A_R(f)\)
  maps \(A_R(\Bisp)\) to \(A_R(\Bisp[Y])\) and maps \(\charmap{U}\)
  to~\(\charmap{f(U)}\) as claimed.  If \(W\subseteq \Gr[H]\),
  \(V\subseteq \Gr\) are compact slices, then \(A_R(f)\) maps
  \(\charmap{U}* \charmap{V} = \charmap{U V}\) to
  \(\charmap{f(U)V}= \charmap{f(U)}*\charmap{V}\) and
  \(\charmap{W}* \charmap{U} = \charmap{W U}\) to
  \(\charmap{W f(U)}= \charmap{W}*\charmap{f(U)}\).  Therefore,
  \(A_R(f)\) is an \(A_R(\Gr[H]) , A_R(\Gr)\)-bimodule homomorphism.
  The functoriality of \(f\mapsto A_R(f)\) is trivial.
\end{proof}

Secondly, we want to define natural bimodule isomorphisms
\[
  \mu_{\Bisp,\Bisp[Y]}\colon A_R(\Bisp)\otimes_{ A_R(\Gr)}
  A_R(\Bisp[Y]) \to A_R(\Bisp\circ\Bisp[Y])
\]
for two composable groupoid correspondences.
For later purposes, we construct this map in slightly greater
generality.
The following theorem was also proved independently in
\cite{Miller:Ample_groupoid_homology}*{Proposition~2.9}.

\begin{definition}
  \label{def:mu}
  Let \(\Gr\), \(\Gr[H]\) and~\(\Gr[K]\) be ample groupoids, let
  \(\Bisp\colon \Gr[H] \leftarrow \Gr\) be an ample correspondence.
  Let \(\Bisp[Y]\colon\Gr \leftarrow \Gr[K]\) be a
  \(\Gr,\Gr[K]\)-space with an ample base.
  Define
  \begin{align*}
    \tilde{\mu}_{\Bisp,\Bisp[Y]}\colon A_R(\Bisp)\times
    A_R(\Bisp[Y])
    &\to A_R(\Bisp\circ\Bisp[Y]),\\
    (\alpha,\beta)
    &\mapsto \left[ [x,y] \mapsto
      \sum_{g\in \Gr_{\s(x)}} \alpha(xg^{-1})\cdot\beta(gy)
      \right].
  \end{align*}
\end{definition}

We need~\(\Bisp\) to be a groupoid correspondence in order for
\(\Bisp\circ \Bisp[Y]\) to be a well behaved space.

\begin{theorem}
  \label{the:mu_invertible}
  If \(U\subseteq\Bisp\) and \(V\subseteq \Bisp[Y]\) are compact slices,
  then \(\tilde{\mu}_{\Bisp,\Bisp[Y]}(\charmap{U},\charmap{V}) =
  \charmap{U V}\).
  The map~\(\tilde{\mu}_{\Bisp,\Bisp[Y]}\) induces an \(A_R(\Gr[H]) ,
  A_R(\Gr[K])\)-bimodule isomorphism
  \[
    \mu_{\Bisp,\Bisp[Y]}\colon
    A_R(\Bisp)\otimes_{ A_R(\Gr)} A_R(\Bisp[Y])
    \congto A_R(\Bisp\circ\Bisp[Y]),
  \]
  which is natural for \(2\)\nb-arrows \(\Bisp\to \Bisp'\) and
  \(\Bisp[Y]\to \Bisp[Y]'\).
\end{theorem}

\begin{proof}
  The sum \([x,y] \mapsto \sum_{g\in \Gr_{\s(x)}}
  \alpha(xg^{-1})\cdot\beta(gy)\) does not depend on the
  representation of \([x,y]\), since a different representative
  \((x\tilde{g}^{-1},\tilde{g} y)\in [x,y]\) only changes the order of
  the summands.
  It is easy to check that \(\tilde{\mu}_{\Bisp,\Bisp[Y]}(\charmap{U},
  \charmap{V}) = \charmap{U V}\).

  In the following proof, we pick ample bases for \(\Bisp\), \(\Gr\)
  and~\(\Bisp[Y]\) that satisfy the assumptions in
  Lemma~\ref{lem:base_of_comp}.
  For instance, taking the ample bases of all compact slices will do.
  The balanced tensor product commutes with direct sums and cokernels,
  and \(R\otimes_R R \cong R\).
  Since \(A_R(\Gr)\) is spanned by~\(\charmap{W}\) for
  \(W\in\B_{\Gr}\), the balancing of the tensor product
  \(A_R(\Bisp)\otimes_{ A_R(\Gr)} A_R(\Bisp[Y])\)
  over~\(A_R(\Gr)\) has the same effect as dividing out
  \(x*\charmap{W} \otimes y - x \otimes \charmap{W}*y\) for all
  \(W\in\B_{\Gr}\) and all generators \(x\) and~\(y\) for
  \(A_R(\Bisp)\) and \(A_R(\Bisp[Y])\).
  Therefore, Theorem~\ref{the:steinmod_is_quot_of_sum} implies that
  \(A_R(\Bisp)\otimes_{ A_R(\Gr)} A_R(\Bisp[Y])\) is the quotient of the
  free \(R\)\nb-module on the set of pairs \((U,V)\) with
  \(U\in\B_{\Bisp}\), \(V\in\B_{\Bisp[Y]}\) by the subspace~\(S\)
  generated by
  \[
    \delta_{(U_1 \sqcup U_2,V)} -  \delta_{(U_1,V)} -
    \delta_{(U_2,V)},\quad
    \delta_{(U,V_1 \sqcup V_2)} -  \delta_{(U,V_1)} -
    \delta_{(U,V_2)},\quad
    \delta_{U W,V} - \delta_{U, W V}
  \]
  for all \(U,U_1,U_2\in \B_{\Bisp}\), \(V,V_1,V_2\in \B_{\Bisp}\),
  and \(W\in\B_{\Gr}\)  with \(U_1 \sqcup U_2 = U\) and \(V_1 \sqcup
  V_2 = V\).

  We first implement the relation \(\delta_{U W,V} \equiv \delta_{U,W
    V}\).
  We claim that the quotient by this relation is the free module
  generated by the elements of the base \(\B_{\Bisp\circ \Bisp[Y]}\)
  in Lemma~\ref{lem:base_of_comp}.
  First, it is clear that \((U W) V = U (W V)\) holds in
  \(\B_{\Bisp\circ \Bisp[Y]}\).
  Secondly, let \(U_j,V_j,W_j\) be as in
  Lemma~\ref{lem:base_of_composition_equality}.
  Let \(j\in \{1,2\}\).
  Since \(\rg(W_j) \supseteq \rg(V_j) \cap \s(U_j)\),
  the compact open subset \(\s(U_j) \setminus \rg(W_j)\) is disjoint
  from \(\rg(V_j)\).
  Since the latter is compact, there is a compact open set~\(X_j\)
  with \(\rg(V_j) \subseteq X_j\) and \(X_j \cap \s(U_j) \setminus
  \rg(W_j) = \emptyset\), so that \(X_j \cap \s(U_j) \subseteq
  \rg(W_j)\).
  Then \(X_j V_j = V_j\), so that \((U_j,V_j) \equiv (U_j X_j, V_j) =
  (U_j W_j W_j^{-1}, V_j) \equiv (U_j W_j,W_j V_j)\).
  Thus \((U_1,V_1)\) and \((U_2,V_2)\) are identified when we divide
  out the relations \((U W,V) \equiv (U, W V)\).

  Let \(U,U_1,U_2\in\B_{\Bisp}\), \(V,V_1,V_2\in\B_{\Bisp[Y]}\) be
  such that \(U = U_1 \sqcup U_2\) and \(V = V_1 \sqcup V_2\).
  In \(A_R(\Bisp)\otimes_{ A_R(\Gr)} A_R(\Bisp[Y])\), we also divide
  out the relations \(\delta_{(U_1 \sqcup U_2,V)} - \delta_{(U_1,V)} -
  \delta_{(U_2,V)}\) and \(\delta_{(U,V_1 \sqcup V_2)} -
  \delta_{(U,V_1)} - \delta_{(U,V_2)}\).
  We claim that, after identifying the generators
  \(\delta_{(U_1,V_1)}\) and \(\delta_{(U_2,V_2)}\) when \(U_1 V_1 = U_2
  V_2\), these relations are exactly the same relations as the relation
  \(\delta_X - \delta_{X_1} - \delta_{X_2}\) for all \(X,X_1,X_2 \in
  \B_{\Bisp\circ\Bisp[Y]}\) with \(X= X_1 \sqcup X_2\), which occur in
  the presentation of \(A_R(\Bisp\circ\Bisp[Y])\) in
  Theorem~\ref{the:steinmod_is_quot_of_sum}.

  First, the assumptions above imply \(U V = U_1 V \cup U_2 V = U V_1
  \cup U V_2\).
  In addition, since \(U \supseteq U_1,U_2\) and \(V\supseteq
  V_1,V_2\) are slices, \eqref{eq:intersect_in_BXY} implies that \(U_1 V
  \cap U_2 V = \emptyset\) and \(U V_1 \cap U V_2 = \emptyset\) in
  \(\Bisp\circ\Bisp[Y]\).
  Therefore, \(\delta_{U V} - \delta_{U_1 V} - \delta_{U_2 V}\) and
  \(\delta_{U V} - \delta_{U V_1} - \delta_{U V_2}\) are among the
  generating relations in Theorem~\ref{the:steinmod_is_quot_of_sum}.

  Conversely, let \(X_1\sqcup X_2 = X\) with disjoint
  \(X_1,X_2,X\in\BXY\).
  Write \(X = U V\) for \(U\in\B_{\Bisp}\), \(V\in\B_{\Bisp[Y]}\)
  with \(\s(U) \supseteq \rg(V)\).
  Equation~\eqref{eq:intersect_in_BXY} applied to the trivial
  statements \(X_j = X\cap X_j\) for \(j=1,2\) implies that \(X_j = U
  V_j\) for some \(V_1,V_2 \subseteq V\) with
  \(V_1,V_2\in\B_{\Bisp[Y]}\).
  In addition, it follows that \(X_1 \cap X_2 = U(V_1 \cap
  V_2)\) and \((X\setminus X_1) \setminus X_2 = U \bigl((V\setminus V_1)
  \setminus V_2\bigr)\).
  Since \(\s(U) \supseteq \rg(V)\), these sets would be nonempty if
  \(V_1 \cap V_2\) or \((V\setminus V_1) \setminus V_2\) were
  nonempty, respectively.
  Thus, \(V_1 \cap V_2 = \emptyset\) and \(V_1 \sqcup V_2 =V\).
  So our identification maps the relation \(\delta_{(U,V)} -
  \delta_{(U,V_1)} - \delta_{(U,V_2)}\) from the presentation of
  \(A_R(\Bisp)\otimes_{ A_R(\Gr)} A_R(\Bisp[Y])\) to the relation
  \(\delta_X - \delta_{X_1} - \delta_{X_2}\) from the presentation of
  \(A_R(\Bisp\circ\Bisp[Y])\).
  As a consequence, both \(A_R(\Bisp)\otimes_{ A_R(\Gr)}
  A_R(\Bisp[Y])\) and \(A_R(\Bisp\circ\Bisp[Y])\) are the quotients of
  the free module on~\(\BXY\) by the same relations.
  This makes them isomorphic.
  The isomorphism is the map that maps \(\charmap{U}\otimes
  \charmap{V}\) to~\(\charmap{U V}\).
  The same formula holds for~\(\tilde{\mu}_{\Bisp,\Bisp[Y]}\), and so
  the latter is an isomorphism as asserted.

  If \(X\in\B_{\Gr[H]}\), \(Y\in\B_{\Gr[K]}\), then \(\charmap{X} *
  \charmap{U} = \charmap{X U}\), \(\charmap{V} * \charmap{Y} = \charmap{V
  Y}\), \(\charmap{X} * \charmap{U V} = \charmap{(X U) V}\),
  and \(\charmap{U V} * \charmap{Y} = \charmap{U (V Y)}\).
  These computations on slices imply immediately that the
  map~\(\mu_{\Bisp,\Bisp[Y]}\) is an \(A_R(\Gr[H]) ,
  A_R(\Gr[K])\)-bimodule homomorphism.
  If \(f\colon \Bisp\to\Bisp'\) and \(g\colon \Bisp[Y]\to\Bisp[Y]'\)
  are \(2\)\nb-arrows of groupoid correspondences, then \(f(U)\subseteq
  \Bisp'\) and \(g(V) \subseteq \Bisp[Y]'\) are slices as well, and the
  induced map \(f\circ g\colon \Bisp \circ \Bisp[Y] \to \Bisp' \circ
  \Bisp[Y]'\) maps \(U V\) to \(f_*(U) g_*(V)\).
  This implies that~\(\mu_{\Bisp,\Bisp[Y]}\) is natural.
\end{proof}

\begin{theorem}
  \label{thm:pseudofunctor}
  The following data defines a normal pseudofunctor
  \(\A\colon\Grcat\to\Rings\):
  \begin{itemize}
  \item the map on objects sends an ample groupoid~\(\Gr\) to its
    Steinberg algebra \(A_R(\Gr)\);
  \item for \(\Gr,\Gr[H]\in\Grcat\) the functor
    \[
      \A_{\Gr,\Gr[H]} \colon \Grcat (\Gr,\Gr[H])
      \to \Rings \bigl( A_R(\Gr), A_R(\Gr[H])\bigr)
    \]
    maps an ample groupoid correspondence~\(\Bisp\) to its Steinberg
    bimodule \(A_R(\Bisp)\) and a continuous equivariant map~\(f\) to the
    bimodule homomorphism \(A_R(f)\);
  \item for ample groupoid correspondences
    \(\Bisp\colon \Gr[H]\leftarrow \Gr\) and
    \(\Bisp[Y]\colon\Gr \leftarrow \Gr[K]\), the natural bimodule
    isomorphisms
    \(\mu_{\Bisp,\Bisp[Y]}\colon A_R(\Bisp)\otimes_{ A_R(\Gr)}
    A_R(\Bisp[Y]) \to A_R(\Bisp\circ\Bisp[Y])\).
  \end{itemize}
  The pseudofunctor~\(\A\) maps the subbicategory of proper groupoid
  correspondences to the subbicategory of proper smooth bimodules.
\end{theorem}

\begin{proof}
  Being normal means that the identity groupoid
  correspondence on an ample groupoid~\(\Gr\) is mapped to the
  identity bimodule on the Steinberg algebra~\(A_R(\Gr)\) and that the
  multiplicativity maps \(\mu_{\Gr,\Bisp[Y]}\) and \(\mu_{\Bisp,\Gr}\)
  are the canonical bimodule maps.
  This is easy to check using that \(\mu_{\Bisp,\Bisp[Y]}\) maps
  \(\charmap{U} \otimes \charmap{V} \mapsto \charmap{U V}\) for
  compact slices \(U,V\) in \(\Bisp,\Bisp[Y]\).
  It is also easy to check that~\(\mu_{\Bisp,\Bisp[Y]}\) is natural
  with respect to maps of groupoid correspondences.
  For a pseudofunctor, it remains to check the following commuting
  diagram for three composable ample correspondences \(\Bisp\colon
  \Gr[H]\leftarrow \Gr\), \(\Bisp[Y]\colon\Gr \leftarrow \Gr[K]\) and
  \(\Bisp[Z]\colon \Gr[K]\leftarrow\Gr[L]\):
  \[
    \begin{tikzcd}[column sep=large]
      {\bigl( A_R(\Bisp)\otimes_{ A_R(\Gr)} A_R(\Bisp[Y])\bigr)\otimes_{ A_R(\Gr[K])} A_R(\Bisp[Z])}
      \arrow[r, "{\mu_{\Bisp,\Bisp[Y]}\otimes \id}", "{\cong}"']
      \arrow[d, "\mathrm{assoc}", "{\cong}"']
      & { A_R(\Bisp\circ \Bisp[Y])\otimes_{A_R(\Gr[K])} A_R(\Bisp[Z])}
      \arrow[d, "{\mu_{\Bisp\circ \Bisp[Y], \Bisp[Z]}}", "{\cong}"'] \\
      { A_R(\Bisp)\otimes_{ A_R(\Gr)} \bigl(A_R(\Bisp[Y]) \otimes_{A_R(\Gr[K])} A_R(\Bisp[Z])\bigr)}
      \arrow[d, "{\id\otimes\mu_{\Bisp[Y],\Bisp[Z]}}", "{\cong}"']
      & { A_R({(\Bisp\circ \Bisp[Y])\circ_{\Gr[K]}\Bisp[Z]})}
      \arrow[d, "\A(\mathrm{assoc})", "{\cong}"']\\
      { A_R(\Bisp)\otimes_{ A_R(\Gr)} A_R({\Bisp[Y]\circ_{\Gr[K]}\Bisp[Z]})}
      \arrow[r, "{\mu_{\Bisp, \Bisp[Y]\circ_{\Gr[K]} \Bisp[Z]}}", "{\cong}"']
      & { A_R({\Bisp\circ (\Bisp[Y]\circ_{\Gr[K]}\Bisp[Z])})}
    \end{tikzcd}
  \]
  This commutes because both ways around the diagram map
  \(\charmap{U} \otimes \charmap{V} \otimes \charmap{W}\) for compact
  slices \(U\), \(V\) and~\(W\) in \(\Bisp\), \(\Bisp[Y]\)
  and~\(\Bisp[Z]\) to \(\charmap{U V W}\).
  The pseudofunctor maps proper correspondences to proper bimodules by
  Proposition~\ref{prop:Steinpro}.
\end{proof}

\begin{example}
  \label{exa:graph_Steinberg}
  We show that Example~\ref{ex:graphmod} comes from a rather trivial
  groupoid correspondence.  Let \(\Gr=\Gr[H]=V\) be a discrete set
  with only identity arrows.  This is an ample groupoid and
  \(A_\mathbb{K}(V) = \bigoplus_{v\in V} \mathbb{K}\).  The maps
  \(\rg,\s\colon E\rightrightarrows V\) make~\(E\) a groupoid
  correspondence on~\(V\), and
  \(A_\mathbb{K}(E) = \bigoplus_{e\in E} \mathbb{K}\) with the
  bimodule structure in Example~\ref{ex:graphmod}.  The
  correspondence~\(E\) is proper if and only if~\(\rg\) is
  finite-to-one.

  Let \(n\in\N\).
  The composite correspondence~\(E^{\circ n}\) is the set of all paths
  of length~\(n\) in the graph defined by \(\rg,\s\colon
  E\rightrightarrows V\).
  The multiplicativity of the Steinberg construction gives the
  well-known isomorphism
  \[
    A_\mathbb{K}(E^{\circ n})
    \cong A_\mathbb{K}(E)^{\otimes_{A_\mathbb{K}(V)} n}.
  \]
\end{example}

\begin{corollary}[see \cite{Steinberg:Sheaves}*{Corollary 3.6}]
  \label{cor:Morita_equivalent_groupoids}
  If two ample groupoids \(\Gr\) and~\(\Gr[H]\) are Morita equivalent,
  then their Steinberg algebras are Morita equivalent.
\end{corollary}

\begin{proof}
  The equivalences in the bicategories of groupoid correspondences and
  of bimodules are exactly the Morita equivalences of groupoids and
  rings with local units, respectively.
  For the groupoid case, this is
  \cite{Meyer:Groupoid_models_relative}*{Theorem~2.3}.
  The case of rings is trivial if Morita equivalence is defined using
  equivalence bimodules.
  Since a pseudofunctor maps equivalences to equivalences, the
  Steinberg algebra pseudofunctor maps Morita equivalences of
  groupoids to Morita equivalences of rings with local units.
\end{proof}

\section{Diagrams of rings and covariant representations}
\label{sec:diagrams}

We define diagrams of rings and bimodules and their covariance rings,
characterised by a universal property for covariant representations.
We prove some basic results about covariance rings that work for
all proper diagrams.
This includes the existence of covariance rings and a description by
generators and relations in case the diagram comes from a diagram of
groupoid correspondences.
We also prove that the covariance ring is a bicategorical limit.

From now on, we fix a monoid~\(P\), written multiplicatively.
Let~\(1\) denote its unit element.
We view~\(P\) as a category and thus as a strict bicategory.
We are going to define diagrams of shape~\(P\) as pseudofunctors
\(P\to\Rings\).
We are going to define (lax) covariant representations of such diagrams
and use these to define the covariance ring of a diagram.
We could do all this in the more general setting of a lax functor from
any small bicategory to~\(\Rings\), but we only work with monoids for
simplicity.
The following definition makes the definition of a lax functor and a
pseudofunctor explicit in the concrete case that we need.
See also Definition~\ref{def:diagrams_in_Grcat} for a similar
definition in the bicategory of groupoid correspondences.

\begin{definition}
  \label{def:diagram_in_rings}
  A \emph{lax diagram in~\(\Rings\)} of shape~\(P\) is a normal
  lax functor \(\F\colon P\to \Rings\), that is, it is described by the
  following data \(\mathcal{F}=(P, F_p, \mu_{p,q})\) and conditions:
  \begin{itemize}
  \item a ring with local units \(A\);
  \item for \(p\in P\), a smooth \(A,A\)-bimodule \(F_p\);
  \item for \(p,q\in P\), an \(A,A\)-bimodule homomorphism
    \(\mu_{p,q}\colon F_p\otimes_{A} F_q\to F_{pq}\);
  \end{itemize}
  such that
  \begin{itemize}
  \item if \(p=1\), then~\(F_1\) is the unit \(A,A\)-bimodule~\(A\), and
    the maps \(\mu_{1,q}\) and~\(\mu_{q,1}\) are the canonical
    isomorphisms \(A \otimes_{A} F_q \congto F_q\) and
    \(F_q \otimes_{A} A \congto F_q\);
  \item if \(p,q,t\in P\), then the following diagram commutes:
    \[
      \begin{tikzcd}[column sep=large]
        F_p\otimes_{A} F_q \otimes_{A} F_t
        \arrow[d, "{\mu_{p,q}\otimes \id}"]
        \arrow[r, "{\id\otimes \mu_{q,t}}"]
        & F_p\otimes_{A} F_{q t} \arrow[d, "{\mu_{p,q t}}"] \\
        F_{pq}\otimes_{A} F_t \arrow[r, "{\mu_{p q,t}}"]
        & F_{p q t}
      \end{tikzcd}
    \]
  \end{itemize}
  If, in addition, the maps~\(\mu_{p,q}\) are isomorphisms,
  then~\(\F\) is a \emph{pseudofunctor} and we call~\(\F\) a
  \emph{diagram} or a \emph{strong diagram} in \(\Rings\).
\end{definition}

Since \(A=F_1\) with the multiplication map~\(\mu_{1,1}\), the
ring~\(A\) and the \(A\)\nb-bimodule structures on~\(F_p\) are
redundant.
We just write~\(F_1\) for~\(A\) in the following.

\begin{definition}
  \label{proper}
  A (lax) diagram~\(\F\) is called \emph{proper} if for all \(p\in P\) the
  \(F_1\)-bimodule~\(F_p\) is proper, that is, if \(e\cdot F_p\) is fgp for
  each idempotent~\(e\) in~\(F_1\).
\end{definition}

In category theory, we are interested in the limit of a
diagram~\(\F\), which is a representing object of the cone functor,
where a cone over~\(\F\) with summit~\(D\) is a natural transformation
\(D\Rightarrow \F\).
Similarly, a limit in a bicategory is defined by suitable cones
\(D\Rightarrow \F\) and a universal object representing the cone
pseudofunctor, see Section~\ref{sec:bilimit}.
By design, such bicategorical limits are only unique up to
equivalence.
In \(\Rings\), this means up to Morita equivalence
(see~\cite{Anh-Marki:Morita_without_identity}).
Therefore, we prefer a slightly more technical definition, which
distinguishes a particular limit as its \emph{covariance ring}.
This is based on \emph{covariant representations}, which are a special
kind of cones over the diagram.
We will see in Section~\ref{sec:bilimit} that the covariance ring is a
limit both in \(\Rings\) and~\(\Ringspr\).

\begin{definition}
  \label{def:cov_rep_of_F}
  For a lax diagram in \(\Rings\) given by \(\F=(P, F_p, \mu_{p,q})\)
  and a ring~\(D\) with local units, a \emph{lax covariant
    representation of}~\(\F\) in~\(D\) is a family of additive maps
  \(\tilde\nu_p\colon F_p \to D\) for \(p\in P\) that satisfy
  \(\tilde\nu_p(x) \cdot \tilde\nu_q(y) = \tilde\nu_{p
    q}(\mu_{p,q}(x\otimes y))\) for all \(x\in F_p\), \(y\in F_q\),
  \(p,q\in P\).
  These maps induce maps
  \[
    \nu_p\colon F_p \otimes_{F_1} D \to D,\qquad
    x\otimes d \mapsto \tilde\nu_p(x)(d).
  \]
  We call \((\tilde\nu_p)_{p\in P}\) a \emph{strong covariant
    representation} or just \emph{covariant representation} of~\(\F\)
  in~\(D\) if the maps~\(\nu_p\) are isomorphisms for all \(p\in P\).
\end{definition}

Let \(\URing\) denote the concrete category that has rings with local
units as objects and nondegenerate homomorphisms as arrows.

\begin{proposition}
  For a lax diagram~\(\F\) and a ring~\(D\) with local units, we
  define \(\LaxCovRep{\F}\) as the set of all lax covariant
  representations of~\(\F\) in~\(D\).  This becomes a functor
  \(\URing \to\Set\), where \(f\in \URing(D_1, D_2)\) induces the map
  \[
    f_* \colon\LaxCovRep[D_1]{\F}\to\LaxCovRep[D_2]{\F}
  \]
  that maps \((\tilde\nu_p)_{p\in P}\) to
  \((f\circ\tilde\nu_p)_{p\in P}\) for all \(p\in P\).  If the
  covariant representation \((\tilde\nu_p)_{p\in P}\) is strong, then
  so is \((f\circ\tilde\nu_{p})_{p\in P}\).  Hence there is a functor
  \(\CovRep[-]{\F}\colon \URing\to\Set\) that maps~\(D\) to the set
  of all covariant representations of~\(\F\) in~\(D\).
\end{proposition}

\begin{proof}
  It is trivial that the maps \(\tilde{\eta}_p\defeq
  f\circ\tilde\nu_p\) for \(p\in P\) form a lax covariant
  representation if \((\tilde\nu_p)\) does and that
  \(\LaxCovRep[-]{\F}\) is a functor.
  Assume now that \((\tilde\nu_p)_{p\in P}\) is a strong covariant
  representation.  Let \(p\in P\).  The bimodule homomorphism
  \(F_p\otimes_{F_1}D_2 \to D_2\) corresponding to~\(\tilde{\eta}_p\)
  is the composite map
  \[
    F_p\otimes_{F_1}D_2
    \cong F_p\otimes_{F_1}(D_1\otimes_{D_1}D_2)
    \cong (F_p\otimes_{F_1}D_1)\otimes_{D_1}D_2
    \xrightarrow[\cong]{\nu_p} D_1\otimes_{D_1}D_2
    \cong D_2.
  \]
  Since this is an isomorphism, \((\tilde{\eta}_p)_{p\in P}\) is a
  strong covariant representation.
\end{proof}

\begin{definition}
  \label{def:covariance_ring}
  Let~\(\F\) be a lax diagram.  We call a ring that represents the
  functor \(\CovRep[-]{\F}\) a \emph{strong covariance ring of~\(\F\)} or
  just \emph{covariance ring of~\(\F\)}.
\end{definition}

A covariance ring is unique up to ring isomorphism by the Yoneda
Lemma.

\subsection{The covariance ring as a Cohn localisation}
\label{sec:Cohn_localisation}

We are going to prove that any diagram of proper bimodules has a
covariance ring.
We construct this as a Cohn localisation of the graded ring built from
the diagram.
The Cohn localisation description of the covariance ring offers
interesting information because of some special properties of
localisations.
For instance, it is known that Cohn localisations of quasifree rings
remain quasifree, and this gives the most transparent proof why Leavitt path
algebras are quasifree (see~\cite{Gundelach:Master} for more details
on this).

Let~\(P\) be any monoid and let \(\F=(F_p,\mu_{p,q})\) be a
\(P\)\nb-shaped diagram in \(\Ringspr\).  Let
\(L \defeq \bigoplus_{p\in P} F_p\) with the multiplication that
restricts to the maps
\(\mu_{p,q}\colon F_p \otimes_{F_1} F_q \to F_{p q} \subseteq L\) for
\(p,q\in P\) on the direct summands.  This is a \(P\)\nb-graded ring
with a local unit in \(F_1 \subseteq F_p\).  Even more, it carries a
strong grading because all the maps~\(\mu_{p,q}\) are isomorphisms.
For \(p\in P\), \(F_p \otimes_{F_1} L\) becomes an
\(F_1,L\)\nb-bimodule, and the multiplication maps~\(\mu_{p,q}\) with
fixed~\(p\) and variable~\(q\) define a canonical bimodule map
\(\psi_p\colon F_p \otimes_{F_1} L \to L\).  If \(e\in F_1\) is
idempotent, we may restrict this to a right \(L\)\nb-module map
\[
  e\psi_p\colon e F_p \otimes_{F_1} L \to e L.
\]
Since~\(F_p\) is a proper bimodule, \(e F_p\) is an fgp
\(F_1\)\nb-bimodule.  Thus \(e F_p \otimes_{F_1} L\) is an fgp
\(L\)\nb-module.  So is~\(e L\) by definition.  So \(e\psi_p\) is a
module homomorphism between two fgp \(L\)\nb-modules.

\begin{lemma}
  \label{lem:covariant_rep_through_graded_algebra}
  A covariant representation of~\(\F\) in~\(D\) is the same as a
  nondegenerate homomorphism \(\varphi\colon L\to D\) such that the
  induced maps
  \[
    e\psi_p\otimes_{L} \id_D\colon (e F_p\otimes_{F_1} L) \otimes_L D \to
    ( eL) \otimes_L D
  \]
  are invertible for all \(p\in P\) and all idempotents \(e\in F_1\).
  In fact, it is enough to assume the above for~\(p\) in a given set
  of generators of~\(P\) and~\(e\) in a given local unit in~\(F_1\).
\end{lemma}

\begin{proof}
  A covariant representation consists of maps \(\tilde\nu_p\colon F_p
  \to D\) for \(p\in P\) with some properties.
  The property \(\tilde\nu_p(x) \cdot \tilde\nu_q(y) = \tilde\nu_{p
    q}(\mu_{p,q}(x\otimes y))\) for all \(x\in F_p\), \(y\in F_q\),
  \(p,q\in P\) holds if and only if \(\bigoplus \tilde\nu_p\colon L
  \to D\) is a ring homomorphism.
  The map \(\nu_1\colon F_1 \otimes_{F_1} D \to D\) is an isomorphism
  if and only if the representation of~\(F_1\) in~\(D\) is
  nondegenerate, if and only if the representation of~\(L\) in~\(D\)
  is nondegenerate.
  In the nondegenerate case, the maps~\(\nu_p\) in
  Definition~\ref{def:cov_rep_of_F} are isomorphisms for all~\(p\) if
  and only if their restrictions \(e\nu_p\colon e F_p \otimes_{F_1} \id_D
  \to e D\) are isomorphisms for all \(e,p\).
  These maps are equivalent to the maps \(e\psi_p\otimes_{L} \id_D\).

  We prove the last statement.
  If the map \(e\psi_p\otimes_{L} \id_D\) is an isomorphism, then so
  is the map \(f\psi_p\otimes_{L} \id_D\) for any idempotent~\(f\)
  with \(f \le e\), that is, \(f e = f\).
  Therefore, if \(e\psi_p\otimes_{L} \id_D\) is an isomorphism for
  all~\(e\) in a local unit, then it is an isomorphism for all
  idempotents \(e\in F_1\), and then \(\psi_p\otimes_L \id_D\) is an
  isomorphism.
  This property is hereditary for products in~\(P\), so it suffices if
  this holds for~\(p\) in a given set of generators.
\end{proof}

Next we generalise the usual definition of Cohn localisation to rings
with local units.
The lemma above expresses that the covariance ring is a Cohn
localisation of the ring~\(L\) at the family of maps~\(e\psi_p\) for
all idempotents \(e\in F_1\) and all \(p\in P\).

\begin{definition}[\cite{Schofield:Representation_rings}]
  \label{def:Cohn_localisation}
  Let~\(R\) be a ring with local units.
  Let \(u_i\colon P_i \to Q_i\) for \(i\in I\) be a set of right
  \(R\)\nb-module maps between fgp right \(R\)\nb-modules \(P_i\)
  and~\(Q_i\).
  The \emph{Cohn localisation} of~\(R\) at the set
  \(\setgiven{u_i}{i\in I}\) is the universal ring~\(R'\) with local
  units with a nondegenerate homomorphism \(R\to R'\) such that the
  maps \(u_i \otimes_R \id_{R'}\colon P_i \otimes_R R' \to Q_i
  \otimes_R R'\) are invertible for all \(i\in I\).
  That is, if~\(D\) is another ring with local units and \(f\colon R
  \to D\) is a nondegenerate homomorphism, then~\(f\) factors
  through~\(R'\) if and only if \(u_i \otimes_R \id_D\colon P_i
  \otimes_R D \to Q_i \otimes_R D\) is invertible for all \(i\in I\),
  and this factorisation is unique if it exists.
\end{definition}

\begin{theorem}
  \label{the:covariance_exists}
  A Cohn localisation as above always exists and is unique up to
  isomorphism.
  Any diagram of proper bimodules has a covariance ring.
\end{theorem}

\begin{proof}
  The Yoneda Lemma implies that all Cohn localisations are canonically
  isomorphic.
  The existence proof below uses the same idea as in the unital case.
  First, we use Lemma~\ref{lem:fgp} to identify \(P_i \cong e_i
  R^{m_i}\), \(Q_i \cong f_i R^{n_i}\) using \(m_i,n_i\in\N\) and
  idempotent matrices \(e_i,f_i\).
  Any module map \(P_i \to Q_i\) such as~\(u_i\) becomes the map of
  left multiplication by a matrix \(u_i' \in f_i \Mat_{n_i,m_i}(R) e_i
  \subseteq \Mat_{n_i,m_i}(R)\).
  To make the map~\(u_i\) invertible, we need to adjoin quasi-inverses
  to the matrices~\(u_i'\).
  More precisely, we let~\(R'\) be the ring with adjoined elements
  \((u_i^\dagger)_{1\le j\le m_i, 1\le k\le n_i}\) for \(i\in I\) --
  that is, we adjoin the entries of the
  \(m_i,n_i\)-matrices \(u_i^\dagger\) to~\(R\) -- subject to the
  relations \(u_i^\dagger = e_i\cdot u_i^\dagger \cdot f_i\), \(u_i'
  \cdot u_i^\dagger = f_i\), \(u_i^\dagger \cdot u_i' = e_i\),
  rewritten in terms of the entries of these matrices, that is, as
  relations involving the entries in \(R\) or~\(R'\) of the matrices
  \(u_i'\), \(e_i\), \(f_i\) and~\(u_i^\dagger\).
  These relations express exactly that left multiplication by the
  matrix~\(u_i^\dagger\) is a map \(f_i (R')^{n_i} \to e_i
  (R')^{m_i}\) that is inverse to left multiplication by~\(u_i'\).
  This gives the desired inverse to~\(u_i\).
  The resulting ring~\(R'\) has the desired universal property.
  The universal property of the Cohn localisation of~\(L\) at the
  maps~\(e\psi_p\) is exactly the same as the universal property of
  the covariance ring of the diagram.
\end{proof}

\begin{remark}
  \label{rem:lax_cov_ring_is_given}
  The discussion above shows along the way that the ring \(L\defeq
  \bigoplus_{p\in P} F_p\) with the canonical multiplication \(a\cdot
  b\defeq \mu_{p,q}(a\otimes b)\in F_{pq}\) for \(p,q\in P\) and
  \(a\in F_p\), \(b\in F_q\) is a \emph{lax covariance ring}
  of~\(\F\).
  That is, its nondegenerate representations are in bijection with lax
  covariant representations of~\(\F\).
\end{remark}

\subsection{The covariance ring for a diagram of groupoid
  correspondences}
\label{sec:covariance_ring_Grcorr}

Let~\(P\) be a monoid.
We recall the definition of a (proper or tight) diagram of ample
groupoid correspondences:

\begin{definition}[{compare \cite{Meyer:Diagrams_models}*{Proposition 3.1}}]
  \label{def:diagrams_in_Grcat}
  A \emph{diagram in~\(\Grcat\)} of shape~\(P\) is a \emph{normal
    pseudofunctor} \(P\to\Grcat\), that is, it is described by the
  data \(\X=(P, \Gr, \Bisp_p, \mu_{p,q})\) with
  \begin{itemize}
  \item an ample groupoid~\(\Gr\);
  \item ample groupoid correspondences
    \(\Bisp_p\colon \Gr\leftarrow \Gr\) for all \(p\in P\);
  \item isomorphisms of correspondences
    \(\mu_{p,q}\colon \Bisp_p\circ \Bisp_q\congto \Bisp_{pq}\)
    for all \(p,q\in P\);
  \end{itemize}
  subject to the following conditions:
  \begin{enumerate}
  \item \label{en:diagrams_in_Grcat_1}%
    \(\Bisp_1\) for the unit \(1\in P\) is the identity
    correspondence~\(\Gr\) on~\(\Gr\);
  \item \label{en:diagrams_in_Grcat_2}%
    \(\mu_{p,1}\colon \Bisp_p \circ \Gr \congto \Bisp_p\) and
    \(\mu_{1,p}\colon \Gr \circ \Bisp_p \congto \Bisp_p\) for
    \(p\in P\) are the canonical left and right multiplication maps;
  \item \label{en:diagrams_in_Grcat_3}%
    for all \(p,q,t\in P\), the following diagram of isomorphisms
    commutes:
    \begin{equation}
      \label{eq:coherence_category-diagram}
      \begin{tikzpicture}[yscale=1.2,xscale=2.5,baseline=(current bounding
      box.west)]
        \node (m-1-1) at (144:1)
        {\((\Bisp_{p}\circ \Bisp_{q})
          \circ \Bisp_{t}\)};
        \node (m-1-1b) at (216:1) {\(\Bisp_{p}\circ
          (\Bisp_{q}\circ \Bisp_{t})\)};
        \node (m-1-2) at (72:1)
        {\(\Bisp_{p q}\circ\Bisp_{t}\)};
        \node (m-2-1) at (288:1)
        {\(\Bisp_{p}\circ\Bisp_{q t}\)};
        \node (m-2-2) at (0:.8) {\(\Bisp_{p q t}\)};
        \draw[dar] (m-1-1) -- node[swap]
        {\scriptsize\textup{associator}} (m-1-1b);
        \draw[dar] (m-1-1.north) -- node[very near end]
        {\(\scriptstyle\mu_{p,q}\circ\id_{\Bisp_{t}}\)}
         (m-1-2.west);
        \draw[dar] (m-1-1b.south) -- node[swap,very near end]
        {\(\scriptstyle\id_{\Bisp_{p}}\circ\mu_{q,t}\)}
        (m-2-1.west);
        \draw[dar] (m-1-2.south) -- node[inner sep=0pt]
        {\(\scriptstyle\mu_{p q,t}\)} (m-2-2);
        \draw[dar] (m-2-1.north) -- node[swap,inner sep=1pt]
        {\(\scriptstyle\mu_{p,q t}\)} (m-2-2);
      \end{tikzpicture}
    \end{equation}
  \end{enumerate}
  If all the ample correspondences~\(\Bisp_p\) are tight or proper, we
  call the diagram~\(\X\) \emph{tight} or \emph{proper}, respectively.
\end{definition}

Let \(\X=(P, \Gr, \Bisp_p, \mu_{p,q})\) be a diagram of proper
groupoid correspondences as above.
The diagram~\(\X\) corresponds to a pseudofunctor \(P \to
\Grcat_\prop\).
Its composition with the Steinberg algebra pseudofunctor~\(\A\) in
Theorem~\ref{thm:pseudofunctor} is a pseudofunctor \(\A*\X\colon P\to
\Ringspr\).
This corresponds, in turn, to the data of a proper diagram
in~\(\Ringspr\) as in Definition~\ref{def:diagram_in_rings}.
Unravelling the definitions, we see that \(F_1 = A_R(\Gr)\) and \(F_p
= A_R(\Bisp_p)\) for all \(p\in P\), equipped with the canonical
\(A_R(\Gr)\)-bimodule structure; this is proper by
Proposition~\ref{prop:Steinpro}.
The multiplication map is the composite
\[
  A_R(\Bisp_p) \otimes_{A_R(\Gr)} A_R(\Bisp_q)
  \congto A_R(\Bisp_p \circ \Bisp_q)
  \congto A_R(\Bisp_{p q}),
\]
where the first bimodule isomorphism comes from
Theorem~\ref{the:mu_invertible} and the second is~\(\A(\mu_{p,q})\).
We are going to describe the covariance ring of the diagram \(\A*\X\)
in~\(\Ringspr\).
This uses ample bases~\(\B_p\) of~\(\Bisp_p\) consisting of compact
slices for all \(p\in P\).
We assume
\begin{equation}
  \label{eq:slices_to_generators}
  \mu_{p,q}(U V) \in \B_{p q},\quad
  \braket{U_1}{U_2} \in \B_1
  \qquad\text{for all } p,q\in P,\ U,U_1,U_2\in\B_p,\ V\in\B_q;
\end{equation}
for instance, letting~\(\B_p\) be the set of all compact open slices
for all \(p\in P\) will do.
It is useful to choose~\(\B_p\) smaller to get a small presentation of
the covariance ring.

\begin{theorem}
  \label{the:presentation_covariance_ring_groupoid}
  Let \(\X=(P, \Gr, \Bisp_p, \mu_{p,q})\) be a diagram of proper
  groupoid correspondences and let~\(\B_p\) be
  ample bases for~\(\Bisp_p\) for \(p\in P\)
  satisfying~\eqref{eq:slices_to_generators}.
  Then the covariance ring of the resulting diagram \(\A*\X\)
  in~\(\Ringspr\) is the \(R\)\nb-algebra with the following
  presentation.
  Its generators are elements \(\delta_U\) and~\(\delta_U^*\) for
  \(U\in\B_p\), \(p\in P\), and the relations are
  \begin{itemize}
  \item \(\delta_{U_1} + \delta_{U_2} = \delta_{U}\), \(\delta^*_{U_1}
    + \delta^*_{U_2} = \delta^*_U\) if \(U,U_1,U_2\in\B_p\) for some
    \(p\in P\) and \(U_1 \sqcup U_2 = U\);
  \item \(\delta_U \delta_V = \delta_{\mu_{p,q}(U V)}\) and
    \(\delta_V^* \delta_U^* = \delta_{\mu_{p,q}(U V)}^*\) for all
    \(U\in\B_p\), \(V\in\B_q\), \(p,q\in P\);
  \item \(\delta_{U_1}^* \delta_{U_2} = \delta_{\braket{U_1}{U_2}}\)
    if \(U_1,U_2\in\B_p\), so that \(\braket{U_1}{U_2}\in \B_1\);
  \item let \(p\in P\), \(K\subseteq \Gr^0\) with \(K\in \B_1\), and
    \(U_1,\dotsc,U_n\in\B_p\) be chosen as in
    Proposition~\textup{\ref{pro:Steinberg_proper_with_bases}}; in
    particular, \(\charmap{K} * A_R(\Bisp_p) \cong \bigoplus_{j=1}^n
    \charmap{\s(U_j)} * A_R(\Gr)\); then \(\delta_K = \sum_{j=1}^n
    \delta_{U_i} \delta^*_{U_i}\).
  \end{itemize}
\end{theorem}

Before the proof of the theorem, we develop some more
theory.
This theory will also be used later when~\(P\) is an Ore monoid, in
order to relate the covariance ring of \(\A*\X\) to the Steinberg
algebra of the groupoid model of~\(\X\).
This theory makes it easier to handle the generators~\(\delta_U^*\).

Let \(F=(F_p,\mu_{p,q})\) be any diagram of rings and proper
bimodules.
Equip the space \(\Hom_{-,F_1}(F_p,F_1)\) of right \(F_1\)\nb-module
homomorphisms \(F_p \to F_1\) with the canonical \(F_1\)\nb-bimodule
structure defined above Theorem~\ref{thm:fgp_is_nice}, and let \(F_p^*
\defeq \Hom_{-,F_1}(F_p,F_1) \cdot F_1\) be the subspace on which the
right \(F_1\)\nb-module structure is nondegenerate.
Let~\(\OOO\) be a covariance ring for~\(F\).
The universal covariant representation \(\tilde\nu_p\colon F_p \to
\OOO\) generates canonical \(F_1,\OOO\)-bimodule isomorphisms
\(\nu_p\colon F_p \otimes_{F_1} \OOO \congto \OOO\), \(x\otimes
y\mapsto \tilde\nu_p(x)y\).
Define a map
\[
  \kappa^*_p\colon F_p^* \to \OOO,\qquad
  T\cdot e \mapsto \nu_1 (T\otimes_{F_1} \id_{\OOO}) \nu_p^{-1}
  (\tilde\nu_1(e)),
\]
that is, we evaluate the composite map
\[
  \OOO \xrightarrow{\nu_p^{-1}}
  F_p \otimes_{F_1} \OOO \xrightarrow{T\otimes_{F_1} \id_{\OOO}}
  F_1 \otimes_{F_1} \OOO \xrightarrow{\nu_1}
  \OOO
\]
on the element \(\tilde\nu_1(e)\in \OOO\).
Recall that~\(\nu_1\) is the canonical isomorphism \(x\otimes y\mapsto
x\cdot y\) from the bicategory of rings and bimodules.

\begin{lemma}
  \label{lem:dual_in_covariance_ring}
  The map \(\kappa^*_p\colon F_p^* \to \OOO\) above is well-defined.
\end{lemma}

\begin{proof}
  The only choice was the factorisation of an element of~\(F_p^*\) as
  \(T\cdot e\) for some \(T\in \Hom_{-,F_1}(F_p,F_1)\) and some \(e\in
  F_1\).
  The formulas above define a map
  \[
    \Hom_{-,F_1}(F_p,F_1) \otimes_{F_1} F_1 \to \OOO,\qquad
    T\otimes e\mapsto
    \nu_1 (T\otimes_{F_1} \id_{\OOO}) \nu_p^{-1}
    (\tilde\nu_1(e)).
  \]
  This proves that~\(\kappa_p^*\) is well-defined because the
  multiplication map is an isomorphism
  \(\Hom_{-,F_1}(F_p,F_1) \otimes_{F_1} F_1 \congto
  \Hom_{-,F_1}(F_p,F_1) F_1\).
\end{proof}

Next we define the ``dual space''~\(\Bisp^*\) of a groupoid
correspondence \(\Bisp\colon \Gr[H] \leftarrow \Gr\).
Its underlying space is~\(\Bisp\), but we denote its elements
by~\(x^*\) for \(x\in \Bisp\).
The dual anchor maps are \(\rg^*,\s^*\colon
\Bisp^*\rightrightarrows\Gr^0\), \(\s^*(x)\defeq \rg(x^*)\) and
\(\rg^*(x^*)\defeq \s(x)\), and the actions of \(\Gr\) and~\(\Gr[H]\)
are defined by \(x^* h \defeq (h^{-1} x)^*\), \(g x^* \defeq (x
g^{-1})^*\) for \(x\in\Bisp\), \(g\in\Gr\), \(h\in\Gr[H]\) that are
suitably composable.
Although~\(\Bisp^*\) need not be a groupoid correspondence, the same
formulas turn \(A_R(\Bisp^*)\) into a well-defined smooth bimodule
over the ring \(A_R(\Gr)\).

\begin{proposition}
  \label{pro:dual_of_correspondence}
  Let \(\Bisp\colon \Gr[H]\leftarrow\Gr\) be a proper groupoid
  correspondence over an ample groupoid~\(\Gr\).
  Then the following map is well-defined and an isomorphism of
  \(A_R(\Gr)\)-\(A_R(\Gr[H])\)-bimodules:
  \begin{align*}
    \mathcal{I}_{\Bisp}\colon A_R(\Bisp^*)
    &\to \Hom_{-, A_R(\Gr)}\bigl( A_R(\Bisp), A_R(\Gr)\bigr) A_R(\Gr[H]),\\
    f &\mapsto
        \left[h \mapsto \Bigl[\gamma\mapsto
        \sum_{\substack{x^*\in \Bisp^* \\ \rg^*(x^*)=\s(\gamma)}}
    f(\gamma \cdot x^*)h(x) \Bigr] \right],
  \end{align*}
  for \(f\in A_R(\Bisp^*)\), \(h \in A_R(\Bisp)\), \(\gamma\in\Gr\).
  Let \(U\subseteq \Bisp\) be a slice and denote its image
  in~\(\Bisp^*\) by~\(U^*\).
  Then \(\mathcal{I}_{\Bisp}(\charmap{U^*})\) is the unique
  \(R\)\nb-module map \(A_R(\Bisp) \to A_R(\Gr)\) that
  maps~\(\charmap{V}\) for a slice \(V\subseteq \Bisp\) to
  \(\charmap{\braket{U}{V}}\).
\end{proposition}

\begin{proof}
  We first prove the last claim and let \(U,V\subseteq\Bisp\) be
  slices.
  If \(f=\charmap{U^*}\), \(h = \charmap{V}\), then
  \(\mathcal{I}_{\Bisp}(\charmap{U^*})(\charmap{V})\) is the function
  on~\(\Gr\) that maps \(\gamma\in \Gr\) to the number of
  \(x\in\Bisp\) with \(\s(x) = \s(\gamma)\), \(x\gamma^{-1}\in U\) and
  \(x\in V\).
  Since~\(\s|_V\) is injective, this number is either \(0\) or~\(1\).
  It is~\(1\) if and only if \(x\in V\) as above exists.
  Recall that \(\braket{x}{y}\) for \(x,y\in \Bisp\) is the unique
  \(g\in\Gr\) with \(\rg(g)= \s(x)\) and \(x\cdot g = y\).
  So \(\gamma = \braket{x\gamma^{-1}}{x}\), and thus \(x\in\Bisp\)
  with \(\s(x) = \s(\gamma)\), \(x\gamma^{-1}\in U\) and \(x\in V\)
  exist if and only if \(\gamma = \braket{x \gamma^{-1}}{x} \in
  \braket{U}{V}\).
  Therefore, \(\mathcal{I}_{\Bisp}(\charmap{U^*})(\charmap{V}) =
  \charmap{\braket{U}{V}}\) as claimed.

  By
  \cite{Antunes-Ko-Meyer:Groupoid_correspondences}*{Proposition~3.5},
  if \(U,V\subseteq \Bisp\) and \(W\subseteq \Gr\), \(X\subseteq
  \Gr[H]\) are slices, then
  \begin{equation}
    \label{eq:braket_products}
    \braket{X U}{X V} = \braket{U}{V},
    \qquad
    \braket{U}{V W} = \braket{U}{V} W,\qquad
    \braket{U W}{V} = W^{-1} \braket{U}{V}.
  \end{equation}
  The first condition is equivalent to
  \begin{equation}
    \label{eq:braket_products_2}
    \braket{X^{-1} U}{V} = \braket{U}{X V}.
  \end{equation}
  Characteristic functions of slices span the relevant Steinberg
  spaces as \(R\)\nb-modules by Theorem~\ref{the:steinmod_is_quot_of_sum},
  and the map \((f,h)\mapsto \mathcal{I}_{\Bisp}(f)(h)\) is clearly
  \(R\)\nb-bilinear.
  Therefore, \eqref{eq:braket_products} implies that the map
  \(\mathcal{I}_{\Bisp}(f)\colon A_R(\Bisp) \to A_R(\Gr)\) is right
  \(A_R(\Gr)\)-linear and that~\(\mathcal{I}_{\Bisp}\) is a right
  \(A_R(\Gr[H])\)-linear map to \(\Hom_{-, A_R(\Gr)}\bigl( A_R(\Bisp),
  A_R(\Gr)\bigr)\).
  It is also left \(A_R(\Gr)\)-linear by~\eqref{eq:braket_products_2}.
  Since \(A_R(\Bisp^*)\) is a nondegenerate right module over
  \(A_R(\Gr[H])\), it follows that the range
  of~\(\mathcal{I}_{\Bisp}\) belongs to \(\Hom_{-, A_R(\Gr)}\bigl(
  A_R(\Bisp), A_R(\Gr)\bigr) A_R(\Gr[H])\).
  So~\(\mathcal{I}_{\Bisp}\) is a well-defined
  \(A_R(\Gr)\)-\(A_R(\Gr[H])\)-bimodule map.

  Let \(R,S\) be rings with local units and let \({}_R M_S\) be a
  smooth bimodule.
  Proposition~\ref{prop:loc_then_smooth} implies \(M \cong \varinjlim
  M\cdot e\) as left \(R\)\nb-modules, where the inductive limit runs
  over the set of idempotents in~\(S\).
  We use this for \(S=A_R(\Gr[H])\).
  We may replace the set of all idempotents by the local unit consisting of the
  idempotents \(\charmap{K}\in A_R(\Gr[H])\) for compact open subsets
  \(K\subseteq \Gr[H]^0\) because this is a cofinal subset.
  Since~\(\mathcal{I}_{\Bisp}\) is right \(A_R(\Gr[H])\)-linear, it
  restricts to a map from \(A_R(\Bisp^*)* \charmap{K}\) to \(\Hom_{-,
    A_R(\Gr)}\bigl( A_R(\Bisp), A_R(\Gr)\bigr) A_R(\Gr[H])*
  \charmap{K}\).
  It suffices to prove that these maps are bijective for all~\(K\).
  Here
  \begin{multline*}
    \Hom_{-, A_R(\Gr)}\bigl( A_R(\Bisp), A_R(\Gr)\bigr) A_R(\Gr[H])* \charmap{K}
    =\Hom_{-, A_R(\Gr)}\bigl( A_R(\Bisp), A_R(\Gr)\bigr)* \charmap{K}
    \\\cong \Hom_{-, A_R(\Gr)}\bigl(\charmap{K}* A_R(\Bisp), A_R(\Gr)\bigr).
  \end{multline*}
  Since \(\Bisp\) is proper, Proposition~\ref{prop:Steinpro} implies
  that \(\charmap{K}* A_R(\Bisp)\) is an fgp right \(A_R(\Gr)\)-module.
  Even more, the proof of the proposition shows that \(\charmap{K}*
  A_R(\Bisp) \cong \bigoplus_{i=1}^n \charmap{\s(W_i)}* A_R(\Gr)\) for
  compact open slices \(W_i \subseteq \Bisp\), \(i=1,\dotsc,n\), whose
  images in~\(\Bisp/\Gr\) cover \(\rg_*^{-1}(K) \subseteq \Bisp/\Gr\).
  Thus \(A_R(\Bisp^*)* \charmap{K}\) is an fgp \emph{left}
  \(A_R(\Gr)\)-module and \(A_R(\Bisp^*)* \charmap{K} \cong
  \bigoplus_{i=1}^n A_R(\Gr)* \charmap{\s(W_i)}\) as left
  \(A_R(\Gr)\)-modules.
  By the proof of Theorem~\ref{thm:fgp_is_nice}, evaluation at
  \(\charmap{\s(W_i)}\in \charmap{\s(W_i)}* A_R(\Gr)\) is an
  isomorphism
  \[
    \Hom_{-, A_R(\Gr)}\bigl(\charmap{\s(W_i)}* A_R(\Gr), A_R(\Gr)\bigr)
    \congto A_R(\Gr) * \charmap{\s(W_i)}.
  \]
  The image of \(\charmap{\s(W_i)}* A_R(\Gr)\) in \(A_R(\Bisp)\) is
  \(\charmap{W_i}\).
  The restriction of \(\mathcal{I}_{\Bisp}(\charmap{U^*})\) to
  \(\charmap{\s(W_i)}* A_R(\Gr)\) maps this to
  \(\charmap{\braket{U}{W_i}}\).
  Therefore, we get an isomorphism
  \[
    \Hom_{-, A_R(\Gr)}\bigl(\charmap{K}* A_R(\Bisp), A_R(\Gr)\bigr)
    \cong \bigoplus_{i=1}^n A_R(\Gr)* \charmap{\s(W_i)}
    \cong A_R(\Bisp^*) * \charmap{K},
  \]
  and it maps \(\mathcal{I}_{\Bisp}(\charmap{U^*})\) for a slice
  \(U\subseteq \Bisp\) to the sum of the characteristic functions of
  \(\braket{U}{W_i}W_i^* = (W_i\braket{W_i}{U})^*\).
  This sum is the characteristic function of \((U\cap
  \rg^{-1}(K))^*\).
  As a consequence, this isomorphism is inverse to the restriction
  of~\(\mathcal{I}_{\Bisp}\).
  This implies that the latter is an isomorphism.
\end{proof}

\begin{proof}[Proof of Theorem~\textup{\ref{the:presentation_covariance_ring_groupoid}}]
  Let~\(\mathcal{Q}\) be the \(R\)\nb-algebra with the presentation in
  the theorem.
  We are going to prove that~\(\mathcal{Q}\) is isomorphic to the
  covariance ring of \(\A*\X\), which exists by
  Theorem~\ref{the:covariance_exists}.
  Theorem~\ref{the:steinmod_is_quot_of_sum} says that each
  \(A_R(\Bisp_p)\) is isomorphic to the \(R\)\nb-module that is
  generated by elements~\(\delta_U\) for \(U\in \B_p\) subject to the
  relations \(\delta_U = \delta_{U_1} + \delta_{U_2}\) if
  \(U,U_1,U_2\in\B_p\) satisfy \(U_1 \sqcup U_2 = U\).
  Then the \(P\)\nb-graded ring \(L \defeq \bigoplus_{p\in P}
  A_R(\Bisp_p)\) defined in Section~\ref{sec:Cohn_localisation} is
  generated as an \(R\)\nb-algebra by~\(\delta_U\) for all
  \(U\in\B_p\), \(p\in P\), subject to the relation above and the
  relation \(\delta_U \delta_V = \delta_{U V}\) for all \(p,q\in P\), \(U\in\B_p\)
  and \(V\in\B_q\).
  Mapping each~\(\delta_U\) above to the corresponding generator
  of~\(\mathcal{Q}\) defines a homomorphism \(L \to \mathcal{Q}\).
  We claim that this satisfies the criterion in
  Lemma~\ref{lem:covariant_rep_through_graded_algebra}, so that it
  corresponds to a covariant representation of \(\A*\X\) and induces a
  homomorphism \(\OOO\to \mathcal{Q}\).

  It suffices to check the condition in
  Lemma~\ref{lem:covariant_rep_through_graded_algebra} for idempotents
  of the form~\(\charmap{K}\) for \(K\subseteq \Gr^0\) with
  \(K\in\B_1\) because any compact open subset of~\(\Gr^0\) is a
  disjoint union of such subsets by
  Proposition~\ref{pro:ample_base_unions}, and the idempotent
  elements~\(\charmap{K}\) for \(K\subseteq \Gr^0\) form a local unit
  in \(A_R(\Gr)\) by Proposition~\ref{prop:steinalg_loc_units}.
  Proposition~\ref{pro:Steinberg_proper_with_bases} identifies
  \(\charmap{K} * A_R(\Bisp_p) \cong \bigoplus_{i=1}^n \delta_{\s(U_i)}
  * A_R(\Gr)\) with certain \(U_1,\dotsc,U_n \in \B_p\).
  So the criterion in
  Lemma~\ref{lem:covariant_rep_through_graded_algebra} is that
  \[
    \psi_{K,p}\colon \bigoplus_{i=1}^n \delta_{\s(U_i)} * \mathcal{Q}
    \congto \delta_K * \mathcal{Q},
    \qquad \sum_{i=1}^n x_i\mapsto \sum_{i=1}^n \delta_{U_i}* x_i
  \]
  is invertible for all \(p\in P\), \(K\subseteq \Gr^0\) with \(K\in \B_1\), and
  \(U_1,\dotsc,U_n\in\B_p\) as above.
  We claim that the map
  \[
    \psi^*_{K,p}\colon \delta_K * \mathcal{Q}
    \congto \bigoplus_{i=1}^n \delta_{\s(U_i)} * \mathcal{Q},
    \qquad x \mapsto (\delta^*_{U_i}* x)_{i=1,\dotsc,n},
  \]
  is inverse to~\(\psi_{K,p}\).
  Indeed, \(\psi_{K,p}\psi^*_{K,p}\) is the identity map on \(\delta_K
  * \mathcal{Q}\) because of the relation  \(\delta_K =
  \sum_{j=1}^n \delta_{U_i} \delta^*_{U_i}\) in~\(\mathcal{Q}\).
  The composite  \(\psi_{K,p}^*\psi_{K,p}\) is multiplication by the
  matrix in \(\Mat_n(\mathcal{Q})\) with entries \(\delta_{U_i}^*
  \delta_{U_j} = \delta_{\braket{U_i}{U_j}}\).
  Since \(\Qu(U_i)\cap \Qu(U_j)=\emptyset\) for \(i\neq j\), this matrix
  is diagonal.
  Its \(i\)th entry is \(\delta_{\braket{U_i}{U_i}} =
  \delta_{\s(U_i)}\), which describes the identity map on the summand
  \(\delta_{\s(U_i)} * \mathcal{Q}\).
  So the criterion in
  Lemma~\ref{lem:covariant_rep_through_graded_algebra} is satisfied.
  Thus the tautological formula \(\delta_U \mapsto \delta_U\) defines
  a covariant representation of~\(\A*\X\) in~\(\mathcal{Q}\) and induces
  a homomorphism \(\OOO\to\mathcal{Q}\).

  Next, we construct a homomorphism \(\mathcal{Q} \to \OOO\).
  We must find images for all the generators \(\delta_U\)
  and~\(\delta_U^*\) of~\(\mathcal{Q}\).
  Of course, we map \(\delta_U\in \mathcal{Q}\) for \(U\in\B_p\),
  \(p\in P\) to the image of  \(\charmap{U}\in A_R(\Bisp_p)\)
  in~\(\OOO\).
  We interpret \(\delta_U^*\) as the characteristic function
  \(\charmap{U^*} \in A_R(\Bisp_p^*)\).
  Now we use the isomorphism
  \[
    \mathcal{I}_{\Bisp_p}\colon A_R(\Bisp_p^*)
    \to \Hom_{-, A_R(\Gr)}\bigl( A_R(\Bisp_p), A_R(\Gr)\bigr) A_R(\Gr)
    = A_R(\Bisp_p)^*
  \]
  in Proposition~\ref{pro:dual_of_correspondence} and the canonical
  map \(A_R(\Bisp_p)^* \to \OOO\) defined in
  Lemma~\ref{lem:dual_in_covariance_ring} to map \(\charmap{U^*}\) to
  an element of~\(\OOO\), which we denote by \(\delta_U^*\).
  The relations \(\delta_{U_1} + \delta_{U_2} = \delta_{U}\) and
  \(\delta^*_{U_1} + \delta^*_{U_2} = \delta^*_U\) are satisfied because
  our maps on generators are parts of maps \(A_R(\Bisp_p) \to \OOO\) and
  \(A_R(\Bisp_p)^* \to \OOO\).
  The relation \(\delta_U \delta_V = \delta_{\mu_{p,q}(U V)}\) for all
  \(U\in\B_p\), \(V\in\B_q\), \(p,q\in P\) is already built into the
  covariance ring.
  Next we prove \(\delta_V^* \delta_U^* = \delta_{\mu_{p,q}(U V)}^*\)
  for \(U\in\B_p\), \(V\in\B_q\), \(p,q\in P\) in~\(\OOO\).
  Recall that the canonical map \(\nu_t\colon F_t \otimes_{F_1} \OOO
  \to \OOO\) is an isomorphism for all \(t\in P\).
  Therefore, the map \(\nu_p (\nu_q\otimes \id)\colon F_p
  \otimes_{F_1} F_q \otimes_{F_1} \OOO \to \OOO\) is an isomorphism.
  Since elements of the form \(\charmap{W}\) for \(W\in \B_t\)
  generate~\(F_t\) as an \(R\)\nb-module, it follows that two elements
  \(x_1,x_2\in\OOO\) are equal if \(x_1 \cdot \charmap{W}
  \cdot \charmap{X} = x_2 \cdot \charmap{W}\cdot \charmap{X}\) for all
  \(W \in \B_p\), \(X\in \B_q\).
  Here \(\charmap{W} \cdot \charmap{X} = \charmap{W X}\) with \(W X
  \in \B_{p q}\),
  \[
    \delta_V^* \cdot \delta_U^* \cdot \charmap{W} \cdot \charmap{X}
    = \delta_V^*\cdot \charmap{\braket{U}{W}}\cdot \charmap{X}
    = \delta_V^*\cdot \charmap{\braket{U}{W} X}
    = \charmap{\braket{V}{\braket{U}{W} X}}.
  \]
  and \(\delta_{U V}^*\cdot \charmap{W X} = \charmap{\braket{U V}{W
    X}}\).
  Thus our formula follows if \(\braket{V}{\braket{U}{W} X} =
  \braket{U V}{W X}\) as subsets of~\(\Gr\).
  To prove this, we show that \(\braket{v}{\braket{u}{w} x} =
  \braket{u v}{w x}\) holds for all \(v,x\in \Bisp_p\), \(u,w\in
  \Bisp_p\) with \(\s(u) = \rg(v)\) and \(\s(w) = \rg(x)\); this
  equality includes the claim that one side is defined if and only if
  the other side is defined.
  Indeed, \(\braket{u v}{w x}\) is the unique element of~\(\Gr\) with
  \(u v \braket{u v}{w x} = w x\).
  The equation \(u v \braket{u v}{w x} = w x\) says that there is
  \(h\in \Gr\) with \(u h = w\) and \(h^{-1} v \braket{u v}{w x} =
  x\).
  This means that \(\braket{u}{w}\) is defined and equal to~\(h\) and
  \(v \braket{u v}{w x} = h x = \braket{u}{w} x\).
  The last equation says that \(\braket{v}{\braket{u}{w} x} =
  \braket{u v}{w x}\) as desired.

  Next, we check the relation \(\delta_{U_1}^* \delta_{U_2} =
  \delta_{\braket{U_1}{U_2}}\) for \(U_1,U_2\in\B_p\).
  The element \(\delta_{U_2} \in \OOO\) is defined so that
  \(\delta_{U_2}\cdot x = \nu_p(\charmap{U_2}\otimes x)\) for all
  \(x\in \OOO\).
  Thus
  \begin{multline*}
    \delta_{U_1}^* \cdot \delta_{U_2} \cdot x
    = \nu_1(\mathcal{I}_{\Bisp_p}
    (\charmap{U_1})\otimes \id_{\OOO}) (\charmap{U_2}\otimes x)
    \\= \nu_1(\mathcal{I}_{\Bisp_p} (\charmap{U_1})(\charmap{U_2})
    \otimes x)
    = \nu_1( \charmap{\braket{U_1}{U_2}} \otimes x)
    = \charmap{\braket{U_1}{U_2}} \cdot x.
  \end{multline*}
  This implies that \(\delta_{U_1}^* \cdot \delta_{U_2}\) is the image
  \(\delta_{\braket{U_1}{U_2}}\) of \(\charmap{\braket{U_1}{U_2}} \in
  A_R(\Gr)\).

  Finally, we check the relation \(\delta_K = \sum_{j=1}^n
  \delta_{U_i}\delta^*_{U_i}\) if \(p\in P\), \(K\subseteq \Gr^0\) with
  \(K\in \B_1\), and \(U_1,\dotsc,U_n\in\B_p\) are chosen as in
  Proposition~\ref{pro:Steinberg_proper_with_bases}.
  To check this, we map \(\OOO\) faithfully to
  \(\Endo_{-,\OOO}(\OOO)\) and use the bimodule isomorphism~\(\nu_p\) to
  identify the latter with the endomorphism ring of the right
  \(\OOO\)\nb-module \(F_p \otimes_{F_1} \OOO\).
  So it suffices to prove that \(\delta_K\) and \(\sum_{j=1}^n
  \delta_{U_i}\delta^*_{U_i}\) induce the same map on \(F_p
  \otimes_{F_1} \OOO\).
  Here~\(\delta_K\) acts by left multiplication with \(\charmap{K}\in
  A_R(\Gr)\) on the tensor factor \(F_p = A_R(\Bisp_p)\).
  Thus~\(\delta_K\) maps \(\charmap{V}\otimes x\) for \(V\in \B_p\),
  \(x\in \OOO\) to \(\charmap{K}*\charmap{V}\otimes x =
  \charmap{\rg_{\Bisp_p}^{-1}(K)\cap V}\).
  Left multiplication by \(\delta^*_{U_i}\) applies
  \(\mathcal{I}_{\Bisp_p}(\charmap{U_i^*})\otimes \id_{\OOO}\) to
  \(\charmap{V}\otimes x\), and then multiplies using~\(\nu_1\).
  This gives \(\charmap{\braket{U_i}{V}} \cdot x\).
  Then~\(\delta_{U_i}\) maps this back to \(\charmap{U_i} \otimes
  \charmap{\braket{U_i}{V}} \cdot x\).
  Since the tensor product is balanced over \(A_R(\Gr)\), this is
  equal to \(\charmap{U_i} * \charmap{\braket{U_i}{V}} \otimes x =
  \charmap{U_i \cdot \braket{U_i}{V}} \otimes x\).
  Here \(U_i \braket{U_i}{V}\) is the set of all \(v\in V\) for which
  \(\Qu(v) \in \Bisp_p/\Gr\) belongs to \(\Qu(U_i)\).
  As a consequence, the sum \(\sum_{i=1}^n \charmap{U_i \cdot
  \braket{U_i}{V}}\) is the characteristic function of
  \[
    \bigsqcup_{i=1}^n V\cap \Qu^{-1}(\Qu(U_i))
    = V \cap \rg^{-1}(K)
  \]
  by construction of \(U_1,\dotsc,U_n\).
  This finishes the proof that \(\delta_K\) and \(\sum_{j=1}^n
  \delta_{U_i}\delta^*_{U_i}\) give the same operator on
  \(A_R(\Bisp_p)\otimes_{A_R(\Gr)} \OOO\).
  This gives the required homomorphism \(\mathcal{Q} \to \OOO\).

  The maps back and forth clearly send each generator~\(\delta_U\) to
  itself.
  Since a covariant representation only specifies how the~\(\delta_U\)
  act, a homomorphism \(\OOO \to \OOO\) is determined uniquely by its
  values on the generators~\(\delta_U\).
  Thus the composite map \(\OOO \to \mathcal{Q} \to \OOO\) is the
  identity map.
  To finish the proof, we prove that the homomorphism \(\OOO \to
  \mathcal{Q}\) is surjective.
  It is enough to prove that its image contains all generators
  \(\delta_U\) and~\(\delta_U^*\) for \(U\in\B_p\), \(p\in P\).
  The generators~\(\delta_U\) belong to the image by construction.
  So do the generators~\(\delta^*_{U_i}\) whenever \(U_i \in\B_p\)
  occurs in the situation of
  Proposition~\ref{pro:Steinberg_proper_with_bases} for some
  \(K\subseteq \Gr^0\) with \(K\in\B_1\).
  
  Let \(U\in\B_p\) for some \(p\in P\).
  Choose \(K\subseteq \Gr^0\) with \(\rg(U) \subseteq K\).
  If necessary, write \(K = \bigsqcup_{i=1}^\ell K_j\) with \(K_j\in
  \B_1\) by Proposition~\ref{pro:ample_base_unions}.
  Decompose \(\rg^{-1}(K) = \bigsqcup_{i=1}^n V_i \cdot \Gr\) as in
  Proposition~\ref{pro:Steinberg_proper_with_bases}.
  Then \(U = K U = \bigsqcup_{j=1}^\ell K_j U\).
  The sets~\(K_j U\) all belong to~\(\B_p\), and we may get any union
  of them from~\(U\) by taking the ones not in the union away
  from~\(U\).
  This implies that any union of these sets belongs to~\(\B_p\).
  So the relation \(\delta_U^* = \sum_{j=1}^\ell \delta_{K_j U}^*\)
  holds in~\(\mathcal{Q}\).
  Therefore, it suffices to show that~\(\delta_{K_j U}^*\) belongs to
  the image of the homomorphism \(\OOO \to \mathcal{Q}\).
  All this shows is that we may assume without loss of generality that
  there is \(K \in\B_1\) with \(\rg(U) \subseteq K\).
  Let \(V_1,\dotsc,V_n\in\B_p\) be such that \(\rg^{-1}(K) =
  \bigsqcup_{i=1}^n V_i \cdot \Gr\) as in
  Proposition~\ref{pro:Steinberg_proper_with_bases}.
  Then the relation \(\delta_K = \sum_{j=1}^n \delta_{V_j}
  \delta_{V_j}^*\) is imposed in~\(\mathcal{Q}\).

  The relations defining~\(\mathcal{Q}\) imply \(\delta_K^* =
  \delta_K\) because both multiplication by~\(\delta_K^*\) and
  by~\(\delta_K\) are inverse to multiplication by~\(\delta_K\)
  as a map \(\delta_K \mathcal{Q} \to \delta_K \mathcal{Q}\),
  and the inverse of any map is unique.
  Therefore,
  \[
    \delta_U^*
    = \delta_{K U}^*
    = \delta_U^* \delta_K^*
    = \delta_U^* \delta_K
    = \delta_U^* \sum_{j=1}^n \delta_{V_j} \delta_{V_j}^*
    = \sum_{j=1}^n \delta_{\braket{U}{V_j}} \delta_{V_j}^*.
  \]
  Since \(\delta_{V_j}^*\) and~\(\delta_{\braket{U}{V_j}}\) belong to
  the image of the homomorphism \(\OOO \to \mathcal{Q}\), so
  does~\(\delta_U^*\).
  Thus the latter homomorphism is surjective and our two homomorphisms
  are inverse to each other.
\end{proof}

\begin{corollary}
  \label{cor:covariance_ring_star}
  Assume the situation of
  Theorem~\textup{\ref{the:presentation_covariance_ring_groupoid}}.
  There is a unique anti-homomorphism \(\iota\colon \OOO \to \OOO\) with
  \(\iota(\delta_U) = \delta_U^*\) and \(\iota(\delta_U^*) =
  \delta_U\) for all \(U\in\B_p\), \(p\in P\).
\end{corollary}

\begin{proof}
  It is manifest that~\(\iota\) preserves all the relations in the
  presentation in the theorem, and so it defines a homomorphism \(\OOO
  \to \OOO^\op\).
\end{proof}

Similarly, the covariance ring over the complex numbers carries a
canonical \Star{}algebra structure.

The presentation in
Theorem~\ref{the:presentation_covariance_ring_groupoid} becomes more
transparent if~\(\Gr^0\) is discrete as a topological space.
For a single correspondence, relative Cuntz--Pimsner algebras in this
case are also studied in~\cite{Meyer:Groupoid_models_relative}.
Since each anchor map \(\s\colon \Bisp_p \to \Gr^0\) is a local
homeomorphism and~\(\Gr^0\) is discrete, all~\(\Bisp_p\) are discrete.
Then
\[
  \B_p \defeq \{\emptyset\} \cup \setgiven[\big]{ \{x\}}{x\in\Bisp_p}
\]
is an ample base for each \(p\in P\).
To simplify notation, we omit \(\{\}\) and denote its nonempty
elements as \(x\in\Bisp_p\).
This family of ample bases satisfies the requirements for
Theorem~\ref{the:presentation_covariance_ring_groupoid}.
Before we write down the resulting presentation,
we note one trivial simplification.
We only have \(U_1 \sqcup U_2\in\B_q\) for two disjoint sets
\(U_1,U_2\in\B_p\) if \(U_1 =\emptyset\) or \(U_2 =\emptyset\).
So the only effect of the relation \(\delta_{U_1} + \delta_{U_2} =
\delta_{U_1 \sqcup U_2}\) is that \(\delta_{\emptyset}=0\) for the
empty set as an element of~\(\B_p\) for all \(p\in P\).
Thus we may simply drop these generators and only
consider~\(\delta_x\) for \(x\in \Bisp_p\), \(p\in P\).
In the relations, we must beware that \(\delta_{x\cdot y}\) or
\(\delta_{\braket{x}{y}}\) are interpreted as~\(0\) when \(x\cdot y\)
or \(\braket{x}{y}\) are not defined, that is, \(\{x\}\cdot \{y\} =
\emptyset\) or \(\braket{\{x\}}{\{y\}}=\emptyset\).
With this simplification,
Theorem~\ref{the:presentation_covariance_ring_groupoid} now gives the
following presentation of the covariance ring:

\begin{corollary}
  \label{cor:presentation_covariance_ring_groupoid}
  Let \(\X=(P, \Gr, \Bisp_p, \mu_{p,q})\) be a diagram of proper
  correspondences between discrete groupoids~\(\Gr\).
  Then the covariance ring of the resulting diagram \(\A*\X\)
  in~\(\Ringspr\) is the \(R\)\nb-algebra with the following
  presentation.
  Its generators are elements \(\delta_x\) and~\(\delta_x^*\) for all
  \(x\in\Bisp_p\), \(p\in P\).
  These are subject to the following relations:
  \begin{enumerate}
  \item \(\delta_x \delta_y = \delta_{\mu_{p,q}(x,y)}\) and
    \(\delta_y^* \delta_x^* = \delta_{\mu_{p,q}(x,y)}^*\) for all
    \(x\in\Bisp_p\), \(y\in\Bisp_q\), \(p,q\in P\); this is understood to
    be~\(0\) if \(\s(x) \neq \rg(y)\) in~\(\Gr^0\);
  \item \(\delta_{x_1}^* \delta_{x_2} = \delta_{\braket{x_1}{x_2}}\)
    if \(x_1,x_2\in\Bisp_p\), so that \(\braket{x_1}{x_2}\in \Bisp_1 =
    \Gr\); this is understood to be~\(0\) if \(\Qu(x_1) \neq \Qu(x_2)\) in
    \(\Bisp/\Gr\);
  \item if \(p\in P\), \(x\in \Gr^0\), and
    \(y_1,\dotsc,y_n\in\Bisp_p\) are representatives for the finitely
    many right \(\Gr\)\nb-orbits in \(\rg_{\Bisp_p}^{-1}(\{x\})
    \subseteq \Bisp_p\), then \(\delta_x = \sum_{j=1}^n \delta_{y_j}
    \delta_{y_j}^*\).
  \end{enumerate}
  These relations imply \(\delta_x^* \delta_y^* = \delta_{\mu_{p,q}(y
    x)}^*\) for all \(x\in\B_p\), \(y\in\B_q\), \(p,q\in P\).
\end{corollary}

In the third relation, \(\rg^{-1}_{\Bisp_p}(\{x\}) \subseteq \Bisp_p\)
consists of finitely many \(\Gr\)\nb-orbits because~\(\Bisp_p\) is a
proper correspondence.
Note that this relation is imposed even if \(\rg^{-1}_{\Bisp_p}(\{x\})
= \emptyset\), when it says that \(\delta_x=0\).
Thus the universal property of the covariance ring leads to rather
undesirable relations if the range map \(\rg\colon \Bisp_p \to
\Gr^0\) fails to be surjective for some \(p\in P\).

\subsection{The covariance ring as a bicategorical limit}
\label{sec:bilimit}

In this section, we prove that the covariance ring of a proper diagram
is a particular realisation of the bicategorical limit of the diagram.
The covariance ring has the advantage that it is unique up to
isomorphism.
An advantage of bicategorical limits is their functoriality.
Namely, the map sending a diagram to its limit is part of a
pseudofunctor.
In particular, a pseudonatural transformation between two proper
diagrams induces a correspondence between the covariance rings in a
natural way.
This is proven exactly as for the bicategory of groupoids and groupoid
correspondences in \cite{Meyer:Diagrams_models}*{Section~10},
compare the proof of Corollary~10.7.
We decided not to discuss this here because this article is already
getting rather long.
In a future project, we will relate pseudonatural transformations in
the special case of higher-rank graphs to the \(k\)\nb-morphs
of~\cite{Kumjian-Pask-Sims:k-morphs} on the groupoid level and to the
bridging modules
of~\cite{Hazrat-Mukherjee-Pask-Sardar:Higher_graphs_K} on the algebra
level.

A limit of a diagram in a bicategory~\(\Cat\) is defined as a
representing object of a pseudofunctor from~\(\Cat\) to the bicategory
of categories (see~\cite{Johnson-Yau:2-Dim}).
We first define this pseudofunctor and then show that the covariance
ring is such a representing object.
Let~\(P\) be a monoid and let \(\F\colon P \to \Ringspr\) be a diagram
of rings and proper bimodules, consisting of a ring with local
units~\(F_1\), smooth, proper \(F_1\)\nb-bimodules~\(F_p\) for \(p\in
P\) and bimodule isomorphisms \(\mu_{p,q}\colon F_p \otimes_{F_1} F_q
\congto F_{p q}\) as in Definition~\ref{def:diagram_in_rings}.
To simplify notation, we shall not keep track of associators for
bimodule tensor products, that is, we pretend that \((X\otimes_{D_1}
Y) \otimes_{D_2} Z = X\otimes_{D_1} (Y \otimes_{D_2} Z)\) and leave
out the associators identifying these bimodules in diagrams.

Let~\(D\) be an object of \(\Ringspr\), that is, a ring with local
units.
A \emph{cone} over~\(\F\) with summit~\(D\) has the following data:
\begin{itemize}
\item a smooth proper \(F_1,D\)-bimodule~\(E\) --~that is, an arrow
  from~\(D\) to~\(F_1\);
\item bimodule isomorphisms \(\nu_p\colon F_p \otimes_{F_1} E \congto
  E\) for \(p\in P\);
\end{itemize}
this is subject to the following conditions: \(\nu_1\colon F_1
\otimes_{F_1} E \congto E\) is the canonical isomorphism as in
Proposition~\ref{pro:Rings_bicat} and the following diagrams commute
for all \(p,q\in P\):
\begin{equation}
  \label{eq:cone_equation}
  \begin{tikzcd}[column sep=huge]
    F_p \otimes_{F_1} F_q \otimes_{F_1} E
    \arrow[r, "\mu_{p,q} \otimes \id_E"]
    \arrow[d, "\id_{F_p} \otimes \nu_q"']
    &
    F_{p q} \otimes_{F_1} E
    \arrow[d, "\nu_{p q}"]
    \\
    F_p \otimes_{F_1} E
    \arrow[r, "\nu_p"']
    &
    E
  \end{tikzcd}
\end{equation}

Let \((E,\nu_p)\) and \((E',\nu_p')\) be two such cones.
An arrow between them is an \(F_1,D\)-bimodule map \(f\colon E\to E'\)
such that \(f\circ \nu_p = \nu'_p \circ (\id_{F_p} \otimes_{F_1} f)\)
for all \(p\in P\).
These arrows are composed in the obvious way, and identity maps
on~\(E\) provide unit arrows.
So the cones over~\(\F\) with summit~\(D\) form a category
\(\Cone(D,F_1)\).
Notice that we consistently treat bimodules as arrows from right to
left, so that the composition is the balanced tensor product in the
same order.
Were we to use the opposite convention and treat a bimodule as an
arrow from left to right, we would get colimits instead of limits as
in~\cite{Albandik-Meyer:Colimits}.

We now describe a pseudofunctor from~\(\Ringspr\) to the bicategory of
categories that maps~\(D\) to the category of cones above.
Let \(D_1\) and~\(D_2\) be two rings with local units and let~\(X\) be
a smooth proper \(D_1,D_2\)-bimodule.
Let \((E,\nu_p)\) be a cone over~\(\F\) with summit~\(D_1\).
Then \((E\otimes_{D_1} X,\nu_p \otimes_{D_1} \id_X)\) is a cone
over~\(\F\) with summit~\(D_2\).
An arrow \(f\colon (E,\nu_p) \to (E',\nu_p')\) in \(\Cone(D_1,\F)\)
induces an arrow \(f\otimes_{D_1} \id_X\) in \(\Cone(D_2,\F)\), and
this makes \({-} \otimes_{D_1} X\) a functor \(\Cone(D_1,\F) \to
\Cone(D_2,\F)\).
A bimodule homomorphism \(g\colon X\to X'\) between two such bimodules
induces an arrow \(\id_E \otimes_{D_1} g\) of cones, and this
construction defines a natural transformation between the functors
\({-} \otimes_{D_1} X\) and \({-} \otimes_{D_1} X'\) from
\(\Cone(D_1,\F)\) to~\(\Cone(D_2,\F)\).
If~\(X\) is the identity bimodule, then \({-} \otimes_{D_1} X\) is
canonically isomorphic to the identity functor and if \(X,Y\) are
bimodules \(D_1 \leftarrow D_2 \leftarrow D_3\), then the functor
\({-} \otimes_{D_1} (X \otimes_{D_2} Y)\) is naturally isomorphic to
the composite functor \(({-} \otimes_{D_1} X) \otimes_{D_2} Y\).
This data makes \(\Cone\) a pseudofunctor from the
bicategory~\(\Ringspr\) to the bicategory~\(\Cats\) of categories.

A \emph{representing object} for this pseudofunctor
\(\Ringspr \to \Cats\) is a ring~\(L\) such that for all rings~\(D\),
there are natural equivalences of categories between the category
\(\Ringspr(D,L)\) of arrows and \(2\)\nb-arrows \(D\leftarrow L\) and
the category \(\Cone(D,\F)\) of cones over~\(\F\) with summit~\(D\).
Such a representing object is also called a \emph{limit} or
\emph{bilimit} of the diagram~\(\F\).

There is a variant of the definitions above where we do not require
the bimodules \(E, X,X'\) above to be proper.
Equivalently, we treat~\(\F\) as a diagram in the larger bicategory
\(\Rings\) and form cones and representing objects there.
We will see that the limit of the diagram in~\(\Ringspr\) is also a
limit in~\(\Rings\).
That is, it makes no difference in which bicategory we take the limit.
If, however, the original diagram~\(\F\) is not proper, then we do not
know whether a limit exists or what it should be.
So it is crucial to assume the bimodules in the diagram to be proper,
but all other bimodules need not be proper.

\begin{theorem}
  \label{the:covariance_ring_limit}
  Let~\(P\) be any monoid and let \(\F=(F_p,\mu_{p,q})\) be a
  \(P\)\nb-shaped diagram in \(\Ringspr\).
  Let~\(\OOO_\F\) be a covariance ring for~\(\F\).
  Then~\(\OOO_\F\)  is a limit of~\(\F\) in \(\Ringspr\) and
  in~\(\Rings\).
\end{theorem}

\begin{proof}
  Let~\(D\) be a ring with local units and let~\(E\) be a smooth
  \(F_1,D\)-bimodule; we do not assume~\(E\) to be proper.
  Roughly speaking, we are going to construct a canonical bijection
  from covariant representations \(\tilde\nu_p\colon F_p\to
  \Endo_{-,D}(E)\) to cones over~\(\F\) with summit~\(D\) based on the
  \(F_1,D\)-bimodule~\(E\).
  More precisely, since covariant representations are implicitly
  assumed to be nondegenerate, we replace \(\Endo_{-,D}(E)\) by the
  target ring \(F_1 \Endo_{-,D}(E) F_1\); a local unit in~\(F_1\) also
  provides one in the latter ring.

  On the one hand, a cone over~\(\F\) with summit~\(D\) based on the
  \(F_1,D\)-bimodule~\(E\) means a family of \(F_1,D\)-bimodule
  isomorphisms \(\nu_p\colon F_p \otimes_{F_1} E \congto E\) such
  that~\(\nu_1\) is the given left \(F_1\)\nb-module structure
  on~\(E\) and the diagram~\eqref{eq:cone_equation} commutes.
  As in Definition~\ref{def:cov_rep_of_F}, \(\nu_p\) corresponds to a
  map \(\tilde\nu_p\colon F_p \to \Endo_{-,D}(E)F_1\) given by
  \(\tilde\nu_p(x)(v) = \nu_p(x\otimes v)\).
  Here~\(\tilde\nu_1\) is the given \(F_1\)\nb-module structure
  and~\eqref{eq:cone_equation} is equivalent to
  \(\tilde\nu_p(x)\tilde\nu_q(y) = \tilde\nu_{p q}(x y)\) for all
  \(p,q\in P\), \(x\in F_p\), \(y\in F_q\).
  In particular, each \(\tilde\nu_p\) is \(F_1\)\nb-linear on the left
  and right.
  Since~\(F_p\) is a smooth \(F_1\)\nb-bimodule, it follows
  that~\(\tilde\nu_p\) is a map to \(F_1 \Endo_{-,D}(E) F_1\).
  Conversely, a family of maps \(\tilde\nu_p\colon F_p \to
  F_1\Endo_{-,D}(E)F_1\) comes from a cone over~\(\F\) with
  underlying \(F_1\)\nb-bimodule~\(E\) if and only if
  \(\tilde\nu_p(x)\tilde\nu_q(y) = \tilde\nu_{p q}(x y)\) for all
  \(p,q\in P\), \(x\in F_p\), \(y\in F_q\), \(\tilde\nu_1\) is the
  given left \(F_1\)\nb-module structure on~\(E\), and the following
  maps for \(p\in P\) are bijective:
  \begin{equation}
    \label{eq:bijection_for_cone}
    \nu_p\colon F_p \otimes_{F_1} E \to E,\qquad
    x\otimes v \mapsto \tilde\nu_p(x)(v).
  \end{equation}
  On the other hand, a covariant representation of our diagram in
  \(F_1\Endo_{-,D}(E)F_1\) is a family of maps
  \((\tilde\nu_p)_{p\in P}\) such that
  \(\tilde\nu_p(x)\tilde\nu_q(y) = \tilde\nu_{p q}(x y)\) for all
  \(p,q\in P\), \(x\in F_p\), \(y\in F_q\), \(\tilde\nu_1\) is the
  given left \(F_1\)\nb-module structure on~\(E\), and the following
  maps for \(p\in P\) are bijective:
  \begin{equation}
    \label{eq:bijection_for_covrep}
    \nu_p\colon F_p \otimes_{F_1} F_1\Endo_{-,D}(E)F_1 \to
    F_1\Endo_{-,D}(E)F_1,\qquad
    x\otimes T \mapsto \tilde\nu_p(x)\cdot T.
  \end{equation}
  The only difference is that~\eqref{eq:bijection_for_covrep}
  replaces~\eqref{eq:bijection_for_cone}.
  So it remains to prove that these two conditions are equivalent.
  To begin with, let us assume that~\eqref{eq:bijection_for_covrep} is
  bijective.
  Then the maps~\(\tilde\nu_p\) induce a nondegenerate representation
  of the covariance ring~\(\OOO_\F\) in \(F_1 \Endo_{-,D}(E)F_1\) by
  the universal property of the covariance ring.
  Since~\(E\) is a nondegenerate left \(F_1\)\nb-module, the ring
  \(F_1 \Endo_{-,D}(E)F_1\) acts nondegenerately on~\(E\).
  Hence \(\OOO_\F \otimes_{\OOO_\F} E \cong E\).
  The universal covariant representation in~\(\OOO_\F\) contains
  \(F_1,\OOO_\F\)-bimodule isomorphisms \(F_p \otimes_{F_1} \OOO_\F
  \cong \OOO_\F\).
  Therefore,
  \[
    F_p \otimes_{F_1} E
    \cong F_p \otimes_{F_1} \OOO_\F \otimes_{\OOO_\F} E
    \cong \OOO_\F \otimes_{\OOO_\F} E
    \cong E.
  \]
  Thus the maps in~\eqref{eq:bijection_for_cone} are bijective if the
  maps in~\eqref{eq:bijection_for_covrep} are bijective.

  Conversely, assume that the maps in~\eqref{eq:bijection_for_cone}
  are bijective.
  Theorem~\ref{thm:proper_is_nice} implies
  \begin{multline*}
    F_p \otimes_{F_1} F_1 \Endo_{-,D}(E) F_1
    = F_1 F_p \otimes_{F_1} \Hom_{-,D}(E,E) F_1
    \\\cong F_1 \Hom_{-,D}(E, F_p \otimes_{F_1} E) F_1
    \cong F_1 \Hom_{-,D}(E, E) F_1
    = F_1 \Endo_{-,D}(E) F_1
  \end{multline*}
  because \(F_p\) is a proper \(F_1\)\nb-bimodule
  and~\eqref{eq:bijection_for_cone} is bijective.

  Along the way, we have replaced covariant representations in
  \(F_1 \Endo_{-,D}(E) F_1\) based on the given action of~\(F_1\) by
  homomorphisms \(\OOO_\F \to F_1 \Endo_{-,D}(E) F_1\) extending the
  given homomorphism on~\(F_1\).
  The latter are also in bijection with homomorphisms \(\OOO_\F \to
  \Endo_{-,D}(E)\) because \(F_1 \to \OOO_\F\) is nondegenerate.
  So we get a bijection between cones over~\(F\) based on the
  \(F_1,D\)-correspondence~\(E\) and \(\OOO_\F,D\)-correspondences
  that restrict to~\(E\) when we restrict the left
  \(\OOO_\F\)\nb-action to~\(F_1\).

  These bijections for different~\(E\) are natural in two ways.
  First, an \(F_1,D\)-correspondence map \(\varphi\colon E\to E'\)
  intertwines the representations of~\(\OOO_\F\) if and only if it
  intertwines the covariant representations of the diagram, if and
  only if it is a morphism of cones.
  Therefore, our bijections combine to an isomorphism of categories
  between the categories of cones over~\(F\) with summit~\(D\) and of
  \(\OOO_\F,D\)-correspondences.
  Secondly, given a \(D,D'\)-correspondence~\(X\), the bijections for
  \(E\) and \(E\otimes_D X\) are related as expected.
  Namely, if \((\nu_p)\) is a cone based on~\(E\), then
  \((\nu_p\otimes_D X)\) is a cone based on \(E\otimes_D X\), and the
  bijection maps the latter cone to the representation of~\(\OOO_\F\)
  that is induced by the covariant representation given by
  \(\tilde\nu_p(\xi) \otimes_D \id_X\) for \(p\in P\), \(\xi\in F_p\).
  Therefore, the isomorphisms of categories above are natural
  in~\(D\), giving a pseudonatural transformation of pseudofunctors
  from \(\Rings\) (or \(\Ringspr\)) to~\(\Cats\).
  This finishes the proof that the covariance ring~\(\OOO_\F\)
  represents the cones pseudofunctor from \(\Rings\) or \(\Ringspr\)
  to~\(\Cats\).
  So it is a bicategorical limit as asserted.
\end{proof}

\section{Covariance rings of proper Ore diagrams}
\label{sec:cov_ring_construction}

We construct the covariance ring explicitly for proper diagrams over
an Ore monoid and show that it is a bicategorical limit in \(\Rings\)
and \(\Ringspr\).
We fix a diagram \(\F=(P, F_p, \mu_{p,q})\) of proper correspondences
in~\(\Ringspr\), where~\(P\) is an \emph{Ore monoid}, that is, it
satisfies the following Ore conditions:

\begin{definition}[compare
    \cite{The_Stacks_project}*{\href{https://stacks.math.columbia.edu/tag/04VB}{Tag
      04VB}}]
  \label{def:Ore_monoid}
  For a monoid~\(P\), the following two properties are called the
  \emph{Ore conditions}:
  \begin{enumerate}[label=\textup{(O\arabic*)}]
  \item \label{enum:O1}%
    For all \(x_1,x_2\in P\), there are \(y_1,y_2\in P\) with
    \(x_1 y_1=x_2 y_2\).
  \item \label{enum:O2}%
    For all \(x, y_1 , y_2\in P\) with \(x y_1 =x y_2\), there is a
    \(z\in P\) with \(y_1 z=y_2 z\).
  \end{enumerate}
  We call~\(P\) an \emph{Ore monoid} if it has these two properties.
\end{definition}

For example, groups and commutative monoids are Ore monoids.
Any Ore monoid has a well-behaved group completion, and the Ore
conditions hold if and only if certain coslice categories that we will
need later are filtered.
First, we need to define the group completion~\(G\) of the Ore
monoid~\(P\).

\begin{definition}
  \label{def:group_completion}
  For an Ore monoid~\(P\), the \emph{group completion}~\(G\) of~\(P\) is the
  set of equivalence classes
  \[
    G\defeq \bigslant{P\times P}{\sim},
  \]
  where \((p_1,p_2)\sim (q_1,q_2)\) if there are \(t_1,t_2\in P\) with
  \(p_1 t_1 =q_1 t_2\) and \(p_2 t_1 = q_2 t_2\).  We denote an element
  of~\(G\) represented by \((p_1,p_2)\) as \(p_1p_2^{-1}\in G\).  The group
  operation in~\(G\) is
  \(p_1p_2^{-1}\cdot q_1q_2^{-1} \defeq (p_1t_1)(q_2t_2)^{-1}\) if
  \(t_1,t_2\in P\) are such that \(p_2t_1=q_1t_2\) (given
  by~\ref{enum:O1}).  The neutral element of the group is
  \(e\defeq 11^{-1}\in G\).
\end{definition}

Details about why this defines a group are checked in
\cite{The_Stacks_project}*{\href{https://stacks.math.columbia.edu/tag/04VB}{Tag
    04VB}}.
The canonical monoid homomorphism \(P\to G,\,p\mapsto p1^{-1}\), need
not be injective.
We still sometimes write “\(p\in G\)” and mean the element
\(p1^{-1}\in G\) for \(p\in P\).
The explicit covariance ring we will construct is naturally
\(G\)\nb-graded.

\begin{definition}[\cite{Albandik-Meyer:Product}*{Definition
    3.14}]
  \label{def:Rg_CPg}
  For \(g\in G\), let
  \[
    R_g \defeq
    \setgiven[\big]{(p_1,p_2)\in P\times P}{p_1p_2^{-1}=g\in G}.
  \]
  Let~\(\CPg\) be the category with~\(R_g\) as its set of objects,
  \(R_g\times P\) as its set of arrows, where \((p_1,p_2,q)\colon
  (p_1,p_2)\to (p_1q, p_2q)\), and with the composition defined by
  \((p_1 q, p_2 q, t)\cdot(p_1, p_2, q) = (p_1, p_2, qt)\)
  for \(p_1,p_2,q,t\in P\).
\end{definition}

If~\(P\) is an Ore monoid, then the category~\(\CPg\) is filtered for each
\(g\in G\) (see \cite{Albandik-Meyer:Product}*{Lemma~3.15}).
We are going to define a diagram of smooth \(F_1\)\nb-bimodules over
this category.
The inductive limits of these diagrams for \(g\in G\) will be the
homogeneous summands of the covariance ring, which is \(G\)\nb-graded
by construction.

Since each~\(F_p\) is an \(F_1\)\nb-bimodule,
\(\Hom_{-,F_1}(F_{p_1},F_{p_2})\) is an \(F_1\)\nb-bimodule in a
natural way with \((a\cdot f\cdot b)(x) \defeq a\cdot (f(b\cdot x))\)
for all \(a,b\in F_1\), \(x\in F_{p_1}\),
\(f\in\Hom_{-,F_1}(F_{p_1},F_{p_2})\).  This bimodule is not smooth,
so we replace it by its largest smooth subbimodule:

\begin{lemma}
  \label{lem:Hom_F_simplify}
  The \(F_1\)\nb-subbimodule \(\Hom_{-,F_1}(F_{p_1}, F_{p_2})F_1\) of
  \(\Hom_{-,F_1}(F_{p_1}, F_{p_2})\) is smooth and isomorphic to
  \(F_{p_2} \otimes_{F_1} F_{p_1}^*\), where
  \(F_{p_1}^* \defeq \Hom_{-,F_1}(F_{p_1},F_1)\cdot F_1\).
\end{lemma}

\begin{proof}
  Since~\(F_{p_1}\) is a proper \(F_1\)\nb-bimodule by assumption,
  Theorem~\ref{thm:proper_is_nice} applies here and gives the asserted
  bimodule isomorphism.  The bimodule
  \(F_{p_2} \otimes_{F_1} F_{p_1}^*\) is smooth on the left
  because~\(F_{p_2}\) is, and smooth on the right by construction.
\end{proof}

Since~\(F_1\) is a ring with local units, the following lemma shows
that \(F_1^* \cong F_1\) and thus \(\Hom_{-,F_1}(F_1, F_p)F_1 \cong
F_p \otimes_{F_1} F_1^* \cong F_p\) for all \(p\in P\):

\begin{lemma}
  \label{lem:endo_smooth}
  For any ring with local units, \(\Endo_{-,D}(D)D \cong D\), embedded
  as left multiplication operators.
\end{lemma}

\begin{proof}
  This follows because \(T(d\cdot x) = T(d)\cdot x\) for all \(d,x\in D\) and
  \(T\in\Endo_{-,D}(D)\).
\end{proof}

\begin{definition}
  \label{def:varphi}
  For \((p_1,p_2)\in R_g\) and \(q\in P\) define the maps
  \begin{align*}
    \varphi_{p_1,p_2,q}\colon \Hom_{-,F_1}(F_{p_2}, F_{p_1})F_1
    &\to \Hom_{-,F_1}(F_{p_2q}, F_{p_1q})F_1,\\
    T &\mapsto  \mu_{p_1,q}\circ (T\otimes_{F_1} \id_{F_q}) \circ \mu_{p_2,q}^{-1}.
  \end{align*}
\end{definition}

\begin{lemma}
  \label{lem:varphig_props}
  The map \(\varphi_{p_1,p_2,q}\) for \((p_1,p_2)\in R_g\), \(q\in P\) is
  an \(F_1\)\nb-bimodule homomorphism and its image lies in
  \(\Hom_{-,F_1}(F_{p_2q},F_{p_1q})F_1\).  These maps satisfy
  \begin{align}
    \label{eq:varphig_unit}
    \varphi_{p_1,p_2,1}
    &=\id_{\Hom_{-,F_1}(F_{p_2}, F_{p_1})F_1},\\
    \label{eq:varphig_mult}
    \varphi_{p_1q,p_2q,t}\circ \varphi_{p_1,p_2,q}
    &=\varphi_{p_1,p_2,qt}.
  \end{align}
  The assignment \((p_1,p_2,q)\mapsto \varphi_{p_1,p_2,q}\) defines a
  filtered diagram of \(F_1\)\nb-bimodules over~\(\CPg\), which we
  denote by~\(\Hg\).
\end{lemma}

\begin{proof}
  The formula for~\(\varphi_{p_1,p_2,q}\) defines an
  \(F_1\)\nb-bimodule homomorphism
  \[
    \Hom_{-,F_1}(F_{p_2}, F_{p_1})
    \to \Hom_{-,F_1}(F_{p_2q}, F_{p_1q})
  \]
  because each~\(\mu_{p_2,q}\) is an \(F_1\)\nb-bimodule homomorphism.
  As a bimodule map, \(\varphi_{p_1,p_2,q}\) maps
  \(\Hom_{-,F_1}(F_{p_2}, F_{p_1}) F_1\) to
  \(\Hom_{-,F_1}(F_{p_2q}, F_{p_1q}) F_1\).
  Equation~\eqref{eq:varphig_unit} follows because~\(\mu_{p_i,1}\) for
  \(i=1,2\) is natural for bimodule maps.
  Equation~\eqref{eq:varphig_mult} follows because the following
  diagram commutes for all \(T\in \Hom_{-,F_1}(F_{p_2},F_{p_1})F_1\):
  \[
    \begin{tikzcd}[ampersand replacement=\&,column sep=2.4em]
      {F_{p_1q}\otimes_{F_1} F_t}
      \& {F_{p_1}\otimes_{F_1} F_q \otimes_{F_1} F_t}
      \& {F_{p_2}\otimes_{F_1} F_q \otimes_{F_1} F_t}
      \& {F_{p_2q}\otimes_{F_1} F_t} \\
      {F_{p_1qt}} \& {F_{p_1}\otimes_{F_1} F_{qt}} \& {F_{p_2}\otimes_{F_1} F_{qt}} \& {F_{p_2qt}}
      \arrow["{{\mu_{p_1q,t}}}", "{\cong}"', from=1-1, to=2-1]
      \arrow["{{\id\otimes \mu_{q,t}}}", "{\cong}"', from=1-2, to=2-2]
      \arrow["{{\mu_{p_1,q}\otimes \id}}"',"{\cong}",  from=1-2, to=1-1]
      \arrow["{{\mu_{p_2,q}\otimes \id}}","{\cong}"',  from=1-3, to=1-4]
      \arrow["{{\id\otimes \mu_{q,t}}}","{\cong}"',  from=1-3, to=2-3]
      \arrow["{T \otimes\id\otimes\id}"', from=1-3, to=1-2]
      \arrow["{{\mu_{p_2q,t}}}","{\cong}"',  from=1-4, to=2-4]
      \arrow["{{\mu_{p_1,qt}}}"',"{\cong}",  from=2-2, to=2-1]
      \arrow["{{\mu_{p_2,qt}}}","{\cong}"',  from=2-3, to=2-4]
      \arrow["{T \otimes\id}"', from=2-3, to=2-2]
    \end{tikzcd}
  \]
  Here the left and right squares commute by
  Definition~\ref{def:diagram_in_rings} and the middle square
  commutes because the tensor product is a bifunctor.
\end{proof}

The category of \(F_1\)\nb-bimodules is cocomplete, so the diagram
built above has a colimit.  Since the category~\(\CPg\) is filtered,
the following construction describes this colimit and the universal
cone explicitly.  Define the set
\[
  \mathcal{O}_{\sqcup, g} \defeq
  \bigsqcup_{(p_1,p_2)\in R_g}  \Hom_{-,F_1}(F_{p_2}, F_{p_1})F_1
\]
and let~\(\sim\) be the equivalence relation
on~\(\mathcal{O}_{\sqcup, g}\) generated by
\[
  \bigl(x,(p_1,p_2)\bigr)\sim
  \bigl(\varphi_{p_1,p_2,q}(x),( p_1q, p_2q)\bigr)
\]
for all \((p_1,p_2)\in R_g\), \(q\in P\) and
\(x\in \Hom_{-,F_1}(F_{p_2}, F_{p_1})F_1\).  Let~\(\OO[g]\) be the set of
equivalence classes with elements denoted as \([x,(p_1,p_2)]\in\OO[g]\).
Define an abelian group structure on~\(\OO[g]\) by
\begin{align*}
  \bigl[x,(p_1,p_2)\bigr]+\bigl[y,(q_1,q_2)\bigr]
  &\defeq  \bigl[\varphi_{p_1,p_2,t}(x)+\varphi_{q_1,q_2,u}(y), (p_1t,p_2t)\bigr]
\end{align*}
for \((p_1,p_2),(q_1,q_2)\in R_g\),
\(x\in\Hom_{-,F_1}(F_{p_2}, F_{p_1})F_1\),
\(y\in \Hom_{-,F_1}(F_{q_2}, F_{q_1})F_1\) and \(t,u\in P\) such that
\(p_1t=q_1 u\) and \(p_2 t = q_2 u\); these exist because
\(p_1p_2^{-1}=g=q_1q_2^{-1}\).
With the obvious left and right multiplication by~\(F_1\), this
makes~\(\OO[g]\) an \(F_1\)\nb-bimodule.
The canonical maps
\[
  \iota_{p_1,p_2}\colon\Hom_{-,F_1}(F_{p_2}, F_{p_1})F_1\to \OO[g]
\]
for all \((p_1,p_2)\in R_g\) are \(F_1\)\nb-bimodule homomorphisms and
form the universal cone under~\(\Hg\).

\begin{definition}
  \label{def:multiOO}
  Let \(p_i,q_i\in P\) for \(i=1,2\) and let \(g=p_1p_2^{-1}\),
  \(h=q_1q_2^{-1}\) in~\(G\).
  There are \(t_i\in P\) with \(p_2t_1=q_1t_2\) by the Ore conditions.
  Hence \(gh=(p_1t_1)(q_2t_2)^{-1}\).
  We define the map
  \begin{align*}
    \w_{g,h}\colon \OO[g]\times\OO[h]
    &\to\OO[gh],\\
    \Bigl(\bigl[x,(p_1,p_2)\bigr],\bigl[y,(q_1,q_2)\bigr]\Bigr)
    &\mapsto
    \bigl[\varphi_{p_1,p_2,t_1}(x)\circ\varphi_{q_1,q_2,t_2}(y),
    (p_1t_1,q_2t_2)\bigr].
  \end{align*}
\end{definition}

\begin{lemma}
  The map~\(\w_{g,h}\) is well-defined and descends to an
  \(F_1\)\nb-bimodule map \(\OO[g]\otimes_{F_1}\OO[h] \to \OO[g h]\).
  \end{lemma}

\begin{proof}
  The definition of~\(w_{g,h}\) uses representatives \((p_1,p_2)\in
  R_g\), \((q_1,q_2)\in R_h\) for \(g,h \in G\) and \(t_1,t_2\in P\)
  such that \(p_2t_1=q_1t_2\).
  First, we check that the value of~\(w_{g,h}\) does not depend on the
  choice of \(t_1, t_2\).
  So take \(t_1,t_2\) as above and let \(u_1,u_2\in P\) also satisfy
  \(p_2 u_1=q_1u_2\).
  Condition~\ref{enum:O1} provides \(x_1,x_2\in P\) with
  \(t_1x_1= u_1x_2\).
  Then \(q_1 t_2x_1=p_2 t_1 x_1=p_2 u_1 x_2=q_1 u_2 x_2\).
  Condition~\ref{enum:O2} gives \(n\in P\) with \(t_2 x_1n = u_2 x_2
  n\).
  Let \(b_1\defeq x_1 n\) and \(b_2\defeq x_2 n\).
  Then \(t_1 b_1= u_1 b_2\) and \(t_2b_1= u_2b_2\).
  To show that~\(\w_{g,h}\) does not depend on \(t_1,t_2\), we must prove that
  \[
    \varphi_{p_1 t_1,q_2 t_2,b_1}
    \bigl(\varphi_{p_1,p_2,t_1}(x)\circ\varphi_{q_1,q_2,t_2}(y)\bigr)
    = \varphi_{p_1 u_1,q_2 u_2,b_2}
    \bigl(\varphi_{p_1,p_2,u_1}(x)\circ\varphi_{q_1,q_2,u_2}(y)\bigr).
  \]
  This  works because the diagram in Figure~\ref{fig:wgh_defined}
  commutes.
  \begin{figure}[htbp]
    \[
      \begin{tikzcd}[ampersand replacement=\&,row sep=2.25em, column sep = 3.25em]
        \& {F_{q_2 t_2 b_1}} \\
        {F_{q_2 t_2}\otimes_{F_1} F_{b_1}} \&\& {F_{q_2 u_2}\otimes_{F_1} F_{b_2}} \\
        {F_{q_2}\otimes_{F_1} F_{t_2}\otimes_{F_1} F_{b_1}} \& {F_{q_2}\otimes_{F_1}F_{t_2b_1}} \& {F_{q_2}\otimes_{F_1} F_{u_2}\otimes_{F_1} F_{b_2}} \\
        {F_{q_1}\otimes_{F_1} F_{t_2}\otimes_{F_1} F_{b_1}} \& {F_{q_1}\otimes_{F_1}F_{t_2b_1}} \& {F_{q_1}\otimes_{F_1} F_{u_2}\otimes_{F_1} F_{b_2}} \\
        {F_{q_1 t_2}\otimes_{F_1} F_{b_1}} \& {F_{q_1 t_2 b_1}}
        \& {F_{q_1 u_2}\otimes_{F_1} F_{b_2}} \\
        {F_{p_2}\otimes_{F_1} F_{t_1}\otimes_{F_1} F_{b_1}} \&
        {F_{p_2}\otimes_{F_1}F_{t_1 b_1}}
        \& {F_{p_2}\otimes_{F_1} F_{u_1}\otimes_{F_1} F_{b_2}} \\
        {F_{p_1}\otimes_{F_1} F_{t_1}\otimes_{F_1} F_{b_1}}
        \& {F_{p_1}\otimes_{F_1}F_{t_1b_1}}
        \& {F_{p_1}\otimes_{F_1} F_{u_1}\otimes_{F_1} F_{b_2}} \\
        {F_{p_1t_1}\otimes_{F_1} F_{b_1}}
        \&\& {F_{p_1 u_1}\otimes_{F_1} F_{b_2}} \\
        \& {F_{p_1t_1b_1}}
        \arrow["{\mu_{q_2 t_2,b_1}}", "{\cong}"', from=2-1, to=1-2]
        \arrow["{{\mu_{q_2,t_2}\otimes\id}}", "{\cong}"', from=3-1, to=2-1]
        \arrow["{{\mu_{q_2,u_2}\otimes\id}}"', "{\cong}", from=3-3, to=2-3]
        \arrow["{{\id\otimes \mu_{t_2,b_1}}}", "{\cong}"', from=3-1, to=3-2]
        \arrow["{\mu_{q_2 u_2,b_2}}"', "{\cong}", from=2-3, to=1-2]
        \arrow["{\mu_{q_2,t_2b_1}}"', "{\cong}", from=3-2, to=1-2]
        \arrow["{\id\otimes \mu_{u_2,b_2}}"', "{\cong}", from=3-3, to=3-2]
        \arrow["y\otimes\id\otimes\id"', from=3-1, to=4-1]
        \arrow["y\otimes\id\otimes\id", from=3-3, to=4-3]
        \arrow["{y\otimes \id}", from=3-2, to=4-2]
        \arrow["{\id\otimes \mu_{u_2,b_2}}"', "{\cong}", from=4-3, to=4-2]
        \arrow["{{\id\otimes \mu_{t_2,b_1}}}", "{\cong}"', from=4-1, to=4-2]
        \arrow["{{\mu_{q_1,t_2}\otimes\id}}"', "{\cong}", from=4-1, to=5-1]
        \arrow["{{\mu_{q_1,u_2}\otimes\id}}", "{\cong}"', from=4-3, to=5-3]
        \arrow["{{\mu_{p_2,t_1}\otimes\id}}", "{\cong}"', from=6-1, to=5-1]
        \arrow["{{\mu_{p_2,u_1}\otimes\id}}"', "{\cong}", from=6-3, to=5-3]
        \arrow["{{\id\otimes \mu_{t_1,b_1}}}", "{\cong}"', from=6-1, to=6-2]
        \arrow["{{\id\otimes \mu_{u_1,b_2}}}"', "{\cong}", from=6-3, to=6-2]
        \arrow["{\mu_{q_1,t_2b_1}}", "{\cong}"', from=4-2, to=5-2]
        \arrow["{\mu_{q_1 t_2,b_1}}", "{\cong}"', from=5-1, to=5-2]
        \arrow["{\mu_{p_2,t_1b_1}}"', "{\cong}", from=6-2, to=5-2]
        \arrow["{\mu_{q_1 u_2,b_2}}"', "{\cong}", from=5-3, to=5-2]
        \arrow["{{\id\otimes \mu_{t_1,b_1}}}", "{\cong}"', from=7-1, to=7-2]
        \arrow["{{\id\otimes \mu_{u_1,b_2}}}"', "{\cong}", from=7-3, to=7-2]
        \arrow["x\otimes\id", from=6-2, to=7-2]
        \arrow["x\otimes\id\otimes\id"', from=6-1, to=7-1]
        \arrow["x\otimes\id\otimes\id", from=6-3, to=7-3]
        \arrow["{{\mu_{p_1,t_1}\otimes\id}}"', "{\cong}", from=7-1, to=8-1]
        \arrow["{{\mu_{p_1,u_1}\otimes\id}}", "{\cong}"', from=7-3, to=8-3]
        \arrow["{\mu_{p_1 t_1,b_1}}"', "{\cong}", from=8-1, to=9-2]
        \arrow["{\mu_{p_1 u_1,b_2}}", "{\cong}"', from=8-3, to=9-2]
        \arrow["{\mu_{p_1,t_1b_1}}", "{\cong}"', from=7-2, to=9-2]
      \end{tikzcd}\]
    \caption{A commuting diagram that is used to prove
      that~\(\w_{g,h}\) is well-defined}
    \label{fig:wgh_defined}
  \end{figure}
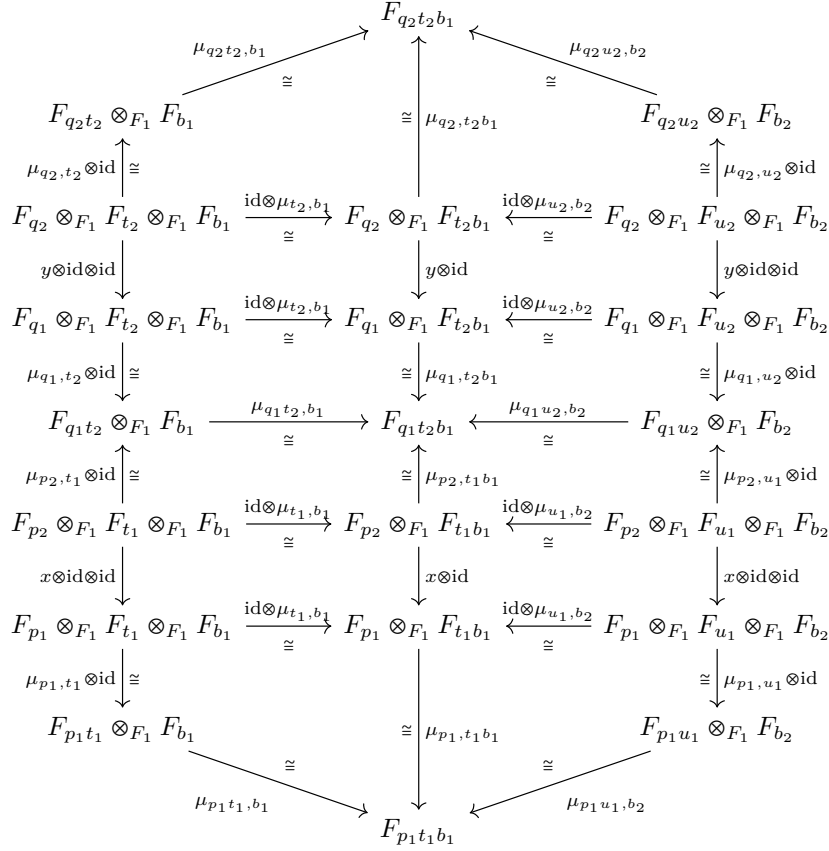

  Next, we check that~\(\w_{g,h}\) is independent of the choices of
  the representatives of \(\bigl[x,(p_1,p_2)\bigr]\) and
  \(\bigl[y,(q_1,q_2)\bigr]\).
  Let \(\bigl[x,(p_1,p_2)\bigr] = \bigl[x', (p_1',p_2')\bigr]\) and
  \(\bigl[y,(q_1,q_2)\bigr] = \bigl[y',(q_1',q_2')\bigr]\).
  Then there are \(n,n'\in P\) with
  \[
    (p_1 n,p_2 n) = (p_1' n', p_2' n')\quad\text{and}\quad
    \varphi_{p_1, p_2,n}(x) =\varphi_{p_1', p_2', n'}(x').
  \]
  Similarly, there are \(m,m'\in P\) with
  \((q_1 m,q_2 m) = (q_1' m', q_2' m')\) and
  \(\varphi_{q_1,q_2,m}(y) = \varphi_{q_1', q_2', m'}(y')\).
  Next, there are \(t_1,t_2\in P\) with \(p_2 n t_1=q_1 m t_2\).
  Then \(p_2'  n' t_1=q_1' m' t_2\).
  Using Lemma~\ref{lem:varphig_props}, this implies
  \begin{align*}
    \varphi_{p_1,p_2, nt_1}(x)
    &\circ\varphi_{q_1,q_2, mt_2}(y)
      = \varphi_{p_1 n,p_2 n, t_1}\bigl(\varphi_{p_1,p_2,n}(x)\bigr)
      \circ \varphi_{q_1 m ,q_2 m ,t_2}\bigl(\varphi_{q_1,q_2,m}(y)\bigr)\\
    &= \varphi_{p_1' n',p_2' n', t_1}\bigl(\varphi_{p_1', p_2',
      n'}(x')\bigr)
      \circ \varphi_{q_1' m' ,q_2' m' ,t_2}\bigl(\varphi_{q_1', q_2', m'}(y')\bigr)\\
    &= \varphi_{p_1',p_2', n' t_1}(x')\circ \varphi_{q_1',q_2', m' t_2}(y').
  \end{align*}
  Hence~\(\w_{g,h}\) is well-defined.  Next we prove that~\(\w_{g,h}\)
  is additive in both variables.  Since~\(\w_{g,h}\) is independent of
  the choice of representatives, it suffices to check this when both
  summands are in the same \(\Hom_{-, F_1}(F_{p_2}, F_{p_1})F_1\).
  Since~\(\varphi_{p_1,p_2,t}\) is additive by
  Lemma~\ref{lem:varphig_props} and composition distributes over
  addition, \(\w_{g,h}\) is additive in each argument.  It is also easy
  to see that \(\w_{g,h}(a x b, y c) = a \w_{g,h}(x, b y) c\) for
  \(a,b,c\in F_1\), \(x\in \OO[g]\), \(y\in \OO[h]\).  This implies
  that~\(\w_{g,h}\) descends to an \(F_1\)\nb-bimodule map
  \(\OO[g]\otimes_{F_1}\OO[h] \to \OO[g h]\).
\end{proof}

\begin{lemma}
  \label{lem:O_is_morph}
  The diagram
  \[
    \begin{tikzcd}[column sep=large]
      {\OO[g]\otimes_{F_1}\OO[h]\otimes_{F_1} \OO[k]}
      \arrow[d, "{\id\otimes_{F_1} \w_{h,k}}"] \arrow[r, "{\w_{g,h}\otimes_{F_1}\id}"] &
      {\OO[gh]\otimes_{F_1} \OO[k]} \arrow[d, "{\w_{gh,k}}"] \\
      {\OO[g]\otimes_{F_1}\OO[hk]} \arrow[r, "{\w_{g,hk}}"]
      & {\OO[ghk]}
    \end{tikzcd}
  \]
  commutes for all \(g,h,k\in G\).  Hence~\(\OO\) is an associative ring,
  the maps \(\w_{e,g}\) and~\(\w_{g,e}\) make~\(\OO[g]\) an
  \(\OO\)\nb-bimodule for all \(g\in G\), and \(\w_{g,h}\) for
  \(g,h\in G\) induces an \(\OO\)-bimodule homomorphism
  \(\w_{g,h}\colon\OO[g]\otimes_{\OO} \OO[h] \to\OO[gh]\).
\end{lemma}

\begin{proof}
  Let \(\tilde{x}\in\OO[g]\), \(\tilde{y}\in\OO[h]\) and
  \( \tilde{z}\in\OO[k]\).
  We may choose representatives of the form
  \(\tilde{x}=\bigl[x,(p_1,p_2)\bigr]\),
  \(\tilde{y}=\bigl[y,(p_2,p_3)\bigr]\) and
  \(\tilde{z}=\bigl[z,(p_3,p_4)\bigr]\) because such triples of pairs
  \((p_1,p_2),(p_2,p_3),(p_3,p_4)\) are cofinal in
  \(\CPg\times\CPg[h]\times\CPg[k]\).
  Then
  \begin{equation}
    \label{eq:w_nice}
    \w_{gh,k}\bigl(\w_{g,h}(\tilde{x} \otimes \tilde{y}) \otimes\tilde{z}\bigr)
    = \w_{g,hk}\bigl(\tilde{x} \otimes\w_{h,k}(\tilde{y} \otimes \tilde{z})\bigr)
  \end{equation}
  because composition of maps is associative.  Equation~\eqref{eq:w_nice} with
  \(g=h = k= e\) says that~\(\OO\) is an associative ring.  Similarly,
  \eqref{eq:w_nice} implies that \(\w_{e,g}\) and~\(\w_{g,e}\)
  make~\(\OO[g]\) an \(\OO\)\nb-bimodule and that the map \(\w_{g,h}\) is
  \(\OO\)-balanced and so descends to an \(\OO\)-bimodule map
  \(\w_{g,h}\colon\OO[g]\otimes_{\OO} \OO[h] \to\OO[gh]\).
  Finally, \eqref{eq:w_nice} says that the diagram in the statement
  commutes.
\end{proof}

\begin{lemma}
  \label{lem:islocunu}
  The ring~\(\OO\) has local units and~\(\OO[g]\) is a smooth
  \(\OO\)\nb-bimodule.
\end{lemma}

\begin{proof}
  As an inductive limit of nondegenerate \(F_1\)\nb-bimodules,
  \(\OO[g]\) is also a nondegenerate \(F_1\)\nb-bimodule.
  The \(\OO\)\nb-bimodule structure on~\(\OO[g]\) extends the
  \(F_1\)\nb-bimodule structure with respect to the canonical map
  \(F_1 \to \OO\).
  Therefore, the idempotents in~\(F_1\) form a local unit in~\(\OO\)
  and~\(\OO[g]\) is nondegenerate as an \(\OO\)\nb-bimodule.
\end{proof}

\begin{definition}
  \label{def:lax_cov_ring_OF}
  Let~\(\OOO_{\F}\) be the \(G\)\nb-graded ring
  \[
    \OOO_{\F}\defeq \bigoplus_{g\in G} \OO[g]
    = \bigoplus_{g\in G} \varinjlim_{(p_1,p_2)\in R_g}
    \Hom_{-,F_1} (F_{p_2}, F_{p_1})F_1
  \]
  with the multiplication
  \[
    a\cdot b\defeq \w_{g,h}(a\otimes b)\in \OO[gh]
  \]
  for \(g,h\in G\), \(a\in\OO[g]\), \(b\in\OO[h]\) and extended
  distributively to~\(\OOO_{\F}\).
\end{definition}

Lemmas \ref{lem:O_is_morph} and~\ref{lem:islocunu} imply
that~\(\OOO_{\F}\) is a ring with local units.
We are going to prove that~\(\OOO_{\F}\) is a covariance ring
for~\(\F\).
A first step towards this is to construct the universal covariant
representation of~\(\F\) in~\(\OOO_{\F}\).
There are canonical \(F_1\)-bimodule homomorphisms
\[
  \tilde\kappa_p\colon
  F_p\congto\Hom_{-,F_1}(F_1,F_p)F_1\to\OO[p]
\]
for all \(p\in P\).
The isomorphism \(F_p\congto\Hom_{-,F_1}(F_1,F_p)F_1\) is explained in
the paragraph after Lemma~\ref{lem:Hom_F_simplify}.
The maps~\(\tilde\kappa_p\) induce \(F_1,\OO\)-bimodule homomorphisms
\[
  \kappa_p\colon
  F_p\otimes_{F_1}\OO\to\OO[p]\otimes_{F_1}\OO\to\OO[p],
  \qquad
  x\otimes y\mapsto \w_{p,e}(\tilde\kappa_p(x) y).
\]

\begin{proposition}
  \label{pro:kappa_iso}
  Let \(p\in P\), \(g\in G\), and identify~\(p\) with its image~\(p
  1^{-1}\) in~\(G\).
  The map
  \[
    \kappa_{p,g}\colon F_p\otimes_{F_1} \OO[g] \to\OO[pg],\qquad
    x\otimes [f,(p_1,p_2)] \mapsto
    \Bigl[\mu_{p,p_1}\bigl(x\otimes f(-)\bigr),(pp_1,p_2)\Bigr],
  \]
  is an isomorphism of \(F_1,\OO\)\nb-bimodules; here
  \(f\in \Hom_{-,F_1}(F_{p_2},F_{p_1}) F_1\).
\end{proposition}

\begin{proof}
  We use Lemma~\ref{lem:Hom_F_simplify} to replace
  \(\Hom_{-,F_1}(F_{p_1}, F_{p_2})F_1\) in the definition
  of~\(\OO[g]\) by \(F_{p_2} \otimes_{F_1} F_{p_1}^*\).  If
  \(p,p_1,p_2\in P\), then there is a chain of isomorphisms
  \[
    \resizebox{\linewidth}{!}{
      \begin{tikzcd}[ampersand replacement=\&,column sep= small, row sep=tiny]
        {F_p\otimes_{F_1} \Hom_{-,F_1}(F_{p_2},F_{p_1})F_1} \& {F_p\otimes_{F_1} F_{p_1}\otimes_{F_1} F_{p_2}^*F_1} \& {F_{pp_1}\otimes_{F_1} F_{p_2}^*F_1} \& {\Hom_{-,F_1}(F_{p_2},F_{pp_1})F_1,} \\
        {x\otimes y\psi(-)} \& {x\otimes y\otimes\psi} \& {\mu_{p,p_1}(x\otimes y)\otimes\psi} \& {\mu_{p,p_1}(x\otimes y)\psi(-)},
        \arrow["\cong"', from=1-2, to=1-1]
        \arrow["\cong", from=1-2, to=1-3]
        \arrow["\cong", from=1-3, to=1-4]
        \arrow[maps to, from=2-2, to=2-3]
        \arrow[maps to, from=2-3, to=2-4]
        \arrow[maps to, from=2-2, to=2-1]
      \end{tikzcd}
    }
  \]
  The composite isomorphisms from left to right for \((p_1,p_2) \in
  R_g\) are compatible with the maps~\(\varphi_{p_1,p_2,t}\) in
  Definition~\ref{def:varphi}.
  That is, they provide an isomorphism of diagrams over~\(\CPg\).
  The inductive system on the left is \(F_p\otimes_{F_1} \Hg\), and
  its inductive limit is \(F_p \otimes_{F_1} \OO[g]\) because the
  tensor product commutes with colimits.
  We claim that the functor \(\CPg \to \CPg[pg]\) that maps
  \((p_1,p_2) \in R_g\) to \((p p_1, p_2)\) is cofinal.
  This implies that the inductive system on the right has the same
  colimit as~\(\Hg[pg]\).
  So the isomorphism of inductive systems above implies the desired
  isomorphism \(F_p \otimes_{F_1} \OO[g] \cong \OO[p g]\).
  It remains to prove the asserted cofinality.

  Let \(q_1,q_2\in P\) be such that \(q_1 q_2^{-1} = p g = p p_1
  p_2^{-1}\).
  The latter means that there are \(t_1,t_2\in P\) with \(p p_1 t_1
  =q_1 t_2\) and \(p_2 t_1 = q_2 t_2\).
  So \((q_1,q_2,t_2)\) is an arrow in~\(\CPg[pg]\) from \((q_1,q_2)\)
  to \((p p_1 t_1, p_2 t_1) = (q_1 t_2, q_2 t_2)\), which is in the
  image of our functor.
  Let \(u_1,u_2\in P\) be another choice with
  \(p p_1 u_1 =q_1 u_2\) and \(p_2 u_1 = q_2 u_2\), so that
  \((q_1,q_2,u_2)\) is another arrow in~\(\CPg[pg]\) from
  \((q_1,q_2)\) to an object
  \((p p_1 u_1, p_2 u_1) = (q_1 u_2, q_2 u_2)\) in the image of our
  functor.  Then~\ref{enum:O1} gives \(v,w\in P\) with
  \(u_2 v = t_2 w\).  Then
  \(p_2 (u_1 v) = q_2 u_2 v = q_2 t_2 w = p_2 (t_1 w)\).
  Then~\ref{enum:O2} gives \(x\in P\) with \(u_1 v x= t_1 w x\).
  Since \(u_2 v x = t_2 w x\) as well, the arrows \((p p_1 t_1,p_2
  t_1,w x)\colon (p p_1 t_1,p_2 t_1) \to (p p_1 t_1 w x, p_2 t_1 w
  x)\) and \((p p_1 u_1,p_2 u_1,v x)\colon (p p_1 u_1,p_2 u_1) \to
  (p p_1 u_1 v x, p_2 u_1 v x)\) in the image of~\(\CPg\) equalise
  the two arrows above.
  This finishes the proof of cofinality.
\end{proof}

\begin{corollary}
  \label{cor:semisat}
  Let \(p\in P\), \(g\in G\) and identify~\(p\) with its image~\(p
  1^{-1}\) in~\(G\).
  The multiplication map~\(\w_{p,g}\) induces an isomorphism
  \(\OO[p]\otimes_{\OO}\OO[g]\cong\OO[pg]\).
\end{corollary}

\begin{proof}
  We use Proposition~\ref{pro:kappa_iso} twice:
  \[
    \OO[p]\otimes_{\OO}\OO[g]
    \cong F_p\otimes_{F_1} \OO\otimes_{\OO}\OO[g]
    \cong F_p\otimes_{F_1} \OO[g]\cong\OO[pg].\qedhere
  \]
\end{proof}

\begin{corollary}
  \label{cor:presemisat}
  Let \(p\in P\).  The bimodule isomorphisms
  \(\kappa_{p,g}\colon F_p\otimes_{F_1}\OO[g]\congto\OO[pg]\) induce an
  \(F_1,\OOO_\F\)-bimodule isomorphism
  \(\kappa_p\colon
  F_p\otimes_{F_1} \OOO_{\mathcal{F}}\congto\OOO_{\mathcal{F}}\) for all
  \(p\in P\).  The maps~\((\tilde\kappa_p)_{p\in P}\) form a covariant
  representation of~\(\F\) in the ring~\(\OOO_\F\).
\end{corollary}

\begin{proof}
  The isomorphisms~\(\kappa_{p,g}\) give the isomorphism~\(\kappa_p\)
  because the tensor product commutes with direct sums and
  \(g\mapsto p g\) is a bijection on the group~\(G\).  We check that
  these isomorphisms form a covariant representation.
  The only nontrivial point is that
  \(\tilde\kappa_p(x)\cdot \tilde\kappa_q(y) =
  \tilde\kappa_{pq}(\mu_{p,q}(x\otimes y))\) for all \(p,q\in P\),
  \(x\in F_p\) and \(y\in F_q\).
  The left hand side here is represented by the composite map
  \[
    F_1 \to F_q \to F_p \otimes_{F_1} F_q \cong F_{p q},
    \qquad
    z\mapsto \mu_{p,q}(x\otimes \mu_{q,1}(y\otimes z)).
  \]
  The coherence conditions of the diagram in
  Definition~\ref{def:diagram_in_rings} imply that this is
  \(\mu_{pq, 1}(\mu_{p,q}(x\otimes y)\otimes z)\).
  This is a representative for \(\tilde\kappa_{pq}(\mu_{p,q}(x\otimes
  y))\) as needed.
\end{proof}

\begin{theorem}
  \label{thm:cov_rep_F=O}
  Let \(\F=(P, F_p, \mu_{p,q})\) be a proper Ore diagram, that is, \(P\)
  is an Ore monoid and~\(\F\) is a normal pseudofunctor
  \(P \to \Ringspr\).
  Then \((\tilde\kappa_p)_{p\in P}\) is the universal covariant
  representation of~\(\mathcal{F}\), so that~\(\OOO_{\F}\) is a
  covariance ring of~\(\F\).
  That is, for any ring~\(D\),
  \[
    \beta_D\colon \URing(\OOO_\mathcal{F},D)\to\CovRep[D]{\F},\qquad
    (\OOO_\mathcal{F}\xrightarrow{f}D)
    \mapsto ((f\circ \tilde\kappa_p)_{p\in P}),
  \]
  is an isomorphism
  \(\URing(\OOO_\mathcal{F},D) \congto\CovRep[D]{\F}\).
\end{theorem}

\begin{proof}
  Let \((\tilde\nu_p)_{p\in P}\) be a covariant representation
  of~\(\F\) in the ring~\(D\).  Then~\(\tilde\nu_p\) for \(p\in P\)
  induces an \(F_1,D\)-bimodule isomorphism
  \(\nu_p\colon F_p\otimes_{F_1}D\congto D\).  In particular, \(D\) is a
  smooth \(F_1,D\)-bimodule.  Let \(g\in G\) and let
  \((p_1,p_2)\in P\times P\) satisfy \(p_1 p_2^{-1} = g\).
  Define a map
  \begin{align*}
    \tilde{\psi}_{p_1,p_2}\colon
    \Hom_{-,F_1}(F_{p_2},F_{p_1})
    &\to\Endo_{-,D}(D),\\
    T&\mapsto\nu_{p_1}\circ(T\otimes\id_D)\circ\nu_{p_2}^{-1}.
  \end{align*}
  As an \(F_1\)\nb-bimodule map, this maps
  \[
    \Hom_{-,F_1}(F_{p_2},F_{p_1})F_1\to \Endo_{-,D}(D)F_1
    \subseteq \Endo_{-,D}(D)D = D,
  \]
  where the last step is Lemma~\ref{lem:endo_smooth}.
  So~\(\tilde{\psi}_{p_1,p_2}\) restricts to an \(F_1\)\nb-bimodule
  map
  \[
    \psi_{p_1,p_2}\colon \Hom_{-,F_1}(F_{p_2},F_{p_1})F_1 \to D.
  \]
  The covariance condition in Definition~\ref{def:cov_rep_of_F} is
  equivalent to
  \[
    \nu_{p_1 p_2}\circ(\mu_{p_1,p_2}\otimes\id_D)
    = \nu_{p_1}\circ(\id_{F_{p_1}}\otimes\nu_{p_2})
  \]
  for all \(p_1,p_2\in P\).
  By a routine computation, this implies that the
  maps~\(\tilde{\psi}_{p_1,p_2}\) form a cone under the filtered
  diagram~\(\Hg\).  This is inherited by the maps~\(\psi_{p_1,p_2}\).
  Hence they induce an \(F_1\)\nb-bimodule map
  \(\varphi_g\colon \OO[g]\to D\) on the inductive limit.  Let
  \(\varphi\colon \OOO_\F\to D\) be the map that restricts
  to~\(\varphi_g\) on the summand~\(\OO[g]\).  An easy computation shows
  that~\(\varphi\) is a ring homomorphism.
  Since~\(D\) is nondegenerate as an \(F_1\)\nb-bimodule and the
  idempotents in~\(F_1\) form a local unit in~\(\OOO_{\F}\), \(D\) is
  nondegenerate as an \(\OOO_{\F}\)\nb-module as well.  That is, the
  homomorphism~\(\varphi\) is nondegenerate.

  It remains to prove that the map that sends the covariant
  representation~\((\tilde\nu_p)_{p\in P}\) to the nondegenerate
  homomorphism~\(\varphi\) is inverse to~\(\beta_D\).
  In one direction, we must compute \(\beta_D(\varphi)\).
  Lemma~\ref{lem:endo_smooth} identifies \(\Endo_{-,D}(D)D\) with the
  subring isomorphic to~\(D\) consisting of the left multiplication
  operators.
  This works also for~\(F_1\), so that \(F_1^* = \Endo_{-,F_1}(F_1)
  F_1 \cong F_1\).
  Then \(F_p \otimes_{F_1} F_1^* \cong F_p\).
  A simple inspection shows that the map~\(\psi_{p,1}\) becomes the
  given map~\(\tilde\nu_p\) after the resulting identification of
  \(\Hom_{-,F_1}(F_1,F_p)F_1\) with~\(F_p\).
  This says that~\(\beta_D\) maps~\(\varphi\) back
  to~\((\tilde\nu_p)_{p\in P}\).
  Conversely, let us start with a nondegenerate homomorphism
  \(f\colon \OOO_\F \to D\).
  This is mapped by~\(\beta_D\) to the covariant representation
  \(\tilde\nu_p = f\circ \tilde\kappa_p\).
  Let \(\varphi\colon \OOO_\F \to D\) be the nondegenerate
  homomorphism associated to~\((\tilde\nu_p)_{p\in P}\).
  We claim that \(f=\varphi\).
  It suffices to check this on an element of~\(\OOO_\F\) that is the
  image of \(\xi\in \Hom_{-,F_1}(F_{p_2},F_{p_1})F_1\) for some
  \(p_1,p_2\in P\).
  Since~\(D\) is a ring with local units, \(f(\xi)=\varphi(\xi)\)
  follows if \(f(\xi)\cdot d=\varphi(\xi)\cdot d\) holds for all
  \(d\in D\).
  We know that \(\nu_{p_2}\colon F_{p_2} \otimes_{F_1} D \to D\) is an
  isomorphism.
  Therefore, it suffices to prove \(f(\xi)\cdot \nu_{p_2}(x\otimes d)
  = \varphi(\xi)\cdot \nu_{p_2}(x\otimes d)\) for all \(x\in
  F_{p_2}\), \(d\in D\).
  We compute
  \begin{multline*}
    \varphi(\xi)(\nu_{p_2}(x\otimes d))
    = \nu_{p_1}(\xi(x)\otimes d)
    = \tilde\nu_{p_1}(\xi(x))\cdot d
    = f(\kappa_{p_1}(\xi(x)))\cdot d
    \\= f(\kappa_{p_1 p_2^{-1}}(\xi)) f(\kappa_{p_2}(x)) \cdot d
    = f(\xi)(\nu_{p_2}(x\otimes d)).
  \end{multline*}
  Thus the two constructions are inverse to each other and
  so~\(\beta_D\) is bijective.
\end{proof}

\begin{example}
  We continue the study of a regular graph, based on Examples
  \ref{ex:graphmod} and~\ref{exa:graph_Steinberg}.
  The resulting covariance ring is the Leavitt path algebra of the
  graph.
  This comes with a canonical \(\Z\)\nb-grading.
  In our theory, \(\Z\) occurs as the group completion of the Ore
  monoid~\(\N\).
  The \(F_1\)\nb-module~\(F_p\) is \(F_p = A_R(E^{\circ p})\) for
  \(p\in\N\).
  It has \(\delta\)\nb-functions of paths~\(\alpha\) of length~\(p\) as a
  basis.
  Proposition~\ref{pro:dual_of_correspondence} shows that~\(F_p^*\)
  has \(\beta^*\) for paths of length~\(p\) as a basis.
  Then \(\Hom_{-,F_1} (F_{p_1}, F_{p_2}) F_1 \cong F_{p_2}
  \otimes_{F_1} F_{p_1}^*\) has a basis consisting of pairs
  \((\alpha,\beta^*)\) with paths \(\alpha\) and~\(\beta\) of length
  \(p_2\) and~\(p_1\), respectively, satisfying \(\s(\alpha) = \s(\beta)
  = \rg^*(\beta^*)\).
  The maps in the inductive system defining \(\OO[n] = \varinjlim
  F_{p_2} \otimes_{F_1} F_{p_1}^*\) send \((\alpha,\beta^*)\) to
  \(\sum_{\rg(e) = \s(\alpha)} (\alpha e,(\beta e)^*)\).
  The inductive limit description of~\(\OO[n]\) says that any element
  in the covariance ring of degree~\(n\) may be written as a linear
  combination of pairs \((\alpha,\beta^*)\), where the lengths are
  related by \(p_2-p_1 = n\), and that this decomposition is unique if
  \(p_2\) and~\(p_1\) are fixed and sufficiently big.
  All this is very familiar from the study of Leavitt path algebras.
\end{example}

\section{Ore diagrams of ample correspondences and groupoid models}
\label{sec:diagrams_in_grcat}

We recall the explicit construction of the groupoid model of a proper
Ore diagram of ample groupoid correspondences
from~\cites{Albandik:Thesis, Meyer:Diagrams_models}.
It realises a particular bicategorical limit.
The construction proceeds by first tightening the diagram and then
building the model for the tight case.

\subsection{The groupoid model of a tight Ore diagram}
\label{sec:tight_model}

Let~\(P\) be an Ore monoid and let \(\X=(P, \Gr, \Bisp_p, \mu_{p,q})\) be
a tight Ore diagram in \(\Grcat\).
We will reduce proper diagrams to this case later.
Let~\(G\) be the group completion of~\(P\) (see
Definition~\ref{def:group_completion}) and let \(g\in G\).
As in the construction of the covariance ring, we shall use the
filtered category~\(\CPg\) in Definition~\ref{def:Rg_CPg}, which has
the object set
\[
  R_g \defeq  \setgiven[\big]{(p_1,p_2)\in P\times P}{p_1p_2^{-1}=g\in G}
\]
and the arrow set \(R_g \times P\).
We are going to define a functor from~\(\CPg\) to the category of
topological spaces \(\Top\).

For \((p_1,p_2)\in R_g\), we define \(\Bisp_{p_1}\circ \Bisp_{p_2}^*\)
as above Lemma~\ref{lem:base_of_comp}, using that~\(\Bisp_{p_1}\) is a
groupoid correspondence.
That is, it is the quotient of \(\Bisp_{p_1}\times_{\s,\Gr^0,\s}
\Bisp_{p_2}\) by the equivalence relation \({(x_1,x_2)\sim (x_1 g,x_2 g)}\)
for all \(g\in \Gr\) with \(\s(x_1)=\s(x_2)=\rg(g)\).
We write elements as \(x_1 x_2^*\in \Bisp_{p_1}\circ
\Bisp_{p_2}^*\), where \(x_j \in \Bisp_{p_j}\).

\begin{definition}
  \label{def:alpha_pp^q}
  For \((p_1,p_2)\in R_g\) and \(q\in P\), we define the map
  \[
    \alpha_{p_1,p_2}^q\colon \Bisp_{p_1}\circ \Bisp_{p_2}^*
    \to \Bisp_{p_1q}\circ \Bisp_{p_2q}^*,\qquad
    x_1 x_2^* \mapsto x_1 z (x_2 z)^*,
  \]
  where \(z\in \Bisp_q\) is an element with \(\s(x_1)=\s(x_2)=\rg(z)\) and
  \(x_1 z\defeq \mu_{p_1,q}(x_1,z)\in\Bisp_{p_1q}\).
\end{definition}

The definition above uses that~\(\Bisp_q\) is tight, that is, the maps
\(\rg\colon \Bisp_q/\Gr\to \Gr^0\) are homeomorphisms.
Therefore, the element \(z\in \Bisp_q\) exists and is unique up to
right multiplication by some \(g\in \Gr\).
This does not affect the class in \(\Bisp_{p_1 q}\circ \Bisp_{p_2 q}^*\)
because \(x_1 z (x_2 z)^* = x_1 z h (x_2 z h)^*\) for all \(h\in\Gr\)
with \(\rg(h) = \s(z)\).
In addition, replacing \((x_1,x_2)\) by~\((x_1 g,x_2 g)\) for \(g\in G\) and
adjusting~\(z\) to \(g^{-1} z\) leaves \((x_1 z) (x_2 z)^*\)
unchanged.
So~\(\alpha_{p_1,p_2}^q\) is well-defined.

\begin{lemma}
  \label{lem:alpha_nice}
  The map~\(\alpha_{p_1,p_2}^q\) is a well-defined local homeomorphism.
  In addition,
  \[
    \alpha_{p_1q,p_2q}^t\circ\alpha_{p_1,p_2}^q
    = \alpha_{p_1,p_2}^{q t}
  \]
  and \(\alpha_{p_1,p_2}^1 = \id_{\Bisp_{p_1}\circ \Bisp_{p_2}^*}\) for
  all \((p_1,p_2)\in R_g\) and \(t,q\in P\).
\end{lemma}

\begin{proof}
  See \cite{Albandik:Thesis}*{Lemma 3.6}.
\end{proof}

As a consequence, the data above defines a functor
\(\HXg\colon\CPg \to\Top\), mapping the object \((p_1,p_2)\) to
\(\Bisp_{p_1}\circ \Bisp_{p_2}^*\) and the arrow \((p_1,p_2,q)\)
to~\(\alpha_{p_1,p_2}^q\).
Since~\(\CPg\) is filtered, the following definition describes the
colimit of this diagram:

\begin{definition}
  Let~\(\GMH_g\) be the set
  \[
    \GMH_g
    \defeq \varinjlim_{(p_1,p_2)\in R_g} \Bisp_{p_1}\circ \Bisp_{p_2}^*
    =  \biggl(\bigsqcup_{(p_1,p_2)\in R_g} \Bisp_{p_1}\circ
    \Bisp_{p_2}^*\biggr) \biggm/{\sim},
  \]
  where the equivalence relation~\(\sim\) is generated by
  \(x_1 x_2^*\sim \alpha_{p_1,p_2}^q\bigl(x_1 x_2^*\bigr)\) for
  all \(x_1 x_2^*\in\Bisp_{p_1}\circ \Bisp_{p_2}^*\) and
  \((p_1,p_2)\in R_g\), \(q\in P\).
  We give~\(\GMH_g\) the quotient topology.
\end{definition}

\begin{lemma}
  \label{lem:lambda_nice}
  The canonical maps
  \begin{align*}
    \lambda_{p_1,p_2}\colon\Bisp_{p_1}\circ \Bisp_{p_2}^* &\to \GMH_g,\text{ and}\\
    \lambda\colon \bigsqcup_{(p_1,p_2)\in R_g} \Bisp_{p_1}\circ \Bisp_{p_2}^* &\to \GMH_g
  \end{align*}
  are local homeomorphisms.
\end{lemma}

\begin{proof}
  The first map is a local homeomorphism by
  \cite{Albandik:Thesis}*{Lemma 3.9}, and this implies that the second
  map is so as well.
\end{proof}

\begin{definition}
  \label{def:gpmod}
  Define the topological groupoid
  \[
    \GMH \defeq  \bigsqcup_{g\in G} \GMH_g
    =  \bigsqcup_{g\in G}\varinjlim_{(p_1,p_2)\in R_g}
    \Bisp_{p_1}\circ \Bisp_{p_2}^*
  \]
  with
  \begin{itemize}
  \item object set \(\GMH^0\defeq \Gr^0\);
  \item range and source maps
    \(\rg(x_1 x_2^*)\defeq \rg(x_1)\) and
    \(\s(x_1 x_2^*)\defeq \rg(x_2) = \s^*(x_2^*)\); and
  \item composition uniquely determined by
    \(x_1 x_2^*\cdot x_2 x_3^*\defeq x_1 x_3^*\).
  \end{itemize}
\end{definition}

\begin{theorem}
  \label{thm:gpmodtight}
  The data~\(\GMH\) above defines an ample topological groupoid.
\end{theorem}

\begin{proof}
  It is shown in \cite{Albandik:Thesis}*{Proposition 3.10} that~\(\GMH\)
  is a locally compact, \'etale, topological groupoid.
  It is ample because~\(\Gr\) is ample and \(\GMH^0=\Gr^0\).
\end{proof}

It is shown in \cite{Meyer:Diagrams_models}*{Theorem 8.18} that
\(\GMH\) is a groupoid model of the diagram~\(\X\).
Here we shall only use this concrete description of the groupoid
model.

\subsection{Tightening a proper diagram}
\label{sec:tighten}

Now let \(\X=(P,\Gr,\Bisp_p,\mu_{p,q})\) be a proper diagram, not
necessarily tight.
We recall, without proofs, the construction in
\cite{Meyer:Diagrams_models}*{\S{8.1}}, which produces a tight Ore
diagram in~\(\Grcat\) from~\(\X\).
The point of this construction is that the new diagram has the same
groupoid model.
For our purposes, we may use this to extend the definition of the
groupoid model to proper diagrams that need not be tight.
The construction in~\cite{Meyer:Diagrams_models} does not require the
diagram~\(\X\) to be proper.
We need this here to ensure that the object space of the groupoid
model is again locally compact, so that it has a Steinberg algebra.

Let \(p,q\in P\).
The coordinate projection \(\Bisp_p \times_{\s,\rg}\Bisp_q \to
\Bisp_p\) descends to a map \(\Bisp_p \circ \Bisp_q \to \Bisp_p/\Gr\).
We let~\(\pi^q_p\) be the composite map
\[
  \pi^q_p\colon \Bisp_{p q}/\Gr
  \xrightarrow[\cong]{\mu_{p,q}^{-1}}(\Bisp_p\circ\Bisp_q)/\Gr
  \to \Bisp_p/\Gr.
\]
The left \(\Gr\)\nb-action on~\(\Bisp_p\) induces a \(\Gr\)\nb-action
on~\(\Bisp_p/\Gr\) because the left and right actions commute.
Let \(\Top^{\Gr}\) be the category of \(\Gr\)\nb-spaces.
Let~\(PO\) be the category with object space~\(P\) and with arrows
\(q\colon p \to p q\) for \(p,q\in P\); this is a subcategory of the
category~\(\CPg[e]\) in Definition~\ref{def:Rg_CPg}.
The maps~\(\pi^q_p\) are \(\Gr\)\nb-equivariant and so \((\Bisp_p/\Gr,
\pi^q_p)\) is a contravariant functor \(PO \to \Top^{\Gr}\).

\begin{definition}
  \label{def:omegaspace}
  Let \(\Omega \defeq \varprojlim(\Bisp_p/\Gr,\pi^q_p)\)
  and let \(\pi_p^\infty\colon \Omega\to \Bisp_p/\Gr\) be the
  canonical map.
  Equip~\(\Omega\) with the induced continuous \(\Gr\)\nb-action.
  Its anchor map is the canonical map \(\pi_1^\infty\colon \Omega\to
  \Bisp_1/\Gr \cong \Gr^0\).
\end{definition}

By definition, \(\Omega\) is the subspace of the product space
\(\prod_{p\in P}\Bisp_p/\Gr\) consisting of all families \((x_p)_{p\in
  P}\) with \(x_{p'}=\pi^q_{p'}(x_{p'q})\) for all \(p',q\in P\).
The map~\(\pi_p^\infty\) maps this family to~\(x_p\).
Since~\(P\) is Ore, the category~\(PO\) is also a filtered category,
so that the diagram above is filtered.
This is used in \cite{Meyer:Diagrams_models}*{Lemma~8.5}
to define certain homeomorphisms
\[
  \mu_{p'}\colon \Bisp_{p'}\circ\Omega\congto\Omega.
\]
The assignment \(p\mapsto (\Bisp_{p'p}/\Gr,\pi^q_{p' p})\) gives
another contravariant functor \(PO\to \Top/\Gr^0\) with the same
projective limit~\(\Omega\).
Thus there is a continuous map
\begin{multline*}
  \Bisp_{p'}\times_{\s_{p'},\Gr^0,\pi_1^\infty} \Omega
  =
  \Bisp_{p'}\times_{\s_{p'},\Gr^0,\pi_1^\infty}\varprojlim(\Bisp_p/\Gr,\pi^q_p)
  \cong
  \varprojlim(\Bisp_{p'}\times_{\s_{p'},\Gr^0,(\rg_p)_*}\Bisp_p/\Gr,\id\times\pi^q_p)
  \\\to\varprojlim(\Bisp_{p'}\circ\Bisp_p/\Gr,\id\circ\pi^q_p)
  \cong\varprojlim(\Bisp_{p'p}/\Gr,\pi^q_{p' p})\cong\Omega.
\end{multline*}

Let \(\Gr\Omega\defeq \Gr\ltimes\Omega\) and \(\Bisp_p \Omega\defeq
\Bisp_p\times_{\s,\Gr^0,\pi_1^\infty}\Omega\) for \(p\in P\).
It is shown in \cites{Albandik:Thesis, Meyer:Diagrams_models}
that~\(\Omega\) is a locally compact, Hausdorff space and that
\(\pi_1^\infty\colon \Omega \to \Gr^0\) is a local homeomorphism.
Since~\(\Gr\) is ample, so is~\(\Omega\).
So~\(\Gr\Omega\) is an ample groupoid.
The following is shown in \cite{Meyer:Diagrams_models}*{\S8.1}:

\begin{proposition}
  \label{prop:tightenediag}
  There are well-defined right and left \(\Gr\Omega\)-actions
  on~\(\Bisp_p \Omega\), which turn it into a groupoid correspondence
  \(\Bisp_p \Omega\colon \Gr\Omega\leftarrow\Gr\Omega\).
  Let \(\s_\Omega,\rg_\Omega\colon \Bisp_p \Omega\to \Omega\) denote
  its anchor maps.
  The maps
  \[
    \Bisp_p \Omega\times_{\s_\Omega,\Omega,\rg_\Omega}
    \Bisp_q \Omega \to \Bisp_{p q} \Omega,
    \qquad
    (x,\omega_1,x_2,\omega_2) \mapsto (\mu_{p,q}(x,x_2), \omega_2),
  \]
  induce isomorphisms of correspondences
  \[
    \mu_{p,q}\Omega\colon \Bisp_p \Omega\circ_{\Gr\Omega}
    \Bisp_q \Omega
    \to \Bisp_{p q} \Omega
  \]
  for all \(p,q\in P\).
  The quadruple \((P,\Gr\Omega,\Bisp_p \Omega,\mu_{p,q} \Omega)\) is a
  diagram of tight ample groupoid correspondences of shape~\(P\).
\end{proposition}

We call \(\X\Omega = (P,\Gr\Omega,\Bisp_p \Omega,\mu_{p,q} \Omega)\)
the \emph{tightening} of \(\X=(P,\Gr,\Bisp_p,\mu_{p,q})\).
It is shown in \cite{Meyer:Diagrams_models}*{Theorem 8.10} that
\(\X\Omega\) and~\(\X\) have the same groupoid model.
We simply define the groupoid model of~\(\X\) as the groupoid model of
the tight diagram~\(\X\Omega\).
By Definition~\ref{def:gpmod}, this is
\[
  \GMH \defeq  \bigsqcup_{g\in G} \GMH_g
  =  \bigsqcup_{g\in G} \varinjlim_{(p_1,p_2)\in R_g}
  \Bisp_{p_1} \Omega\circ_{\Gr\Omega} (\Bisp_{p_2} \Omega)^*.
\]

\section{Steinberg algebras of groupoid models}
\label{sec:steinalg_of_grpmodel}

In this section, we are going to prove that the Steinberg algebra of
the groupoid model of a diagram of proper ample groupoid
correspondences of Ore shape~\(P\) is also a covariance ring of the
associated diagram in \(\Ringspr\).
In particular, the Steinberg algebra pseudofunctor from the bicategory
of proper ample groupoid correspondences to \(\Ringspr\) preserves
limits of Ore shape.

\subsection{More computations with duals}
\label{sec:more_duals}

We have already used the ``dual space''~\(\Bisp^*\) of a groupoid
correspondence \(\Bisp\colon \Gr \leftarrow \Gr\) in
Proposition~\ref{pro:dual_of_correspondence} to describe the dual
bimodule \(A_R(\Bisp)^*\) of \(A_R(\Bisp)\).
In Section~\ref{sec:cov_ring_construction}, we have understood the
structure of the covariance ring~\(\OOO\) using the tensor products
\(F_{p_2} \otimes_{F_1} F_{p_1}^*\) for \(p_1,p_2\in P\).
We now describe these tensor products and the canonical maps between
them when the diagram of rings and bimodules comes from a diagram of
proper groupoid correspondences.
Throughout this section, \(P\) is an Ore monoid and \(\X=(P, \Gr,
\Bisp_p, \mu_{p,q})\) is a diagram of proper, ample groupoid
correspondences of shape~\(P\).
We have constructed a normal pseudofunctor
\(\A\colon\Grcat\to\Rings\) (Theorem~\ref{thm:pseudofunctor}) that
maps a groupoid to its Steinberg algebra and a groupoid correspondence
to a bimodule between the Steinberg algebras.
By Proposition~\ref{prop:Steinpro}, this pseudofunctor maps proper
groupoid correspondences to proper bimodules.
In particular, \(\X\) induces a normal pseudofunctor \(\F = \A
* \X\colon P \to \Ringspr\).
The underlying ring of this diagram is \(F_1\defeq A_R(\Gr)\), the
bimodules involved are \(F_p\defeq A_R(\Bisp_p)\), and the
multiplication maps~\(\mu_{p,q}^F\) are induced by the multiplication
maps~\(\mu_{p,q}\) and the canonical isomorphisms \(A_R(\Bisp_p)
\otimes_{A_R(\Gr)} A_R(\Bisp_q) \cong A_R(\Bisp_{p q})\) built after
Definition~\ref{def:mu}.

Recall that the covariance ring of~\(\F\) is a \(G\)\nb-graded ring
\(\OOO_{\F} = \bigoplus_{g\in G} \OO[g]\) with
\[
  \OO[g] \cong \varinjlim_{(p_1,p_2)\in R_g}
  \Hom_{-, A_R(\Gr)}\bigl( A_R(\Bisp_{p_2}), A_R(\Bisp_{p_1})\bigr)
  A_R(\Gr).
\]
We are going to describe~\(\OO[g]\) using the spaces \(\Bisp_{p_1} \circ
\Bisp_{p_2}^*\).

\begin{proposition}
  \label{pro:I_charmap}
  There are canonical \(A_R(\Gr)\)-bimodule isomorphisms
  \begin{align*}
    A_R(\Bisp_{p_1}\circ \Bisp_{p_2}^*)
    &\cong A_R(\Bisp_{p_1}) \otimes_{A_R(\Gr)} A_R(\Bisp_{p_2}^*)
    \\&\cong A_R(\Bisp_{p_1}) \otimes_{A_R(\Gr)}
    \Hom_{-,A_R(\Gr)}
    \bigl( A_R(\Bisp_{p_2}), A_R(\Gr)\bigr) A_R(\Gr)
    \\&\cong \Hom_{-,A_R(\Gr)}
    \bigl( A_R(\Bisp_{p_2}), A_R(\Bisp_{p_1})\bigr) A_R(\Gr)
  \end{align*}
  for all \(p_1,p_2\in P\).  We denote the composite isomorphism by
  \[
    \mathcal{I}_{p_1,p_2}\colon
    A_R(\Bisp_{p_1}\circ \Bisp_{p_2}^*) \congto
    \Hom_{-,A_R(\Gr)}
    \bigl( A_R(\Bisp_{p_2}), A_R(\Bisp_{p_1})\bigr) A_R(\Gr).
  \]
\end{proposition}

\begin{proof}
  The isomorphisms in the beginning of the statement follow from
  Theorem~\ref{the:mu_invertible},
  Proposition~\ref{pro:dual_of_correspondence},
  and Theorem~\ref{thm:proper_is_nice}, respectively.
  Here we use that \(A_R(\Bisp_{p_1})\) is a proper bimodule
  because~\(\Bisp_{p_1}\) is proper.
\end{proof}

We want to compute the map~\(\mathcal{I}_{p_1,p_2}\) explicitly.
This requires an ample base for the space \(\Bisp_{p_1}\circ
\Bisp_{p_2}^*\).
We choose this base in a way that also facilitates several other
computations later on.

Let \(p_1,p_2\in P\).
An element of \(\Bisp_{p_1} \circ \Bisp_{p_2}^*\) is written as \(x_1
x_2^*\), where \(x_j\in\Bisp_{p_j}\) for \(j=1,2\) satisfy \(\s(x_1) =
\s(x_2)\).
Note that \(x_1 x_2^* = (x_1 g)(x_2 g)^*\) if \(g\in\Gr\) satisfies
\(\rg(g) = \s(x_1) = \s(x_2)\).
So we sometimes have to check in computations whether formulas remain
the same if we replace \((x_1,x_2)\) by \((x_1 g,x_2 g)\).

\begin{lemma}
  \label{lem:new_Bipp_base}
  There are well-defined local homeomorphisms \(\Qu_j\colon
  \Bisp_{p_1} \circ \Bisp_{p_2}^* \to \Bisp_{p_j}/\Gr\), \(x_1 x_2^*
  \mapsto \Qu(x_j)\), for \(j=1,2\).
  Both are injective on \(U_1 U_2^*\) if \(U_k\subseteq \Bisp_{p_k}\)
  for \(k=1,2\) are slices.
\end{lemma}

\begin{proof}
  Let \(x_k,y_k\in\Bisp_{p_k}\) for \(k=1,2\) satisfy
  \(\s(x_1) = \s(x_2)\) and \(\s(y_1) = \s(y_2)\), so that \(x_1
  x_2^*, y_1 y_2^* \in \Bisp_{p_1} \circ \Bisp_{p_2}^*\).
  By definition, \(x_1 x_2^* = y_1 y_2^*\) if and only if there is
  \(g\in \Gr\) with \(\rg(g) = \s(x_1) = \s(x_2)\), \(x_1 g = y_1\)
  and \(x_2 g = y_2\).
  This implies \([x_1]=[y_1]\) in \(\Bisp_{p_1}/\Gr\) and
  \([x_2]=[y_2]\) in \(\Bisp_{p_2}/\Gr\).
  So the maps \(\Qu_1\) and~\(\Qu_2\) are well-defined.
  They are local homeomorphisms because the orbit space projections
  \(\Bisp_{p_1} \times_{\s,\s} \Bisp_{p_2} \to \Bisp_{p_1} \circ
  \Bisp_{p_2}^*\) and \(\Bisp_{p_j} \to \Bisp_{p_j}/\Gr\) and the
  coordinate projections \(\Bisp_{p_1} \times_{\s,\s} \Bisp_{p_2} \to
  \Bisp_{p_j}\) are all local homeomorphisms.
  Now assume \(x_k,y_k \in U_k\) for \(k=1,2\).
  Fix \(j\in \{1,2\}\) and assume that \(\Qu_j(x_1 x_2^*) =
  \Qu_j(y_1 y_2^*)\).
  This means \(\Qu(x_j) = \Qu(y_j)\) and implies \(x_j = y_j\)
  because~\(U_j\) is a slice.
  Then \(\s(x_{3-j}) = \s(x_j) = \s(y_j) = \s(y_{3-j})\).
  This implies \(x_{3-j} = y_{3-j}\) because~\(U_{3-j}\) is a slice.
  Thus~\(\Qu_j\) is injective on~\(U_1U_2^*\).
\end{proof}

\begin{definition}
  \label{def:new_Bipp_base}
  An open subset \(W\subseteq \Bisp_{p_1} \circ \Bisp_{p_2}^*\) is
  called a \emph{slice} if \(\Qu_1|_W\) and~\(\Qu_2|_W\) are injective
  on~\(W\).
\end{definition}

\begin{proposition}
  \label{pro:new_base_Bipp}
  The compact slices in \(\Bisp_{p_1} \circ \Bisp_{p_2}^*\) form an
  ample base that is closed under taking arbitrary compact open subsets.
\end{proposition}

\begin{proof}
  It is clear that any compact open subset of a slice in \(\Bisp_{p_1}
  \circ \Bisp_{p_2}^*\) is again a slice.
  Lemma~\ref{lem:base_of_comp}. implies that the subsets~\(U_1
  U_2^*\) with compact open slices \(U_j\subseteq \Bisp_{p_j}\) for
  \(j=1,2\) form an ample base.
  Since~\(U_1 U_2^*\)   is a slice by Lemma~\ref{lem:new_Bipp_base},
  the compact open slices in \(\Bisp_{p_1} \circ \Bisp_{p_2}^*\) form
  an ample base as well.
\end{proof}

\begin{proposition}
  \label{pro:compute_I}
  Let \(U\subseteq \Bisp_{p_1} \circ \Bisp_{p_2}^*\) and \(W\subseteq
  \Bisp_{p_2}\) be slices.
  The map
  \begin{equation}
    \label{eq:compute_I_map}
    U  \times_{\Qu_2,\Bisp_{p_2}/\Gr,\Qu} W \to \Bisp_{p_1},\qquad
    (u_1 u_2^*,w) \mapsto u_1 \braket{u_2}{w},
  \end{equation}
  is a homeomorphism onto a compact open slice in~\(\Bisp_{p_1}\),
  which we denote by~\(U(W)\).
  The map \(\mathcal{I}_{p_1,p_2} (\charmap{U}) \colon
  A_R(\Bisp_{p_2})\to A_R(\Bisp_{p_1})\) in
  Proposition~\textup{\ref{pro:I_charmap}} maps \(\charmap{W}\)
  to~\(\charmap{U(W)}\).
\end{proposition}

\begin{proof}
  Let \(u_1\in\Bisp_{p_1}\) and \(u_2,w\in \Bisp_{p_2}\).
  Then \(u_1 \braket{u_2}{w}\) is defined if and only if \(\Qu(u_2) =
  \Qu(w)\) in \(\Bisp_{p_2}/\Gr\) and \(\s(u_1) = \s(u_2)\) because
  \(\rg(\braket{u_2}{w}) = \s(u_2)\).
  The element \(u_1 \braket{u_2}{w}\) only depends on \(u_1 u_2^* \in
  \Bisp_{p_1}\circ\Bisp_{p_2}^*\) because \(\braket{u_2 g}{w} = g^{-1}
  \braket{u_2}{w}\) for all \(g\in\Gr\) with \(\rg(g) = \s(u_2)\).
  Thus the map in the statement is well-defined, and its
  image~\(U(W)\) consists of all elements that may be written as \(u_1
  \braket{u_2}{w}\) with \(u_1 u_2^*\in U\) and \(w\in W\).
  We claim that the map is also injective.
  Let \(w' = u_1 \braket{u_2}{w}\) for some \(u_1 u_2^*\in U\) and
  \(w\in W\) with \(\Qu(u_2) = \Qu(w)\).
  Then \(\braket{u_2}{w} = \braket{u_1}{w'}\) and \(\Qu(w')=
  \Qu(u_1) = \Qu_1(u_1 u_2^*)\).
  The latter determines \(u_1 u_2^* \in U\) because~\(U\) is a
  slice, and then \(w = u_2 \braket{u_2}{w} = u_2 \braket{u_1}{w'}\)
  is also determined.
  Thus  the map in the statement is injective.
  In addition, \(\s(u_1 \braket{u_2}{w}) = \s(w)\) or \(\Qu(u_1
  \braket{u_2}{w}) = \Qu(u_1) = \Qu_1(u_1 u_2^*)\) determine \(w\)
  or~\(u_1 u_2^*\), respectively, because \(W\) and~\(U\) are
  slices.
  Since \(\Pi|_W\) and \(\Pi_2|_U\) are injective as well, the
  condition \(\Pi_2(u_1 u_2^*) = \Pi(w)\) then determines both \(w\)
  and~\(u_1 u_2^*\).
  As a consequence, \(\s|_{U(W)}\) and \(\Qu|_{U(W)}\) are injective.

  The domain \(U \times_{\Qu_2,\Bisp_{p_2}/\Gr,\Qu} W\) of the map
  in~\eqref{eq:compute_I_map} is a Hausdorff, open subset of
  \((\Bisp_{p_1} \circ \Bisp_{p_2}^*) \times_{\Qu_2,\Bisp_{p_2}/\Gr,\Qu}
  \Bisp_{p_2}\) because \(U\) and~\(W\) are open and Hausdorff.
  It is compact because \(U\times W\) is compact and the equality in
  the Hausdorff space \(\Bisp_{p_2}/\Gr\) defines a closed subset.
  Thus \(U \times_{\Qu_2,\Bisp_{p_2}/\Gr,\Qu} W\) is also compact.
  The map \(\Bisp_{p_2} \times_{\Qu,\Bisp_{p_2}/\Gr,\Qu} \Bisp_{p_2} \to
  \Gr\), \((u_2,w)\mapsto \braket{u_2}{w}\), is a local homeomorphism
  by
  \cite{Antunes-Ko-Meyer:Groupoid_correspondences}*{Proposition~3.5}
  because the right action on~\(\Bisp_{p_2}\) is basic.
  The multiplication \(\Bisp_{p_1}\times_{\s,\Gr^0,\rg} \Gr \to
  \Bisp_{p_1}\) is a local homeomorphism as well
  by \cite{Antunes-Ko-Meyer:Groupoid_correspondences}*{Lemma~2.9}.
  Therefore, the map \(\Bisp_{p_1} \times_{\s,\Gr^0,\s} \Bisp_{p_2}
  \times_{\Qu,\Bisp_{p_2}/\Gr,\Qu} \Bisp_{p_2} \to \Bisp_{p_1}\),
  \((u_1,u_2,w)\mapsto u_1 \braket{u_2}{w}\), is a local
  homeomorphism.
  The induced map on the orbit space of the \(\Gr\)\nb-action
  \((u_1,u_2,w) \cdot g = (u_1 g,u_2 g,w)\) is still a local
  homeomorphism because orbit space projections are surjective local
  homeomorphisms by
  \cite{Antunes-Ko-Meyer:Groupoid_correspondences}*{Lemma~2.10}.
  Since its restriction to \(U  \times_{\Qu_2,\Bisp_{p_2}/\Gr,\Qu} W\)
  is also injective, this restriction is a homeomorphism onto its
  image \(U(W)\), and the latter is a compact open subset
  in~\(\Bisp_{p_1}\).
  Since \(\Qu\) and~\(\s\) are injective on~\(U(W)\), it is a slice.  

  Now we prove that \(\mathcal{I}_{p_1,p_2} (\charmap{U})(\charmap{W})
  = \charmap{U(W)}\) for compact slices \(U\subseteq \Bisp_{p_1} \circ
  \Bisp_{p_2}^*\) and \(W\subseteq \Bisp_{p_2}\).
  First, we assume that \(U = U_1 U_2^*\) for compact slices \(U_j
  \subseteq \Bisp_{p_j}\) for \(j=1,2\).
  The map~\(\mathcal{I}_{p_1,p_2}\) is defined as the composite of
  three isomorphisms
  \begin{align*}
    A_R(\Bisp_{p_1}\circ \Bisp_{p_2}^*)
    &\congto A_R(\Bisp_{p_1}) \otimes_{A_R(\Gr)} A_R(\Bisp_{p_2}^*)
    \\&\congto A_R(\Bisp_{p_1}) \otimes_{A_R(\Gr)}
    \Hom_{-,A_R(\Gr)}
    \bigl( A_R(\Bisp_{p_2}), A_R(\Gr)\bigr) A_R(\Gr)
    \\&\congto \Hom_{-,A_R(\Gr)}
    \bigl( A_R(\Bisp_{p_2}), A_R(\Bisp_{p_1})\bigr) A_R(\Gr).
  \end{align*}
  The formula in Theorem~\ref{the:mu_invertible} shows that the first
  isomorphism maps~\(\charmap{U_1 U_2^*}\) to \(\charmap{U_1} \otimes
  \charmap{U_2^*}\).
  The formula in Proposition~\ref{pro:dual_of_correspondence} shows
  that the second isomorphism maps this on to \(\charmap{U_1} \otimes
  \mathcal{I}_{\Bisp_{p_2}}(\charmap{U_2^*})\), where
  \(\mathcal{I}_{\Bisp_{p_2}}(\charmap{U_2^*})\) is the map
  \(A_R(\Bisp_{p_2}) \to A_R(\Gr)\) that is uniquely determined by
  mapping~\(\charmap{W}\) to~\(\charmap{\braket{U_2}{W}}\).
  The third isomorphism gives the \(R\)\nb-module map that
  sends~\(\charmap{W}\) to~\(\charmap{U_1} \cdot
  \charmap{\braket{U_2}{W}}\).
  This is equal to \(\charmap{U_1\braket{U_2}{W}}\) by
  \cite{Antunes-Ko-Meyer:Groupoid_correspondences}*{Lemma~7.4}.
  If \(U = U_1 U_2^*\), then \(U(W) = U_1 \braket{U_2}{W}\).
  Thus \(\mathcal{I}_{p_1,p_2} (\charmap{U})(\charmap{W}) =
  \charmap{U(W)}\) is true if \(U = U_1 U_2^*\) with compact slices
  \(U_j\subseteq \Bisp_{p_j}\) for \(j=1,2\).

  Any compact slice \(U\subseteq \Bisp_{p_1}\circ \Bisp_{p_2}^*\) is
  compact and Hausdorff.
  The slices of the form \(U_1 U_2^*\) already form an ample base by
  Lemma~\ref{lem:base_of_comp}.
  Therefore, Proposition~\ref{pro:ample_base_unions} allows us to
  write~\(U\) as a disjoint union of sets of the form \(U_{1,k}
  U_{2,k}^*\) with compact slices \(U_{j,k}\subseteq \Bisp_{p_j}\) for
  \(j=1,2\), \(k=1,\dotsc,\ell\).
  Since the map in~\eqref{eq:compute_I_map} is a homeomorphism
  onto~\(U(W)\), the sets \((U_{1,k} U_{2,k}^*)(W) = U_{1,k}
  \braket{U_{2,k}}{W}\) for \(k=1,\dotsc,\ell\) are all compact, open
  and disjoint.
  Since~\(\mathcal{I}_{p_1,p_2}\) is additive, it follows that
  \[
    \mathcal{I}_{p_1,p_2}(\charmap{U})(\charmap{W})
    = \sum_{k=1}^\ell  \mathcal{I}_{p_1,p_2}
    (\charmap{U_{1,k} U_{2,k}^*})(\charmap{W})
    = \sum_{k=1}^\ell  \charmap{U_{1,k} \braket{U_{2,k}}{W}}
    = \charmap{U(W)}.\qedhere
  \]
\end{proof}

The description of~\(\OO[g]\) as an inductive limit involves the
structure maps
\[
  \varphi_{p_1,p_2,q}\colon
  \Hom_{-, A_R(\Gr)}\bigl( A_R(\Bisp_{p_2}), A_R(\Bisp_{p_1})\bigr)
  \to  \Hom_{-, A_R(\Gr)}\bigl( A_R(\Bisp_{p_2 q}),
  A_R(\Bisp_{p_1 q})\bigr)
\]
in Definition~\ref{def:varphi}.
Proposition~\ref{pro:I_charmap} describes \(A_R(\Gr)\)-bimodule
isomorphisms
\[
  \mathcal{I}_{p_1,p_2}\colon
  A_R(\Bisp_{p_1}\circ \Bisp_{p_2}^*)\cong
  \Hom_{-, A_R(\Gr)}\bigl( A_R(\Bisp_{p_2}), A_R(\Bisp_{p_1})\bigr) A_R(\Gr)
\]
for all \(p_1,p_2\in P\).
The maps~\(\varphi_{p_1,p_2,q}\) induce maps
\[
  \Phi_{p_1,p_2,q} \defeq
  \mathcal{I}_{p_1 q,p_2 q}^{-1} \circ \varphi_{p_1,p_2,q} \circ
  \mathcal{I}_{p_1,p_2} \colon
  A_R(\Bisp_{p_1}\circ \Bisp_{p_2}^*) \to  A_R(\Bisp_{p_1 q}\circ \Bisp_{p_2 q}^*).
\]
So \(\OO[g] \cong \varinjlim A_R(\Bisp_{p_1}\circ \Bisp_{p_2}^*)\),
where the inductive system is indexed by the filtered
category~\(\CPg\) introduced in Definition~\ref{def:Rg_CPg} and the
structure maps in the inductive system are~\(\Phi_{p_1,p_2,q}\).
We write elements of \(\Bisp_{p_1 q}\circ \Bisp_{p_2 q}^*\) as words
\((x_1 z_1) (x_2 z_2)^*\) with \(x_1\in\Bisp_{p_1}\),
\(x_2\in\Bisp_{p_2}\) and \(z_1,z_2\in\Bisp_q\) such that \(\s(x_1) =
\rg(z_1)\), \(\s(z_1) = \s(z_2)\), and \(\s(x_2)=\rg(z_2)\).
In formulas, we choose representatives as above, although they are not
unique.
If \(g_1,g_2,h\in\Gr\) are such that \(\rg(g_1) = \s(x_1)\),
\(\rg(g_2) = \s(x_2)\), and \(\rg(h) = \s(z_1) = \s(z_2)\), then
\[
  (x_1 z_1) (x_2 z_2)^* \sim
  (x_1 g_1 \, g_1^{-1} z_1 h) (x_2 g_2\, g_2^{-1} z_2 h)^*,
\]
that is, the quadruple \((x_1 g_1,g_1^{-1} z_1 h, x_2 g_2, g_2^{-1}
z_2 h)\) represents the same element as \((x_1,z_1,x_2,z_2)\).
Conversely, two quadruples describe the same element of \(\Bisp_{p_1
  q}\circ \Bisp_{p_2 q}^*\) if and only if they are related in this way.

Theorem~\ref{the:steinmod_is_quot_of_sum} and
Proposition~\ref{pro:new_base_Bipp} show that \(A_R(\Bisp_{p_1}\circ
\Bisp_{p_2}^*)\) is generated by the characteristic functions of
compact open slices.
Therefore, the following lemma determines~\(\Phi_{p_1,p_2,q}\):

\begin{lemma}
  \label{lem:compute_Phi}
  Let \(p_1,p_2,q\in P\) and let \(U\subseteq \Bisp_{p_1} \circ
  \Bisp_{p_2}^*\) be a compact slice.
  Then~\(\Phi_{p_1,p_2,q}\) maps~\(\charmap{U}\) to the
  characteristic function of the subset
  \begin{multline*}
    \tilde\alpha_{p_1,p_2}^q(U)
    \defeq
    \bigl\{ (u_1 z) (u_2 z)^* \in (\Bisp_{p_1} \circ \Bisp_q) \circ
    (\Bisp_{p_2} \circ \Bisp_q)^* :{}\\
    u_1\in \Bisp_{p_1},\ u_2\in \Bisp_{p_2},\  z\in \Bisp_q,\ \s(x) =
    \s(y) = \rg(z),\
    u_1 u_2^* \in U\bigr\}.
  \end{multline*}
  The subset \(\tilde\alpha_{p_1,p_2}^q(U)\subseteq \Bisp_{p_1 q} \circ
  \Bisp_{p_2 q}^*\) is a compact slice as well.
\end{lemma}

\begin{proof}
  Let \(\Qu_{p_1 p_2^*}\colon \Bisp_{p_1}\times_{\s,\Gr^0,\s} \Bisp_{p_2}^* \to
  \Bisp_{p_1}\circ \Bisp_{p_2}^*\) be the canonical quotient map and
  define \(\Qu_{p_1 p_2^*, q}\colon \Bisp_{p_1}\times_{\s,\Gr^0,\s}
  \Bisp_{p_2}^* \times_{\s,\Gr^0,\rg} \Bisp_q \to \Bisp_{p_1 q}\circ
  \Bisp_{p_2 q}^*\), \((u_1,u_2,z) \mapsto (u_1 z)(u_2 z)^*\).
  The map~\(\Qu_{p_1 p_2^*}\) is an orbit space projection for a basic
  groupoid action, so that it is a local homeomorphism.
  We claim that the second map is a local homeomorphism as well.
  This follows because for any slices \(U_1\subseteq \Bisp_{p_1}\),
  \(U_2 \subseteq \Bisp_{p_2}\), \(X\subseteq \Bisp_q\), the restriction
  of~\(\Qu_{p_1 p_2^*,q}\) to \(U_1 \times_{\s,\s} U_2 \times_{\s,\rg}
  X\) is a homeomorphism onto the slice \((U_1 X) (U_2 X)^*\).

  The set \(\Qu_{p_1 p_2^*}^{-1}(U)\) is open because~\(\Qu_{p_1
    p_2^*}\) is continuous.
  Then \(\Qu_{p_1 p_2^*}^{-1}(U) \times_{\s,\Gr^0,\rg}
  \Bisp_q\) is an open subset of \(\Bisp_{p_1}\times_{\s,\Gr^0,\s} \Bisp_{p_2}^*
  \times_{\s,\Gr^0,\rg} \Bisp_q\).
  The set~\(\tilde\alpha_{p_1,p_2}^q(U)\) is its image under the local
  homeomorphism~\(\Qu_{p_1 p_2^*,q}\), so it is also open.

  Next, we show that~\(\tilde\alpha_{p_1,p_2}^q(U)\) is a slice.
  We prove that~\(\Qu_1\) is injective on it.
  The proof for~\(\Qu_2\) is the same.
  So let \(x_1^k\in \Bisp_{p_1}\), \(x_2^k\in \Bisp_{p_2}\) and \(z^k\in
  \Bisp_q\) satisfy \(\s(x^k) = \s(y^k) = \rg(z^k)\) and \(x_1^k (x_2^k)^* \in
  U\) for \(k=1,2\).
  We assume that \(x_1^1 z^1 = x_1^2 z^2\) in \(\Bisp_{p_1 q}\).
  This means that there are \(g,h\in\Gr\) with \(\rg(g) = \s(x_1^1)\)
  and \(\rg(h) = \s(z^1)\) such that \(x_1^2 = x_1^1 g\) and \(z^2 =
  g^{-1} z^1 h\).
  This implies \(\Qu_1(x_1^1 (x_2^1)^*)\) = \(\Qu_1(x_1^2
  (x_2^2)^*)\).
  Since~\(U\) is a slice and \(x_1^1 (x_2^1)^*, x_1^2 (x_2^2)^*
  \in U\), it follows that \(x_2^2 = x_2^1 g\).
  So
  \[
    (x_1^2 z^2) (x_2^2 z^2)^*
    = (x_1^1 g\, g^{-1} z^1 h) (x_2^1 g\, g^{-1} z^1 h)^*
    = (x_1^1 z^1) (x_2^1 z^1)^*.
  \]

  The proof above also shows that there is a well-defined map
  \(f\colon \tilde\alpha_{p_1,p_2}^q(U) \to U\), \((x_1 z)(x_2 z)^*
  \mapsto x_1 x_2^*\).
  The fibre of this map at \(x_1 x_2^* \in U\) is homeomorphic to
  the set of \(z\in\Bisp_q\) with \(\rg(z) = \s(x_1) = \s(x_2)\).
  This set is compact because \(\rg_*\colon \Bisp_q/\Gr \to \Gr^0\) is
  proper.
  Therefore, \(f\) is proper and so~\(\tilde\alpha_{p_1,p_2}^q(U)\) is
  compact.
  This finishes the proof that~\(\tilde\alpha_{p_1,p_2}^q(U)\) is a compact
  slice.

  Next, we claim that \(U(W) X =  \tilde\alpha_{p_1,p_2}^q(U)(W X)\)
  holds if \(W\subseteq \Bisp_{p_2}\) and \(X\subseteq \Bisp_q\) are
  slices.
  The first set consists of all \(u_1 \braket{u_2}{w} x\) for \(u_1
  u_2^* \in U\), \(w\in W\) and \(x\in X\) with \(\Qu(u_2) = \Qu(w)\),
  \(\s(w) = \rg(x)\).
  The second set consists of all \(u_1 z \braket{u_2 z}{w x}\)
  for \(u_1 u_2^* \in U\), \(z\in\Bisp_q\), \(w\in W\) and \(x\in X\)
  with \(\s(u_1) = \s(u_2) = \rg(z)\), \(\s(w) = \rg(x)\), and
  \(\Qu(u_2 z) = \Qu(w x)\).
  The last equation says that there are \(g,h\in\Gr\) with \(u_2 g =
  w\) and \(g^{-1} z h = x\).
  In particular, \(\Qu(u_2) = \Qu(w)\), and \(g= \braket{u_2}{w}\)
  is uniquely determined.
  Then \(\braket{u_2 z}{w x} = h\) and so \(u_1 z
  \braket{u_2 z}{w x} = u_1 z h = u_1 \braket{u_2}{w} x\).
  This computation shows that indeed \(U(W) X =
  \tilde\alpha_{p_1,p_2}^q(U)(W X)\).

  Proposition~\ref{pro:I_charmap} computes \(\mathcal{I}_{p_1,p_2}
  (\charmap{\tilde\alpha_{p_1,p_2}^q(U)})\): it maps~\(\charmap{V}\) for a
  compact slice \(V\subseteq \Bisp_{p_2 q}\) to
  \(\charmap{\tilde\alpha_{p_1,p_2}^q(U)(V)}\).
  Similarly, \(\mathcal{I}_{p_1,p_2} (\charmap{U})\)
  maps~\(\charmap{W}\) for \(W\subseteq \Bisp_{p_2}\) to
  \(\charmap{U(W)}\).
  The slices of the form~\(W X\) for slices \(W\subseteq \Bisp_{p_2}\)
  and \(X\subseteq \Bisp_q\) form an ample base for~\(\Bisp_{p_2 q}\)
  by Lemma~\ref{lem:base_of_comp}.
  Theorem~\ref{the:mu_invertible} implies that \(\charmap{W X} =
  \mu_{\Bisp_{p_2},\Bisp_q}(\charmap{W} \otimes \charmap{X})\) for the
  canonical isomorphism \(\mu_{\Bisp_{p_2},\Bisp_q}\colon
  A_R(\Bisp_{p_2}) \otimes_{A_R(\Gr)} A_R(\Bisp_q) \congto
  A_R(\Bisp_{p_2 q})\).
  Thus the operator \(\varphi_{p_1,p_2,q}(\mathcal{I}_{p_1,p_2}
  (\charmap{U}))\colon A_R(\Bisp_{p_2 q}) \to A_R(\Bisp_{p_1 q})\) maps
  \(\charmap{W X}\) to
  \begin{align*}
    \varphi_{p_1,p_2,q}(\mathcal{I}_{p_1,p_2} (\charmap{U}))
    (\charmap{W X})
    &= \mu_{\Bisp_{p_1},\Bisp_q} (\mathcal{I}_{p_1,p_2}
    (\charmap{U})(\charmap{W}) \otimes \charmap{X})
    \\&= \mu_{\Bisp_{p_1},\Bisp_q} (\charmap{U(W)}) \otimes
    \charmap{X}
    = \charmap{U(W) X};
  \end{align*}
  here \(U(W) \subseteq \Bisp_{p_1}\) is a compact slice by
  Proposition~\ref{pro:compute_I}, so that \(U(W) X \subseteq
  \Bisp_{p_1 q}\) is a compact slice by Lemma~\ref{lem:base_of_comp}.
  The functions of the form~\(\charmap{W X}\) for slices \(W\)
  and~\(X\) as above span \(A_R(\Bisp_{p_2 q})\)  by
  Theorem~\ref{the:steinmod_is_quot_of_sum} and
  Lemma~\ref{lem:base_of_comp}.
  We already know that \(U(W) X = \tilde\alpha_{p_1,p_2}^q(U)(W X)\).
  So the computations above show that \(\Phi_{p_1,p_2,q}(\charmap{U})
  = \charmap{\tilde\alpha_{p_1,p_2}^q(U)}\).
\end{proof}

\begin{remark}
  The notation for \(\tilde\alpha_{p_1,p_2}^q(U)\) comes from the tight
  case, where there is a well-defined map
  \(\alpha_{p_1,p_2}^q\colon \Bisp_{p_1}\circ \Bisp_{p_2}^* \to
  \Bisp_{p_1q}\circ \Bisp_{p_2 q}^*\), which maps \(x_1 x_2^*\) to
  \((x_1 z) (x_2 z)^*\) for any \(z\in\Bisp_q\) with
  \(\rg(z) = \s(x_1) = \s(x_2)\).
  Clearly, this map~\(\tilde\alpha_{p_1,p_2}^q\) restricts to a
  homeomorphism from~\(U\) onto \(\tilde\alpha_{p_1,p_2}^q(U)\).
\end{remark}

To describe the ring structure on~\(\OOO_\F\), we now transfer the
multiplication to the Steinberg modules \(A_R(\Bisp_{p_1}\circ
\Bisp_{p_2}^*)\).
The following lemma determines this.
It generalises Proposition~\ref{pro:compute_I} when we let \(p_3 = 1\)
and simplify \(\Bisp_p \circ \Bisp_1^* = \Bisp_p \circ \Gr \cong
\Bisp_p\).

\begin{lemma}
  \label{lem:transfer_mult}
  Let \(p_1,p_2,p_3\in P\) and let \(U\subseteq \Bisp_{p_1} \circ
  \Bisp_{p_2}^*\) and \(V\subseteq \Bisp_{p_2} \circ \Bisp_{p_3}^*\)
  be compact slices.
  There is a compact slice \(U\circ V\subseteq \Bisp_{p_1}\circ
  \Bisp_{p_3}^*\) such that the map
  \begin{equation}
    \label{eq:transfer_mult}
    U \times_{\Qu_2,\Bisp_{p_2}/\Gr, \Qu_1} V \to U\circ V,\qquad
    (u_1 u_2^*, v_1 v_2^*) \mapsto u_1 \braket{u_2}{v_1} v_2^*,
  \end{equation}
  is a homeomorphism.
  If \(f\in A_R(\Bisp_{p_1}\circ \Bisp_{p_2}^*)\), \(h\in
  A_R(\Bisp_{p_2}\circ \Bisp_{p_3}^*)\), then there is \(f* h \in
  A_R(\Bisp_{p_1}\circ \Bisp_{p_3}^*)\) with
  \(\mathcal{I}_{p_1,p_2}(f) \mathcal{I}_{p_2,p_3}(h) =
  \mathcal{I}_{p_1,p_3}(f*h)\) in~\(\OOO_\F\).
  Here \(\charmap{U} * \charmap{V} = \charmap{U \circ V}\).
\end{lemma}

\begin{proof}
  Let \((u_1 u_2^*, v_1 v_2^*)\) be in the domain of the map
  in~\eqref{eq:transfer_mult}.
  Then \(\s(u_1) = \s(u_2)\), \(\s(v_1) = \s(v_2)\) and \(\Qu(u_2) =
  \Qu_2(u_1 u_2^*) = \Qu_1(v_1 v_2^*) = \Qu(v_1)\).
  So \(u_1 \braket{u_2}{v_1} v_2^*\) is defined.
  The properties of the bracket map say that \(u_1 \braket{u_2}{v_1}
  v_2^* = u_1 g \braket{u_2 g}{v_1 h} (v_2 h)^*\) for all
  \(g,h\in\Gr\) with \(\rg(g) = \s(u_1) = \s(u_2)\), \(\rg(h) =
  \s(v_1) = \s(v_2)\), so that
  \(u_1\braket{u_2}{v_1} v_2^* \in \Bisp_{p_1} \circ \Bisp_{p_3}^*\)
  depends only on \(u_1 u_2^* \in \Bisp_{p_1} \circ \Bisp_{p_2}^*\) and
  \(v_1 v_2^* \in \Bisp_{p_2} \circ \Bisp_{p_3}^*\).
  Thus the map in~\eqref{eq:transfer_mult} is well-defined.
  Of course, we let \(U\circ V\) be its image, so that this map
  becomes surjective.
  
  We compute \(\Qu_1(u_1 \braket{u_2}{v_1} v_2^*) = \Qu(u_1) =
  \Qu_1(u_1 u_2^*)\) and \(\Qu_2(u_1 \braket{u_2}{v_1} v_2^*) = \Qu(v_2) =
  \Qu_2(v_1 v_2^*)\).
  Since \(U\) and~\(V\) are slices, these determine
  \(u_1 u_2^*\in U\) or \(v_1 v_2^*\in V\), respectively.
  Thus the map in~\eqref{eq:transfer_mult} is injective.
  Since \(\Qu(u_2) = \Qu(v_1)\), and \(U\) and~\(V\) are slices, 
  either \(u_1 u_2^*\in U\) or \(v_1 v_2^*\in V\) determines the
  other.
  As a consequence, \(\Qu_1\) and~\(\Qu_2\) restrict to injective maps
  on \(U\circ V\).
  An argument as in the proof of Proposition~\ref{pro:compute_I} shows
  that the same formula as in~\eqref{eq:transfer_mult} defines a local
  homeomorphism \((\Bisp_{p_1} \circ \Bisp_{p_2}^*)
  \times_{\Qu_2,\Bisp_{p_2}/\Gr,\Qu_1} (\Bisp_{p_2} \circ
  \Bisp_{p_3}^*) \to \Bisp_{p_1} \circ \Bisp_{p_3}^*\).
  Since the restriction of this local homeomorphism
  in~\eqref{eq:transfer_mult} is bijective, it follows that this map
  is a homeomorphism onto an open subset.
  Therefore, \(U\circ V\) is open in \(\Bisp_{p_1} \circ
  \Bisp_{p_3}^*\) and compact Hausdorff.
  Since \(\Qu_1\) and~\(\Qu_2\) restrict to injective maps on \(U\circ
  V\), it is a compact slice.

  Let \(X\subseteq \Bisp_{p_3}\) be a slice.
  By Proposition~\ref{pro:I_charmap},
  \(\mathcal{I}_{p_2,p_3} (\charmap{V})\) maps~\(\charmap{X}\)
  to~\(\charmap{V(X)}\), the slice consisting of all
  \(v_1\braket{v_2}{x}\)
  with \(v_1 v_2^* \in V\), \(x\in X\) and \(\Qu(v_2) = \Qu(x)\).
  Then \(\mathcal{I}_{p_1,p_2} (\charmap{U})\) maps this on to
  \(\charmap{U(V(X))}\), the slice consisting of all \(u_1
  \braket{u_2}{v_1\braket{v_2}{x}} = u_1\braket{u_2}{v_1}
  \braket{v_2}{x}\) for all \(u_1 u_2^*\in U\), \(v_1 v_2^* \in
  V\) and \(x\in X\) with \(\Qu(v_2) = \Qu(x)\) and \(\Qu(u_2) =
  \Qu(v_1)\).
  This is the same as \(\charmap{(U\circ V)(X)}\).

  Recall that the product in~\(\OOO_\F\) of
  \begin{align*}
    \mathcal{I}_{p_2,p_3} (\charmap{V})
    &\in \Hom\bigl(A_R(\Bisp_{p_3}),A_R(\Bisp_{p_2})\bigr)A_R(\Gr),\\
    \mathcal{I}_{p_1,p_2} (\charmap{U})
    &\in \Hom\bigl(A_R(\Bisp_{p_2}),A_R(\Bisp_{p_1})\bigr)A_R(\Gr)
  \end{align*}
  is represented by the composite map in
  \(\Hom\bigl(A_R(\Bisp_{p_3}),A_R(\Bisp_{p_1})\bigr)A_R(\Gr)\).
  This implies that \(f*h\) as in the statement of the lemma
  exists.
  Since \(\mathcal{I}_{p_1,p_3}(\charmap{U\circ V})(\charmap{X}) =
  \charmap{U(V(X))} = \mathcal{I}_{p_1,p_2} (\charmap{U}) \circ
  \mathcal{I}_{p_2,p_3} (\charmap{V})\) for all compact slices~\(X\), it
  follows that
  \[
    \mathcal{I}_{p_1,p_2} (\charmap{U})
    \circ \mathcal{I}_{p_2,p_3} (\charmap{V})
    = \mathcal{I}_{p_1,p_3}(\charmap{U \circ V}).\qedhere
  \]
\end{proof}

\subsection{Comparison to the groupoid model}
\label{sec:compare_GMH}

The groupoid model~\(\GMH\) of~\(\X\) is described in
Definition~\ref{def:gpmod} as the groupoid model of the tight diagram
\(\X\Omega=(P,\Gr\Omega, \Bisp_p \Omega,\mu_{p,q} \Omega)\):
\[
  \GMH \defeq \bigsqcup_{g\in G} \GMH_g
  = \bigsqcup_{g\in G} \varinjlim_{(p_1,p_2)\in R_g}
  \Bisp_{p_1} \Omega\circ_{\Gr\Omega} (\Bisp_{p_2} \Omega)^*.
\]

We are going to construct a canonical nondegenerate homomorphism
from~\(\OOO_\F\) to \(A_R(\GMH)\) and then prove that it is an algebra
isomorphism.
By the universal property of the covariance ring, a nondegenerate
homomorphism \(\OOO_\F\to A_R(\GMH)\) corresponds to a covariant
representation of the diagram of bimodules \(A_R(\Bisp_p)\) in
\(A_R(\GMH)\).
Such a covariant representation consists of maps \(A_R(\Bisp_p)\to
A_R(\GMH)\) for all \(p\in P\) with certain properties.
More generally, we will construct maps \(A_R(\Bisp_{p_1}\circ
\Bisp_{p_2}^*) \to A_R(\GMH)\) for all \(p_1,p_2\in P\).
Later, we will see that these maps are compatible with the structure
maps~\(\Phi_{p_1,p_2,q}\) in Lemma~\ref{lem:compute_Phi}, so that they
descend to the inductive limits~\(\OO[g]\), and that the resulting map
on~\(\OOO_\F\) is bijective.
This gives us the desired isomorphism.

Since \(\Gr\Omega = \Gr \ltimes \Omega\), a composition over the
groupoid~\(\Gr\Omega\) such as \(\Bisp_{p_1} \Omega\circ_{\Gr\Omega}
(\Bisp_{p_2} \Omega)^*\) involves a fibre product over~\(\Omega\) and
an orbit space for a \(\Gr\)\nb-action.
That is, \(\Bisp_{p_1} \Omega \circ_{\Gr\Omega} (\Bisp_{p_2}
\Omega)^*\) is the orbit space of the action of~\(\Gr\) on the triple
fibre product \(\Bisp_{p_1} \times_{\s,\Gr^0,\s} \Bisp_{p_2}
\times_{\s,\Gr^0,\pi_1^\infty} \Omega\) by the \(\Gr\)\nb-action
defined by \((x_1,x_2,\omega)\cdot g = (x_1 g, x_2 g, g^{-1}
\omega)\).
(See the proof of Proposition~\ref{prop:tightenediag} given
in \cite{Meyer:Diagrams_models}*{\S8.1} for more details.)
The map~\(\pi_1^\infty\) is a proper local homeomorphism, and this
property is preserved under pullbacks
\cite{Antunes-Ko-Meyer:Groupoid_correspondences}*{Lemma~5.1}.
Therefore, the following map is proper:
\begin{equation}
  \label{eq:projection_before_orbit}
  \Bisp_{p_1} \times_{\s,\Gr^0,\s} \Bisp_{p_2}
  \times_{\s,\Gr^0,\pi_1^\infty} \Omega \to
  \Bisp_{p_1} \times_{\s,\Gr^0,\s} \Bisp_{p_2},
  \qquad
  (x_1,x_2,\omega)\mapsto (x_1,x_2).
\end{equation}

\begin{lemma}
  \label{lem:pi_proper_local_homeo}
  The map in~\eqref{eq:projection_before_orbit} induces a proper
  continuous map
  \[
    \pi_{p_1,p_2}\colon \Bisp_{p_1} \Omega\circ_{\Gr\Omega}
    (\Bisp_{p_2} \Omega)^* \to \Bisp_{p_1}\circ \Bisp_{p_2}^*.
  \]
\end{lemma}

\begin{proof}
  The map in~\eqref{eq:projection_before_orbit} is
  \(\Gr\)\nb-equivariant for the relevant \(\Gr\)\nb-actions on the
  spaces \(\Bisp_{p_1} \times_{\s,\Gr^0,\s} \Bisp_{p_2}
  \times_{\s,\Gr^0,\pi_1^\infty} \Omega\) and \(\Bisp_{p_1}
  \times_{\s,\Gr^0,\s} \Bisp_{p_2}\), and so it induces a
  map~\(\pi_{p_1,p_2}\) on the orbit spaces as in the lemma.
  The two actions on \(\Bisp_{p_1} \times_{\s,\Gr^0,\s} \Bisp_{p_2}
  \times_{\s,\Gr^0,\pi_1^\infty} \Omega\) and \(\Bisp_{p_1}
  \times_{\s,\Gr^0,\s} \Bisp_{p_2}\) are basic because the right actions
  on \(\Bisp_{p_1}\) and~\(\Bisp_{p_2}\) are basic.
  Therefore, their orbit space projections are surjective local
  homeomorphisms by
  \cite{Antunes-Ko-Meyer:Groupoid_correspondences}*{Lemma~2.10}.

  A continuous map \(f\colon X\to Y\) between two spaces with an ample
  base is proper if and only if any point in~\(Y\) has a Hausdorff,
  compact open neighbourhood~\(U\) such that \(f^{-1}(U)\) is Hausdorff,
  compact.
  In the case at hand, such~\(U\) are provided by compact slices of
  the form~\(U_1 U_2^*\) with compact slices \(U_j
  \subseteq\Bisp_{p_j}\) for \(j=1,2\).
  The preimage \(\pi_{p_1,p_2}^{-1}(U_1 U_2^*)\) is the set of all
  \((u_1, \omega) (u_2, \omega)^*\) with \(u_1\in U_1\), \(u_2\in U_2\),
  and \(\omega \in \Omega\) such that \(\s(u_1) = \s(u_2) =
  \rg(\omega)\).
  Since \(\Qu|_{U_1}\) is injective, the relation \((u_1, \omega)
  (u_2, \omega)^* = (u_1 g, g^{-1}\omega) (u_2 g, g^{-1}\omega)^*\) does
  not identify any of the elements \((u_1, \omega) (u_2, \omega)^*\)
  above.
  So \(\pi_{p_1,p_2}^{-1}(U_1 U_2^*)\) is homeomorphic to \(U_1
  \times_{\s,\Gr^0,\s} U_2 \times_{\rg,\Gr^0,\pi_1^\infty} \Omega\), and
  the latter is Hausdorff and compact because the spaces \(U_1\),
  \(U_2\) and \((\pi_1^\infty)^{-1}(\s(U_2)\cap \s(U_1))\) are
  and~\(\Gr^0\) is Hausdorff.
\end{proof}

A proper continuous map \(g\colon X\to Y\) induces a map
\(g^*\colon A_R(Y) \to A_R(X)\), \(h\mapsto h\circ g\),
by Proposition~\ref{pro:Steinberg_functorial}.
Therefore, the maps in Lemma~\ref{lem:pi_proper_local_homeo} induce
maps
\begin{equation}
  \label{eq:pi_star_pp}
  \pi_{p_1,p_2}^*\colon A_R(\Bisp_{p_1}\circ \Bisp_{p_2}^*) \to
  A_R(\Bisp_{p_1} \Omega\circ_{\Gr\Omega} (\Bisp_{p_2} \Omega)^*).
\end{equation}
Since~\(\X\Omega\) is a tight diagram, the canonical maps
\begin{equation}
  \label{eq:lambda_star_pp}
  \lambda_{p_1,p_2}\colon
  \Bisp_{p_1} \Omega\circ_{\Gr\Omega} (\Bisp_{p_2} \Omega)^*
  \to \varinjlim \Bisp_{p_1} \Omega\circ_{\Gr\Omega} (\Bisp_{p_2} \Omega)^*
  = \GMH_g \subseteq \GMH
\end{equation}
in the inductive limit cone are local homeomorphisms by
Lemma~\ref{lem:lambda_nice}.
By Proposition~\ref{pro:Steinberg_functorial}, they induce maps
\[
  (\lambda_{p_1,p_2})_*\colon
  A_R(\Bisp_{p_1} \Omega\circ_{\Gr\Omega} (\Bisp_{p_2} \Omega)^*)
  \to A_R(\GMH).
\]

\begin{lemma}
  \label{lem:lambda_pi_properties}
  Let \(p_1,p_2\in P\).
  The maps
  \[
    \tilde\alpha_{p_1,p_2}^\infty \defeq (\lambda_{p_1,p_2})_* \circ
    (\pi_{p_1,p_2})^*\colon A_R(\Bisp_{p_1}\circ \Bisp_{p_2}^*)
    \to A_R(\GMH)
  \]
  and the subsets
  \[
    \tilde\alpha_{p_1,p_2}^\infty(U)
    \defeq \lambda_{p_1,p_2}(\pi^{-1}_{p_1,p_2}(U))
    \subseteq \GMH_{p_1 p_2^{-1}}
  \]
  for a slice \(U\subseteq \Bisp_{p_1} \circ \Bisp_{p_2}^*\) have the
  following properties:
  \begin{enumerate}
  \item \label{en:lambda_pi_properties_0}%
    \(\tilde\alpha_{p_1,p_2}^\infty(U)\) is a compact slice
    in~\(\GMH\) and the restriction of~\(\lambda_{p_1,p_2}\) to
    \(\pi_{p_1,p_2}^{-1}(U)\) is a homeomorphism onto
    \(\tilde\alpha_{p_1,p_2}^\infty(U)\);
  \item \label{en:lambda_pi_properties_1}%
    \(\tilde\alpha_{p_1,p_2}^\infty(\charmap{U}) =
    \charmap{\tilde\alpha_{p_1,p_2}^\infty(U)}\);
  \item \label{en:lambda_pi_properties_2}%
    \(\tilde\alpha^\infty_{p_1 q,p_2 q} (\tilde\alpha_{p_1,p_2}^q(U)) =
    \tilde\alpha^\infty_{p_1,p_2}(U)\);
  \item \label{en:lambda_pi_properties_3}%
    \(\tilde\alpha^\infty_{p_1,p_2}(f) * \tilde\alpha^\infty_{p_2,p_3}(h) =
    \tilde\alpha^\infty_{p_1,p_3}(f * h)\) for \(f\in A_R(\Bisp_{p_1}\circ
    \Bisp_{p_2}^*)\), \(h\in A_R(\Bisp_{p_2}\circ \Bisp_{p_3}^*)\),
    where \(f* h\in A_R(\Bisp_{p_1}\circ \Bisp_{p_3}^*)\) is the
    transferred product from~\(\OOO_\F\).
  \end{enumerate}
\end{lemma}

The abuse of notation of using the same
name~\(\tilde\alpha^\infty_{p_1,p_2}\) for the map on functions and
subsets should not lead to confusion because
of~\ref{en:lambda_pi_properties_1}.

\begin{proof}
  We prove \ref{en:lambda_pi_properties_0}
  and~\ref{en:lambda_pi_properties_1} together.
  We first apply~\(\pi_{p_1,p_2}^*\) and
  then~\((\lambda_{p_1,p_2})_*\).

  It is clear that~\(\pi_{p_1,p_2}^*\) in~\eqref{eq:pi_star_pp}
  maps~\(\charmap{U}\) to \(\charmap{\pi_{p_1,p_2}^{-1}(U)}\).
  This forces \(\pi_{p_1,p_2}^{-1}(U)\) to be a compact, open,
  Hausdorff subset of \(\Bisp_{p_1} \Omega\circ_{\Gr\Omega}
  (\Bisp_{p_2} \Omega)^*\).
  Of course, the reason for this is that~\(\pi_{p_1,p_2}\) is a proper
  continuous map.

  By definition, \(\pi_{p_1,p_2}^{-1}(U)\) consists of all
  \((u_1, \omega)(u_2, \omega)^*\) with \(u_1 u_2^*\in U\) and
  \(\omega\in \Omega\) satisfying \(\s(u_1) = \s(u_2) = \rg(\omega)\),
  and where \((u_1, \omega) (u_2, \omega)^*  = (u_1 g, g^{-1}\omega)
  (u_2 g, g^{-1}\omega)^*\) for all \(g\in\Gr\) with \(\s(u_1) =
  \s(u_2) = \rg(\omega) = \rg(g)\).
  We claim that \(\pi_{p_1,p_2}^{-1}(U)\) is a slice in the sense that
  the projections to \(\Bisp_{p_1} \circ_{\Gr} \Omega\) and
  \(\Bisp_{p_2} \circ_{\Gr} \Omega\) that map \((u_1, \omega)(u_2,
  \omega)^*\) to \(u_1, \omega\) and \(u_2, \omega\), respectively, are
  injective on~\(\pi_{p_1,p_2}^{-1}(U)\).

  It suffices to prove this for one of the two projections.
  Assume \((u_1, \omega) = (u_1', \omega')\) for some \((u_1, \omega)(u_2,
  \omega)^*\) and  \((u_1', \omega')(u_2', \omega')^*\)
  in~\(\pi_{p_1,p_2}^{-1}(U)\).
  That is, there is \(g\in\Gr\) with \(\s(u_1) = \rg(g)\), \(u_1' =
  u_1 g\) and \(\omega' = g^{-1} \omega\).
  Then \(\Qu_1(u_1 u_2^*) = \Qu(u_1) = \Qu(u_1') = \Qu_1(u_1' (u_2')^*)\).
  Since~\(U\) is a slice, it follows that \(u_1 u_2^* = u_1'
  (u_2')^*\), that is, \(u_1 h = u_1'\) and \(u_2 h = u_2'\) for some
  \(h\in\Gr\) with \(\s(u_1) = \s(u_2) = \rg(h)\).
  Here \(g=h\) because the right action on~\(\Bisp_{p_1}\) is free.
  Then \((u_1, \omega)(u_2, \omega)^* = (u_1', \omega')(u_2', \omega')^*\)
  in \(\Bisp_{p_1} \Omega\circ_{\Gr\Omega} (\Bisp_{p_2} \Omega)^*\).

  Recall that there are canonical homeomorphisms \(\Bisp_{p_j} \circ
  \Omega \congto \Omega\) which, roughly speaking, identify the
  class of \((u, \omega) \in \Bisp_{p_j}\times_{\s,\Gr^0,\pi_1^\infty}
  \Omega\) for the equivalence relation \((u_1,\omega) \sim (u_1
  g,g^{-1}\omega)\) with the infinite word \(u \omega \in \Omega\);
  more precisely, \(u\omega \in \varprojlim \Bisp_{p_j q}/\Gr\), which
  is isomorphic  to~\(\Omega\) because the relevant subcategory is
  cofinal in the category~\(PO\) defining~\(\Omega\).
  We write an element of \(\Bisp_{p_j} \Omega\) with a comma to
  distinguish it from the corresponding element of \(\Bisp_{p_j} \circ
  \Omega \cong \Omega\).
  When we map \(\Bisp_{p_1} \Omega \circ (\Bisp_{p_2}\Omega)^*\)
  to~\(\GMH\) using~ \(\lambda_{p_1,p_2}\), then the maps that send
  \((u_1, \omega)(u_2, \omega)^*\) to \(u_1 \omega\in\Omega\) and \(u_2
  \omega\in\Omega\) become the range and source maps
  \(\rg_{\GMH},\s_{\GMH}\colon \GMH\rightrightarrows \Omega\).
  So the slice property shown above for \(\pi_{p_1,p_2}^{-1}(U)\) says
  that the two maps
  \[
    \rg_{\GMH}\circ \lambda_{p_1,p_2},
    \s_{\GMH}\circ \lambda_{p_1,p_2}\colon
    \Bisp_{p_1} \Omega\circ_{\Gr\Omega} (\Bisp_{p_2} \Omega)^*
    \to \Omega
  \]
  are injective on the subset \(\pi_{p_1,p_2}^{-1}(U)\).
  This implies that the restriction of~\(\lambda_{p_1,p_2}\) to
  \(\pi_{p_1,p_2}^{-1}(U)\) is injective.
  Then it follows that \((\lambda_{p_1,p_2})_*
  (\charmap{\pi_{p_1,p_2}^{-1}(U)}) = \charmap{\lambda_{p_1,p_2}
  (\pi_{p_1,p_2}^{-1}(U))}\), which is the formula asserted
  in~\ref{en:lambda_pi_properties_1}.
  In particular, this subset must be compact, open and Hausdorff.
  The argument above also shows that \(\rg_{\GMH}\) and~\(\s_{\GMH}\)
  restrict to injective maps on it.
  That is, \(\lambda_{p_1,p_2}(\pi_{p_1,p_2}^{-1}(U))\) is a slice
  in~\(\GMH\).

  We have now proven
  \ref{en:lambda_pi_properties_0}--\ref{en:lambda_pi_properties_1}.
  Next, we prove~\ref{en:lambda_pi_properties_2}.
  Lemma~\ref{lem:compute_Phi} shows that
  \(\Phi_{p_1,p_2,q}(\charmap{U})\) for a compact slice \(U\subseteq
  \Bisp_{p_1}\circ \Bisp_{p_2}^*\) is the characteristic function of the
  compact slice \(\tilde\alpha_{p_1,p_2}^q(U) \subseteq \Bisp_{p_1 q} \circ
  \Bisp_{p_2 q}^*\) defined in Lemma~\ref{lem:compute_Phi}.
  So \((\pi_{p_1 q,p_2 q})^* \circ \Phi_{p_1,p_2,q}\)
  maps~\(\charmap{U}\) to the characteristic function of \(\pi_{p_1
    q,p_2,q}^{-1}(\tilde\alpha_{p_1,p_2}^q(U))\).

  Since the diagram~\(\X\Omega\) is tight,
  Definition~\ref{def:alpha_pp^q} provides well-defined maps
  \[
    \alpha_{p_1,p_2}^q\colon
    \Bisp_{p_1} \Omega\circ_{\Gr\Omega} (\Bisp_{p_2} \Omega)^*
    \to \Bisp_{p_1 q} \Omega\circ_{\Gr\Omega} (\Bisp_{p_2 q} \Omega)^*,
    \qquad
    x y^* \mapsto (x z) (y z)^*.
  \]
  By construction, \(\lambda_{p_1 q,p_2 q} \circ \alpha_{p_1,p_2}^q =
  \lambda_{p_1,p_2}\).
  We claim that
  \begin{equation}
    \label{eq:claim_pi_alpha_order}
    \pi_{p_1 q,p_2,q}^{-1}(\tilde\alpha_{p_1,p_2}^q(U))
    =  \alpha_{p_1,p_2}^q(\pi_{p_1,p_2}^{-1}(U)).
  \end{equation}
  Together with \(\lambda_{p_1 q,p_2 q} \circ \alpha_{p_1,p_2}^q =
  \lambda_{p_1,p_2}\), this implies~\ref{en:lambda_pi_properties_2}.

  The set \(\tilde\alpha_{p_1,p_2}^q(U)\) consists of all
  \((u_1 z) (u_2 z)^* \in \Bisp_{p_1 q} \circ \Bisp_{p_2 q}^*\)
  with \(u_1 u_2^*\in U\), \(z\in \Bisp_q\) and
  \(\s(u_1) = \s(u_2) = \rg(z)\).
  Thus the preimage \(\pi_{p_1 q,p_2,q}^{-1}(\tilde\alpha_{p_1,p_2}^q(U))\)
  consists of all \((u_1 z, \omega) (u_2 z, \omega)^* \in \Bisp_{p_1
    q}\Omega \circ_{\Gr\Omega} (\Bisp_{p_2 q} \Omega)^*\) with \(u_1
  u_2^* \in U\), \(z\in \Bisp_q\), \(\omega \in \Omega\), \(\s(u_1) =
  \s(u_2) = \rg(z)\), and \(\s(z) = \pi_1^\infty(\omega)\).

  The set \(\pi_{p_1 q,p_2,q}^{-1}(U)\) consists of
  \((u_1, \omega') (u_2, \omega')^*\) with \(u_1 (u_2)^*\in U\) and
  \(\omega'\in\Omega\) such that \(\s(u_1') = \s(u_2') =
  \pi_1^\infty(\omega')\).
  Since the range map induces a homeomorphism
  \(\Bisp_q\Omega/\Gr\Omega \cong \Omega\),
  there is \((z,\omega) \in \Bisp_q \Omega\) with
  \(\s(z) = \pi_1^\infty(\omega)\) and
  \(\rg_{\Bisp_q\Omega}(z,\omega) = \omega'\), and its
  \(\Gr\)\nb-orbit is unique.
  This implies \(\rg(z) = \s(u_1) = \s(u_2)\).
  By definition, \(\alpha_{p_1,p_2}^q((u_1, \omega') (u_2, \omega')^*) =
  (u_1 z, \omega) (u_2 z, \omega)^*\).
  Here \(u_1,u_2,z,\omega\) satisfy the
  conditions above to define an element of
  \(\pi_{p_1 q,p_2,q}^{-1}(\tilde\alpha_{p_1,p_2}^q(U))\).
  Conversely, given any such
  \(u_1,u_2,z,\omega\), we may take
  \(\omega' = \rg_{\Bisp_q\Omega}(z, \omega)\) to witness that we have
  an element of \(\pi_{p_1 q,p_2 q}^{-1}(\tilde\alpha_{p_1,p_2}^q(U))\).
  This proves~\eqref{eq:claim_pi_alpha_order} and finishes the proof
  of~\ref{en:lambda_pi_properties_2}.

  Finally, we prove~\ref{en:lambda_pi_properties_3}.
  It suffices to prove this for \(f = \charmap{U}\), \(h =
  \charmap{V}\) for slices \(U\subseteq \Bisp_{p_1} \circ
  \Bisp_{p_2}^*\) and \(V\subseteq \Bisp_{p_2} \circ \Bisp_{p_3}^*\).
  Lemma~\ref{lem:transfer_mult} shows that \(\charmap{U} * \charmap{V}
  = \charmap{U \circ V}\) for a certain slice \(U\circ
  V\subseteq\Bisp_{p_2} \circ \Bisp_{p_3}^*\).
  We have seen above that~\(\tilde\alpha^\infty_{p_1,p_2}\) maps
  \(\charmap{U}\) to the characteristic function of the slice
  \(\tilde\alpha^\infty_{p_1,p_2}(U) = \lambda_{p_1,p_2}(
  \pi_{p_1,p_2}^{-1}(U))\), and similarly for the other relevant
  characteristic functions.
  Since these sets are slices, the convolution of the characteristic
  functions of \(\tilde\alpha_{p_1,p_2}^\infty(U)\) and
  \(\tilde\alpha_{p_2,p_3}^\infty(V)\) is the characteristic function of
  the product slice \(\tilde\alpha_{p_1,p_2}^\infty(U) \cdot
  \tilde\alpha_{p_2,p_3}^\infty(V)\).
  So what we have to prove is that this product slice is equal to
  \(\tilde\alpha_{p_1,p_3}^\infty(U\circ V)\).

  The multiplication in~\(\GMH\) is described in
  Definition~\ref{def:gpmod} by the equation \(x_1 x_2^* \cdot
  x_2 x_3^* = x_1 x_3^*\).
  Here \(x_j \in \Bisp_{p_j}\Omega\) for \(j=1,2,3\).
  Since \(x_1 x_2^* = x_1 g (x_2 g)^*\) if \(g\in\Gr\Omega\) is
  composable with~\(x_2\), this implies \(x_1 g (x_2 g)^* \cdot
  x_2 x_3^* = x_1 x_3^*\).
  The product \(x_1 x_2^* \cdot x_2' x_3^*\) is only defined if
  \(x_2' g = x_2\) for some \(g\in \Gr\Omega\).
  Here~\(g\) is unique, namely, \(g = \braket{x_2'}{x_2}\).
  So \(x_1 x_2^* \cdot x_2' x_3^* = x_1 \braket{x_2}{x_2'} x_3^*\) if
  \(\braket{x_2}{x_2'}\) is defined, and \(x_1 x_2^* \cdot x_2' x_3^*\)
  is not defined otherwise.

  Elements of \(\tilde\alpha_{p_1,p_2}^\infty(U)\) are represented by
  \((u_1, \omega) (u_2, \omega)^*\) with \(u_1 u_2^*\in U\),
  \(\omega\in\Omega\) such that \(\s(u_1) = \s(u_2) =
  \pi_1^\infty(\omega)\).
  The equivalence relation on these representatives is generated by
  \((u_1, \omega) (u_2, \omega)^* \sim (u_1 g, g^{-1}\omega) (u_2 g,
  g^{-1}\omega)^*\) for \(g\in\Gr\) with \(\rg(g) = \s(u_1) = \s(u_2)
  = \rg(\omega)\).
  Similarly, elements of \(\tilde\alpha_{p_2,p_3}^\infty(V)\) are represented by
  \((v_1, \omega') (v_2, \omega')^*\) with \(v_1 v_2^*\in V\) and
  \(\omega'\in\Omega\) such that \(\s(v_1) = \s(v_2) =
  \pi_1^\infty(\omega')\),
  with a similar equivalence relation.
  Now \((u_1, \omega) (u_2, \omega)^* \cdot (v_1, \omega')
  (v_2, \omega')^*\) is only defined if \((u_2, \omega) = (v_1, \omega')\)
  in~\(\Omega\).
  Equivalently, there is \(g\in\Gr\) with \(\rg(g) = \s(u_2)\), \(v_1
  = u_2 g\), and \(\omega' = g^{-1} \omega\).
  Here \(g = \braket{u_2}{v_1}\) and so \(\omega = \braket{u_2}{v_1}
  \omega'\).
  Since \((u_1, \omega) (u_2, \omega)^* = (u_1 g, g^{-1} \omega) (u_2
  g, g^{-1}\omega)^*\), the product is equal to \((u_1
  \braket{u_2}{v_1}, \omega') (v_2,\omega')^*\).
  Now Lemma~\ref{lem:transfer_mult} shows that the product slice is
  indeed equal to \(\tilde\alpha_{p_1,p_3}^\infty(U\circ V)\).
  This finishes the proof of~\ref{en:lambda_pi_properties_3}.
\end{proof}

Lemma \ref{lem:lambda_pi_properties}.\ref{en:lambda_pi_properties_2}
and the description of~\(\OOO_\F\) in Section~\ref{sec:diagrams} show
that the maps~\(\tilde\alpha^\infty_{p_1,p_2}\) for \(p_1 p_2^{-1} = g\) descend to
a map on the inductive limit~\(\OO[g]\).
And these maps for all \(g\in G\) piece together to an \(R\)\nb-module
map \(\varrho\colon \OOO_\F \to A_R(\GMH)\).
Lemma \ref{lem:lambda_pi_properties}.\ref{en:lambda_pi_properties_3}
implies that~\(\varrho\) is an algebra homomorphism.

It is easy to see that~\(\varrho\) restricts to a nondegenerate
homomorphism on \(A_R(\Gr^0)\).
(We do not check this because we will soon prove that~\(\varrho\) is
an isomorphism.)
Then the universal property of the covariance ring implies that the
maps
\[
  \tilde\alpha^\infty_{p,e}\colon A_R(\Bisp_p)
  \cong A_R(\Bisp_p \circ \Bisp_e^*)
  \to A_R(\GMH)
\]
for \(p\in P\) form a covariant representation of the
diagram~\(\A * \X\) in \(A_R(\GMH)\).
We could also have constructed this covariant representation first and
then defined~\(\varrho\) as the nondegenerate homomorphism that it
induces.
The more general statements in Lemma~\ref{lem:lambda_pi_properties}
make it easier to prove that~\(\varrho\) is an isomorphism, which is
our next goal.
We fix \(g\in G\) and want to prove that the restriction
\(\varrho_g\colon \OO[g] \to A_R(\GMH_g)\) is invertible.
This is where our specific choice of ample bases in \(\Bisp_{p_1}
\circ \Bisp_{p_2}^*\) becomes important.
Namely, we are going to prove that their images
\(\tilde\alpha_{p_1,p_2}^\infty(U)\) in the groupoid model form an
ample base and that all the relations in
Theorem~\ref{the:steinmod_is_quot_of_sum} are also realised
in~\(\OOO_\F\).
(The more obvious base in \(\Bisp_{p_1}\circ \Bisp_{p_2}^*\) of
subsets of the form \(U V^*\) for compact
slices \(U\subseteq \Bisp_{p_1}\) and \(V\subseteq \Bisp_{p_2}\) works
poorly here because images of such slices under the
maps~\(\tilde\alpha_{p_1,p_2}^q\) for \(q\in P\) may fail to have the same
form.)

\begin{proposition}
  \label{prop:base_on_Hg}
  The slices \(\tilde\alpha_{p_1,p_2}^\infty(U)\) for compact slices
  \(U\subseteq \Bisp_{p_1}\circ \Bisp_{p_2}^*\) and \(p_1,p_2\in P\)
  form an ample base~\(\B_{\GMH}\) in~\(\GMH\) that is closed under
  taking arbitrary compact open subsets.
\end{proposition}

\begin{proof}
  By
  Lemma~\ref{lem:lambda_pi_properties}.\ref{en:lambda_pi_properties_0}
  these subsets are compact open slices in~\(\GMH\).
  First, we prove that they generate the topology on~\(\GMH\).
  As a projective limit, the space~\(\Omega\) comes with canonical
  maps \(\pi_p^\infty\colon \Omega \to \Bisp_p/\Gr\).
  The topology of~\(\Omega\) is generated by subsets of the form
  \((\pi_p^\infty)^{-1}(U)\) for compact open subset \(U\subseteq
  \Bisp_p/\Gr\).
  The topology on \(\Bisp_p\Omega = \Bisp_p
  \times_{\s,\Gr^0,\pi_1^\infty} \Omega\) is the canonical one on the
  fibre product, so sets of the form \(U \times_{\s,\pi_1^\infty}
  (\pi_q^\infty)^{-1}(V)\) for compact open subsets \(U\subseteq
  \Bisp_p\) and \(V\subseteq \Bisp_q/\Gr\) generate its topology.
  This induces a canonical topology on the spaces \(\Bisp_{p_1} \Omega
  \circ (\Bisp_{p_2} \Omega)^*\).
  A base for it is formed, for instance, by sets of the form
  \begin{equation}
    \label{eq:base_set_in_Bipp_Omega}
    \setgiven{(u_1, \omega)(u_2, \omega)^*}{ u_1 u_2^*\in U,\
      \omega\in (\pi_q^\infty)^{-1}(\Qu(V)),\ \s(u_1) = \s(u_2) =
      \pi_1^\infty(\omega)}
  \end{equation}
  for slices \(U\subseteq \Bisp_{p_1} \circ \Bisp_{p_2}^*\),
  \(q\in P\) and compact slices \(V\subseteq \Bisp_p\).
  Any element of \((\pi_q^\infty)^{-1}(\Pi(V))\) may be written as~\(v
  \omega'\) for \(v\in V\) and \(\omega'\in\Omega\) with \(\s(v) =
  \rg(\omega')\).
  Here \(v\omega' \in\Bisp_q \circ_{\Gr} \Omega \cong \Omega\).
  We claim that the \(\lambda_{p_1,p_2}\)-images of the base sets
  in~\eqref{eq:base_set_in_Bipp_Omega} in~\(\GMH\) form a base of the
  topology.
  Let \(W\subseteq \GMH_g\) be an open neighbourhood of some \(h\in
  \GMH\).
  There are \(p_1,p_2\in P\) and \(x\in \Bisp_{p_1}\circ
  \Bisp_{p_2}^*\) such that \(\lambda_{p_1,p_2}(x)=h\).
  Then \(\lambda_{p_1,p_2}^{-1}(W)\subseteq\Bisp_{p_1}\circ
  \Bisp_{p_2}^*\) is an open neighbourhood of~\(x\).
  Since the sets in~\eqref{eq:base_set_in_Bipp_Omega} form a base, one
  of them is contained in~\(\lambda_{p_1,p_2}^{-1}(W)\).
  Then its \(\lambda_{p_1,p_2}\)-image is a neighbourhood of~\(x\)
  contained in~\(W\).
  The subset in~\eqref{eq:base_set_in_Bipp_Omega} has the same
  \(\lambda_{p_1,p_2}\)-image in~\(\GMH\) as the following subset of
  \(\Bisp_{p_1 q} \Omega \circ (\Bisp_{p_2 q} \Omega)^*\):
  \[
    \setgiven{(u_1v, \omega')(u_2 v, \omega')^*}{u_1 u_2^*\in U,\
      v\in V,\ 
      \omega'\in \Omega,\ \s(u_1) = \s(u_2) = \rg(v),\
      \s(v) = \pi_1^\infty(\omega')}
  \]
  Here the elements \((u_1 v) (u_2 v)^* \in \Bisp_{p_1 q}\circ
  \Bisp_{p_2 q}^*\) for \(u_1 u_2^*\in U\) and \(v\in V\) form a compact
  open subset of the slice \(\tilde\alpha_{p_1,p_2}^q(U)\).
  Therefore, the set above is a compact open slice in \(\Bisp_{p_1 q}\circ
  \Bisp_{p_2 q}^*\).
  As a consequence, already the sets of the form
  \(\tilde\alpha_{p_1,p_2}^\infty(U)\) for compact slices \(U\subseteq
  \Bisp_{p_1}\circ \Bisp_{p_2}^*\) and \(p_1,p_2\in P\) form a base for
  the topology on~\(\GMH\).

  Next, we claim that any compact open subset of a set of the form
  \(\tilde\alpha_{p_1,p_2}^\infty(U)\) may also be written in the same
  form, but usually with different \(p_1,p_2\).
  By Lemma
  \ref{lem:lambda_pi_properties}.\ref{en:lambda_pi_properties_0}, the
  map~\(\lambda_{p_1,p_2}\) maps \(\pi_{p_1,p_2}^{-1}(U)\) for a
  slice~\(U\) homeomorphically onto
  \(\tilde\alpha_{p_1,p_2}^\infty(U)\).
  Therefore, any compact open subset of
  \(\tilde\alpha_{p_1,p_2}^\infty(U)\) is the image of a compact open
  subset \(W\subseteq \pi_{p_1,p_2}^{-1}(U) \subseteq \Bisp_{p_1}\Omega
  \circ (\Bisp_{p_2}\Omega)^*\).
  Since \(\Omega\) is defined as a projective limit, we may
  cover~\(W\) by cylinder sets of the form \(X
  \times_{\s,\Gr^0,\pi_1^\infty} \pi_q^{-1}(\Qu(Y))\) with slices
  \(X\subseteq \Bisp_{p_1}\) and \(Y \subseteq \Bisp_q\) for some
  \(q\in P\).
  Since~\(W\) is compact, finitely many such sets suffice.
  Since~\(P\) is an Ore monoid, there is one \(q\in P\) that dominates
  all of them, so that we only need sets as above with this
  fixed~\(q\).
  Then we replace subsets of \(\Bisp_{p_1}\Omega \circ
  (\Bisp_{p_2} \Omega)^*\) by subsets of \(\tilde\alpha_{p_1,p_2}^q(U)
  \subseteq \Bisp_{p_1 q}\Omega \circ (\Bisp_{p_2 q} \Omega)^*\) as
  above.
  This shows that \(W = \tilde\alpha_{p_1 q,p_2 q}^\infty(W')\) for
  a compact open subset \(W'\subseteq \tilde\alpha_{p_1,p_2}^q(U)\).
  Now~\(W'\) is a compact slice in \(\Bisp_{p_1 q} \circ \Bisp_{p_2
    q}^*\) because it is a compact open subset of a compact slice.
  Thus \(W = \tilde\alpha_{p_1 q,p_2 q}^\infty(W')\) for a compact
  slice in \(\Bisp_{p_1 q} \circ \Bisp_{p_2 q}^*\).
  This finishes the proof that compact open subsets of sets
  in~\(\B_{\GMH}\) again belong to~\(\B_{\GMH}\).
\end{proof}

\begin{lemma}
  \label{lem:relations_base_gm}
  Let \(p_1,p_2,q_1,q_2\in P\) and let \(U\subseteq \Bisp_{p_1}\circ
  \Bisp_{p_2}^*\) and \(V\subseteq \Bisp_{q_1}\circ \Bisp_{q_2}^*\) be
  slices.
  Then
  \begin{enumerate}
  \item \(\tilde\alpha_{p_1,p_2}^\infty(U) \supseteq
    \tilde\alpha_{q_1,q_2}^\infty(V)\) if and only if there are \(a,b\in
    P\) with \(p_1 a = q_1 b\), \(p_2 a = q_2 b\) and
    \(\tilde\alpha_{p_1,p_2}^a(U) \supseteq
    \tilde\alpha_{q_1,q_2}^b(V)\) as slices in \(\Bisp_{p_1 a}\circ
    \Bisp_{p_2 a}^* = \Bisp_{q_1 b}\circ \Bisp_{q_2 b}^*\).
  \item \(\tilde\alpha_{p_1,p_2}^\infty(U) =
    \tilde\alpha_{q_1,q_2}^\infty(V)\) if and only if there are
    \(a,b\in P\) with \(p_1 a = q_1 b\), \(p_2 a = q_2 b\) and
    \(\tilde\alpha_{p_1,p_2}^a(U) = \tilde\alpha_{q_1,q_2}^b(V)\) as slices in
    \(\Bisp_{p_1 a}\circ \Bisp_{p_2 a}^* = \Bisp_{q_1 b}\circ
    \Bisp_{q_2 b}^*\).
  \end{enumerate}
\end{lemma}

\begin{proof}
  The second statement about equality follows easily from the first
  statement about inclusions because \(X = Y\) if and only if both
  \(X\supseteq Y\) and \(Y\supseteq X\).
  Here we use that~\(P\) is an Ore monoid because we need to find an
  upper bound for two elements in~\(P\).
  Thus it suffices to prove the statement about inclusions.
  Since \(\tilde\alpha_{p_1,p_2}^\infty(U) = \tilde\alpha_{p_1 a,p_2
    a}^\infty(\tilde\alpha_{p_1,p_2}^a(U))\) by Lemma
  \ref{lem:lambda_pi_properties}.\ref{en:lambda_pi_properties_2}, an
  inclusion \(\tilde\alpha_{p_1,p_2}^a(U) \supseteq
  \tilde\alpha_{q_1,q_2}^b(V)\) as slices in \(\Bisp_{p_1 a}\circ
  \Bisp_{p_2 a}^* = \Bisp_{q_1 b}\circ \Bisp_{q_2 b}^*\) for some
  \(a,b\in P\) with \(p_1 a = q_1 b\), \(p_2 a = q_2 b\) implies that
  \(\tilde\alpha_{p_1,p_2}^\infty(U) \supseteq
  \tilde\alpha_{q_1,q_2}^\infty(V)\).
  The main point is the converse implication.

  We may assume without loss of generality that already
  \(p_1=q_1\) and \(p_2= q_2\).
  If \(p_1 p_2^{-1} \neq q_1 q_2^{-1}\) in~\(G\), then all sets we
  consider are disjoint anyway.
  If \(p_1 p_2^{-1} = q_1 q_2^{-1}\) in~\(G\), then we may choose
  \(a\) and~\(b\) to move both sets to the same \(\Bisp_{t_1} \circ
  \Bisp_{t_2}^*\) for some \(t_1,t_2\in P\).
  So we assume \(p_1 = p_2\) from now on.
  In addition, we assume that there is no \(q\in P\) with
  \(\tilde\alpha_{p_1,p_2}^q(U) \supseteq
  \tilde\alpha_{p_1,p_2}^q(V)\).
  Equivalently, the sets \(\tilde\alpha_{p_1,p_2}^q(V) \setminus
  \tilde\alpha_{p_1,p_2}^q(U)\) are nonempty for all~\(q\).
  These sets are compact and open in \(\Bisp_{p_1 q}\circ \Bisp_{p_2
    q}^*\) because both \(\tilde\alpha_{p_1,p_2}^q(U)\) and
  \(\tilde\alpha_{p_1,p_2}^q(V)\) are compact open subsets.
  
  We claim that the spaces \(\tilde\alpha_{p_1,p_2}^q(V)\) form a
  projective system using the maps
  \begin{equation}
    \label{eq:projective_system_in_base}
    \tilde\alpha_{p_1,p_2}^{q t}(V) \to
    \tilde\alpha_{p_1,p_2}^q(V),\qquad
    (v_1 w x) (v_2 w x)^* \mapsto (v_1 w) (v_2 w)^*
  \end{equation}
  for all \(v_1 v_2^*\in V\), \(w\in \Bisp_q\), \(x\in\Bisp_t\).
  Recall that the notation \(v_1 v_2^* \in V\) means that
  \((v_1,v_2)\in \Bisp_{p_1} \times_{\s,\Gr^0,\s} \Bisp_{p_2}\) and
  that its image in
  \(\Bisp_{p_1}\circ \Bisp_{p_2}^*\)  belongs to~\(V\).
  The words \((v_1 w x) (v_2 w x)^*\) and \((v_1 w) (v_2 w)^*\)
  use the lifts \(v_j\in\Bisp_{p_j}\) and not just the image \(v_1
  v_2^* \in \Bisp_{p_1}\circ \Bisp_{p_2}^*\).
  We must show that the maps in~\eqref{eq:projective_system_in_base}
  are well-defined.
  The equivalence relation that defines \(\Bisp_{p_1 q t} \circ
  \Bisp_{p_2 q t}^*\) is
  \[
    (v_1 w x) (v_2 w x)^* \sim
    (v_1 g_1\,g_1^{-1} w h_1\,h_1^{-1} x k)
    (v_2 g_2\,g_2^{-1} w h_2\,h_2^{-1} x k)^*
  \]
  for \(g_1,g_2,h_1,h_2,k\in\Gr\) with \(\rg(g_1) = \rg(g_2) =
  \s(v_1)= \s(v_2)\), \(\rg(h_1) = \rg(h_2) = \s(w)\),
  \(\rg(k) = \s(x)\).
  In addition, \((v_1 g_1)(v_2 g_2)^* \in V\), \(g_1^{-1} w h_1 =
  g_2^{-1} w h_2\), and \(h_1^{-1} x k = h_2^{-1} x k\) are needed in
  order for the word \((v_1 g_1\,g_1^{-1} w h_1\,h_1^{-1} x k) (v_2
  g_2\,g_2^{-1} w h_2\,h_2^{-1} x k)^*\)  to belong  to
  \(\tilde\alpha_{p_1,p_2}^q(V)\).
  Since~\(V\) is a slice and \(\Qu_1(v_1 v_2^*) = \Qu(v_1) = \Qu(v_1 g_1) =
  \Qu_1(v_1 g_1  (v_2 g_2)^*)\), this forces \(v_1 g_1  (v_2 g_2)^*
  = v_1 v_2^*\).
  This is equivalent to \(g_1 = g_2\) because the right
  \(\Gr\)\nb-actions on \(\Bisp_{p_1}\) and \(\Bisp_{p_2}\) are free.
  Then \(h_1 = h_2\) follows because the right \(\Gr\)\nb-action
  on~\(\Bisp_q\) is free.
  And the latter is what we need for \((v_1 g_1\,g_1^{-1} w h_1) (v_2
  g_2\,g_2^{-1} w h_2)^* = (v_1 w)(v_2 w)^*\) in \(\Bisp_{p_1 q} \circ
  \Bisp_{p_2 q}^*\).
  This finishes the proof that the map
  in~\eqref{eq:projective_system_in_base} is well-defined.
  It is clear that these maps form a contravariant functor on
  the filtered category~\(PO\) used in Definition~\ref{def:omegaspace}.

  The projective limit of the projective system defined
  in~\eqref{eq:projective_system_in_base} is canonically identified
  with the set \(\tilde\alpha_{p_1,p_2}^\infty(V) \subseteq
  \lambda_{p_1,p_2}(\Bisp_{p_1}\Omega (\Bisp_{p_2}\Omega)^*)\) because
  \(\Omega = \varprojlim \Bisp_q/\Gr\).
  Roughly speaking, in the limit the finite words \(w\in \Bisp_q\)
  approximate an element of~\(\Omega\).
  A crucial point here is that \( (v_1 w) (v_2 w)^* \sim
  (v_1 w h) (v_2 w h)^*\) for all \(h\in\Gr\) with \(\rg(h) = \s(w)\).

  The same remarks apply to~\(U\) instead of~\(V\).
  Even more, we claim that the spaces \(\tilde\alpha_{p_1,p_2}^q(V)
  \setminus \tilde\alpha_{p_1,p_2}^q(U) \subseteq
  \tilde\alpha_{p_1,p_2}^q(V) \subseteq \Bisp_{p_1 q}\circ \Bisp_{p_2
    q}^*\) also form a projective system by the same maps.
  To see this, we show that if \((v_1 w) (v_2 w)^* \in \Bisp_{p_1
    q}\circ \Bisp_{p_2 q}^*\) has another representative that belongs to
  \(\tilde\alpha_{p_1,p_2}^q(U)\), then \((v_1 w x) (v_2 w x)^*\) has
  another representative that belongs to \(\tilde\alpha_{p_1,p_2}^{q
    t}(U)\).
  Another representative for \((v_1 w) (v_2 w)^*\) is of the form
  \((v_1 g_1\, g_1^{-1} w h) (v_2 g_2\, g_2^{-1} w h)\) for some
  \(g_1,g_2,h\in \Gr\) with \(\rg(g_1) = \rg(g_2) = \s(v_1) = \s(v_2)\),
  \(\rg(h) = \s(w)\), and \((v_1 g_1)(v_2 g_2)^* \in U\).
  Then \((v_1 g_1\, g_1^{-1} w h\, h^{-1} x) (v_2 g_2\, g_2^{-1} w h\,
  h^{-1} x)\) is another representative for \((v_1 w x) (v_2 w x)^*\)
  that belongs to \(\tilde\alpha_{p_1,p_2}^{q t}(U)\).
  This proves that the spaces \(\tilde\alpha_{p_1,p_2}^q(V) \setminus
  \tilde\alpha_{p_1,p_2}^q(U)\) also form a projective system.
  By assumption, all these spaces are compact and nonempty.
  Then their projective limit is compact and nonempty as well.
  Their projective limit is \(\tilde\alpha_{p_1,p_2}^\infty(V)
  \setminus \tilde\alpha_{p_1,p_2}^\infty(U)\), and so
  \(\tilde\alpha_{p_1,p_2}^\infty(U)\) does not contain
  \(\tilde\alpha_{p_1,p_2}^\infty(V)\).
  This finishes the proof.
\end{proof}

\begin{lemma}
  \label{lem:inductive_limits_B2}
  Let \(p_1,p_2,q_1,q_2,t_1,t_2\in P\) and let \(U_1\subseteq
  \Bisp_{p_1}\circ \Bisp_{p_2}^*\), \(U_2\subseteq \Bisp_{q_1}\circ
  \Bisp_{q_2}^*\), and \(U_3\subseteq \Bisp_{t_1}\circ \Bisp_{t_2}^*\)
  be slices such that \(\tilde\alpha_{p_1,p_2}^\infty(U_1) \cap
  \tilde\alpha_{q_1,q_2}^\infty(U_2)  = \emptyset\) and
  \(\tilde\alpha_{p_1,p_2}^\infty(U_1) \cup
  \tilde\alpha_{q_1,q_2}^\infty(U_2)  =
  \tilde\alpha_{t_1,t_2}^\infty(U_3)\).
  Then there are \(a,b,c\in P\) with \(p_j a = q_j b = t_j c\) for
  \(j=1,2\) and \(\tilde\alpha_{p_1,p_2}^a(U_1) \cap
  \tilde\alpha_{q_1,q_2}^b(U_2) = \emptyset\) and
  \(\tilde\alpha_{p_1,p_2}^a(U_1) \cup \tilde\alpha_{q_1,q_2}^b(U_2) =
  \tilde\alpha_{t_1,t_2}^c(U_3)\).
\end{lemma}

\begin{proof}
  If one of the sets \(\tilde\alpha_{p_1,p_2}^\infty(U_1)\),
  \(\tilde\alpha_{q_1,q_2}^\infty(U_2)\) or
  \(\tilde\alpha_{t_1,t_2}^\infty(U_3)\) is empty, then the statement
  follows from Lemma~\ref{lem:relations_base_gm}.
  So we may assume that they are all nonempty.
  Then all slices must belong to the same homogeneous
  component~\(\GMH_g\) because otherwise their union cannot be of
  the form \(\tilde\alpha_{t_1,t_2}^\infty(U_3)\).
  As a consequence, there are \(a,b,c\) with \(p_j a = q_j b = t_j c\)
  for \(j=1,2\).
  We may replace \(U_1\), \(U_2\) and~\(U_3\) by
  \(\tilde\alpha_{p_1,p_2}^a(U_1)\),
  \(\tilde\alpha_{q_1,q_2}^b(U_2)\), and
  \(\tilde\alpha_{t_1,t_2}^c(U_3)\)
  to arrange, without loss of generality, that already \(p_j = q_j =
  t_j\) for \(j=1,2\).
  Next, since \(\tilde\alpha_{p_1,p_2}^\infty(U_j) \subseteq
  \tilde\alpha_{p_1,p_2}^\infty(U_3)\) for \(j=1,2\),
  Lemma~\ref{lem:relations_base_gm} shows that there is \(a\in P\) with
  \(\tilde\alpha_{p_1,p_2}^a(U_j) \subseteq
  \tilde\alpha_{p_1,p_2}^a(U_3)\) for \(j=1,2\).
  Replacing \(U_j\) by \(\tilde\alpha_{p_1,p_2}^a(U_j)\) for
  \(j=1,2,3\), we may therefore arrange without loss of generality that
  already \(U_1\cup U_2 \subseteq U_3\).
  This makes \(U_1\cup U_2\) a slice.
  Our assumption \(\tilde\alpha_{p_1,p_2}^\infty(U_1) \cup
  \tilde\alpha_{p_1,p_2}^\infty(U_2)  =
  \tilde\alpha_{p_1,p_2}^\infty(U_3)\) is equivalent to
  \(\tilde\alpha_{p_1,p_2}^\infty(U_1 \cup U_2) =
  \tilde\alpha_{p_1,p_2}^\infty(U_3)\) because
  \(\tilde\alpha_{p_1,p_2}^\infty\) commutes with unions.
  Therefore, Lemma~\ref{lem:relations_base_gm} shows that there is
  \(a\in P\) with \(\tilde\alpha_{p_1,p_2}^a(U_1 \cup U_2) =
  \tilde\alpha_{p_1,p_2}^a(U_3)\).
  Therefore, we may assume without loss of generality that already
  \(U_1 \cup U_2 = U_3\).
  Finally, \(U_1 \cap U_2\) and~\(\emptyset\) are slices in
  \(\Bisp_{p_1}\circ\Bisp_{p_2}^*\), and
  Lemma~\ref{lem:relations_base_gm} also applies to them.
  Since
  \[
    \tilde\alpha_{p_1,p_2}^\infty(U_1 \cap U_2)
    \subseteq
    \tilde\alpha_{p_1,p_2}^\infty(U_1) \cap
    \tilde\alpha_{p_1,p_2}^\infty(U_2)
    = \emptyset
    = \tilde\alpha_{p_1,p_2}^\infty(\emptyset)
  \]
  Lemma~\ref{lem:relations_base_gm} provides \(a\in P\) with
  \(\tilde\alpha_{p_1,p_2}^a(U_1 \cap U_2) \subseteq
  \tilde\alpha_{p_1,p_2}^a(\emptyset) = \emptyset\).
  So the reduction steps above eventually lead to a true statement.
\end{proof}

\begin{theorem}
  \label{the:covariance_ring_is_Steinberg_of_groupoid_model}
  The ring homomorphism \(\varrho\colon \OOO_\F \to A_R(\GMH)\)
  constructed above is an isomorphism.
\end{theorem}

\begin{proof}
  We fix \(g\in G\) to use the inductive limit description of
  \(\OO[g]\subseteq \OOO_\F\) and the corresponding inductive limit
  desccription of \(\GMH_g\).
  Theorem~\ref{the:steinmod_is_quot_of_sum} describes \(A_R(X)\)
  for any space~\(X\) through generators and relations, given any
  ample base~\(\B\).
  Namely, let
  \[
    \B^{(2)} \defeq
    \setgiven{(V_1,V_2)\in\B^2}{V_1\cap V_2 =  \emptyset,\ 
      V_1 \sqcup V_2\in\B};
  \]
  then \(A_R(X)\) is isomorphic to the cokernel of the map
  \[
    \bigoplus_{(V_1,V_2)\in\B^{(2)}} R
    \to  \bigoplus_\B R,\qquad
    \delta_{(V_1,V_2)} \mapsto
    \delta_{V_1} + \delta_{V_2} - \delta_{V_1\sqcup V_2}.
  \]
  By Proposition~\ref{prop:base_on_Hg}, the sets
  \(\tilde\alpha^\infty_{p_1,p_2}(U)\) for \(p_1,p_2\in P\) with \(p_1
  p_2^{-1} =g\) and slices
  \(U\subseteq \Bisp_{p_1} \circ \Bisp_{p_2}^*\) form an ample
  base~\(\B_{\GMH_g}\) for~\(\GMH_g\).

  We describe \(A_R(\Bisp_{p_1}\circ \Bisp_{p_2}^*)\) as in
  Theorem~\ref{the:steinmod_is_quot_of_sum} using the ample
  base~\(\B_{p_1,p_2}\) of all slices.
  These ample bases with the maps~\(\tilde\alpha^q_{p_1,p_2}\)
  form an inductive system by Lemma~\ref{lem:lambda_pi_properties}.
  Lemma~\ref{lem:relations_base_gm} says that our chosen ample base
  for~\(\GMH_g\) is the inductive limit \(\varinjlim \B_{p_1,p_2}\).
  Lemma~\ref{lem:inductive_limits_B2} says that same for~\(\B^{(2)}\).
  Therefore, \(A_R(\GMH_g)\) is the cokernel of the induced map
  \[
    \bigoplus_{\varinjlim \B_{p_1,p_2}^{(2)}} R
    \to  \bigoplus_{\varinjlim \B_{p_1,p_2}} R,\qquad
    \delta_{(V_1,V_2)} \mapsto
    \delta_{V_1} + \delta_{V_2} - \delta_{V_1\sqcup V_2}.
  \]
  Now the construction of free \(R\)\nb-modules and the construction
  of cokernels both commute with arbitrary colimits and, in
  particular, with this inductive limit.
  So \(A_R(\GMH_g)\) is also isomorphic to the inductive limit of the
  system of cokernels, which is \(\OO[g] = \varinjlim A_R(\Bisp_{p_1}\circ
  \Bisp_{p_2}^*)\).
  The way we defined the maps, it is clear that this isomorphism is
  the restriction of~\(\varrho\) defined above.
\end{proof}

\section{From Steinberg algebras to groupoid \texorpdfstring{$\Cst$}{C*}-algebras}
\label{sec:Cstar-consequence}

In~\cite{Clark-Zimmerman:Steinberg_to_Cstar}, Clark and Zimmerman have
shown that the full groupoid \(\Cst\)\nb-algebra of an ample
groupoid~\(\Gr\) is the \(\Cst\)\nb-completion of the Steinberg
algebra of~\(\Gr\) over the complex numbers.
We are going to use this to prove that the groupoid
\(\Cst\)\nb-algebra of the groupoid model of a diagram of ample
groupoid correspondences is the Cuntz--Pimsner algebra of the product
system associated to the original diagram.
This result was also proven by Albandik~\cite{Albandik:Thesis}, even
for general \'etale locally compact groupoids.
In this section, we write \(A(\Gr)\) for \(A_\C(\Gr)\), the complex
Steinberg algebra of an ample groupoid~\(\Gr\).
This is a complex \Star{}algebra in a canonical way.

\begin{definition}
  Let~\(\norm{f}\) for \(f\in A(\Gr)\) be the supremum
  of~\(\norm{\pi(f)}\) for all \Star{}\alb{}homomorphisms \(\pi \colon
  A(\Gr) \to \Bound(\Hils)\) for some Hilbert
  space~\(\Hils\).
\end{definition}

This norm on \(A(\Gr)\) is a \(\Cst\)\nb-norm on \(A(\Gr)\).
The resulting \(\Cst\)\nb-completion of \(A(\Gr)\) is the groupoid
\(\Cst\)\nb-algebra of~\(\Gr\) by
\cite{Clark-Zimmerman:Steinberg_to_Cstar}*{Theorem~4.7}.

Using the Steinberg pseudofunctor \(A \colon \mathfrak{Gr} \to
\mathfrak{Rings}\), the Steinberg bimodule \(A(\Bisp)\) of a
groupoid correspondence \(\Bisp \colon \Gr[H] \leftarrow \Gr\)
becomes an \(A(\Gr[H]), A(\Gr)\)-bimodule.
Similarly, a pseudofunctor \(\Cst\colon \Grcat_\mathrm{inj}  \to \Corr\)
from the groupoid correspondence bicategory to the bicategory of
\(\Cst\)\nb-correspondences is defined in
\cite{Antunes-Ko-Meyer:Groupoid_correspondences}*{Theorem~7.13}.
This uses a variant of the bicategory of groupoid correspondences
where only injective \(2\)\nb-arrows are allowed.
We could have worked in this bicategory throughout this article, we
have never used a noninvertible \(2\)\nb-arrow.
In this section, we restrict to \(\Grcat_\mathrm{inj}\) to work both
with bimodules and \(\Cst\)\nb-correspondences.
In particular, a \(\Cst\)\nb-correspondence \(\Cst(\Bisp)\) is defined
in~\cite{Antunes-Ko-Meyer:Groupoid_correspondences}.
By Proposition~\ref{prop:Steinpro}, the bimodule \(A(\Bisp)\) is
proper if~\(\Bisp\) is proper.
Similarly, the  \(\Cst\)\nb-correspondence \(\Cst(\Bisp)\) is proper
if~\(\Bisp\) is proper by
\cite{Antunes-Ko-Meyer:Groupoid_correspondences}*{Theorem~7.14}.

Define a pairing \(\braket{-}{-} \colon A(\Bisp) \times A(\Bisp) \to
A(\Gr)\) by
\[
  \braket{\xi}{\eta}(g) \defeq \sum_{s(x) = r(g)} \conj{\xi(x)} \eta(x \cdot g)
\]
for \(g \in \Gr\), \(\xi,\eta\in A(\Bisp)\) as in
\cite{Antunes-Ko-Meyer:Groupoid_correspondences}*{Equation~(7.2)}.
It follows from
\cite{Antunes-Ko-Meyer:Groupoid_correspondences}*{Lemma~7.4} that
\(\braket{\charmap{U}}{\charmap{V}} = \charmap{\braket{U}{V}}\) for
two slices \(U,V\subseteq \Bisp\).

Let \(\ContS(\Bisp)\) be the linear span of all
functions in \(\Contc(U)\), extended by~\(0\) to~\(\Bisp\), for open
slices \(U\subseteq \Bisp\).
The \(\Cst\)\nb-correspondence \(\Cst(\Bisp)\) is defined as the
completion of \(\ContS(\Bisp)\) in the norm
\begin{equation}
  \label{eq:norm_on_A_Bisp}
  \norm{f}_{\Bisp}
  \defeq \norm{\braket{f}{f}}_{\Cst(\Gr)}^{1/2}.
\end{equation}
In the ample case, it suffices to take \(\Cont(U)\) for all compact,
open slices.
Since \(A(U) \subseteq \Cont(U)\) is dense for any compact slice~\(U\)
by the Stone--Weierstraß Theorem and since \(\norm{\braket{f}{f}} =
\norm{f}_\infty\) if \(f\in \Cont(U) \subseteq\ContS(\Bisp)\),
it follows that \(A(\Bisp)\) is a dense subspace of~\(\Cst(\Bisp)\).
Thus~\(\Cst(\Bisp)\) is also the completion of \(A(\Bisp)\) in the
norm in~\eqref{eq:norm_on_A_Bisp}.

Let \(\X = (P,\Gr, \Bisp_p, \mu_{p,q})\) be a diagram of proper
groupoid correspondences.
Let~\(\OOO\) be the covariance ring of the resulting diagram \(\A*\X\)
in \(\Ringspr\).
Even if~\(P\) is not an Ore monoid, we have described~\(\OOO\)
through generators and relations in
Section~\ref{sec:covariance_ring_Grcorr} and used this in
Corollary~\ref{cor:covariance_ring_star} to define an
antihomomorphism on it.
Here we may simply take the bases \(\B_{\Gr}\) and~\(\B_{\Bisp_p}\) to
consist of all compact open slices in \(\Gr\) and~\(\Bisp_p\),
respectively.
The same idea provides a conjugate-linear antihomomorphism when
\(R=\C\), so that~\(\OOO\) becomes a \Star{}algebra.
All the generators of~\(\OOO\) are partial isometries, so that any
\(\Cst\)\nb-seminorm on them is at most~\(1\).
Therefore, there is a maximal \(\Cst\)\nb-seminorm on~\(\OOO\).
We let \(\Cst(\OOO)\) be the \(\Cst\)\nb-completion of~\(\OOO\) in
this maximal \(\Cst\)\nb-seminorm.
If~\(P\) is an Ore monoid, then we have identified~\(\OOO\) with the
Steinberg algebra of the groupoid model~\(\GMH\) in
Theorem~\ref{the:covariance_ring_is_Steinberg_of_groupoid_model}.
This homomorphism maps all the generators in
Corollary~\ref{cor:presentation_covariance_ring_groupoid} to
characteristic functions of specific compact open slices in~\(\GMH\).
These formulas show that it is a \Star{}homomorphism.
Thus \cite{Clark-Zimmerman:Steinberg_to_Cstar}*{Theorem~4.7} implies
\(\Cst(\OOO) \cong \Cst(\GMH)\) if~\(P\) is an Ore monoid.

\begin{lemma}
  \label{lem:CPcovconduni}
  Let \((\kappa_p\colon A(\Bisp_p) \to \OOO)_{p \in P}\) be the
  universal covariant representation of~\(\A*\X\).
  If \(f, g \in A(\Bisp_p)\) for some \(p \in P\), then
  \[
    \kappa_p(f)^* \kappa_p(g) = \kappa_1(\braket{f}{g}).
  \]
\end{lemma}

\begin{proof}
  It suffices to prove this when \(f= \charmap{U}\) and
  \(g=\charmap{V}\) for compact open slices \(U,V\subseteq \Bisp_p\).
  Then \(\kappa_p(f) = \delta_U\), \(\kappa_p(g)^* = \delta_V^*\), and
  \[
    \kappa_p(f)^* \kappa_p(g)
    = \delta_U^* \delta_V
    = \delta_{\braket{U}{V}}
    = \kappa_1(\charmap{\braket{U}{V}})
    = \kappa_1(\braket{f}{g}).\qedhere
  \]
\end{proof}

We want to look at nondegenerate homomorphisms from the
\(\Cst\)\nb-completion \(\Cst(\OOO)\) to a  \(\Cst\)\nb-algebra~\(D\).
Here nondegeneracy means that \(\Cst(\OOO)\cdot D\) is equal to~\(D\).
This only implies that \(\OOO\cdot D\) is dense in~\(D\),
and this is not the same as being equal to~\(D\) because the
Cohen--Hewitt Factorisation Theorem does not apply to~\(\OOO\).
Therefore, we must be careful about nondegeneracy in our statements.

Let~\(D\) be a \(\Cst\)\nb-algebra.
Let \(\tilde\nu_p\colon A(\Bisp_p) \to D\) be maps that form a
nondegenerate covariant representation in the sense that
\(\tilde\nu_p(f) \tilde\nu_q(g) = \tilde\nu_{p q}(\mu_{p,q}(f,g))\)
for all \(p,q\in P\), \(f\in A(\Bisp_p)\), \(g\in A(\Bisp_p)\),
and \(\tilde\nu_p(A(\Bisp_p))D\subseteq D\) is dense for all \(p\in
P\).
The above data is a nondegenerate covariant representation of the
diagram \(\A*\X\) in the ring \(A(\Gr) D A(\Gr) \subseteq D\), and so
it generates a nondegenerate representation \(\varphi\colon \OOO \to
A(\Gr) D A(\Gr)\), which is nondegenerate into~\(D\) in the weaker
sense that \(\varphi(\OOO)\cdot D\) is dense in~\(D\).

\begin{proposition}
  \label{prop:covrepCstar}
  The homomorphism \(\varphi\colon \OOO \to D\) above is a
  \Star{}homomorphism if and only if the corresponding covariant
  representation \((\tilde\nu_p \colon A(\Bisp_p) \to D)_{p \in P}\)
  satisfies
  \[
    \tilde\nu_p(f)^* \tilde\nu_p(g) = \tilde\nu_e(\braket{f}{g})
    \qquad \text{for all } f, g \in A(\Bisp_p),\ p \in P.
  \]
\end{proposition}

\begin{proof}
  The natural isomorphism in Theorem~\ref{thm:cov_rep_F=O} is of the
  form \(\tilde\nu_p = \varphi \circ \kappa_p\) for all \(p \in P\) by
  the Yoneda Lemma.
  If~\(\varphi\) is a \Star{}homomorphism, then \(\tilde\nu_p(f)^* \tilde\nu_p(g) =
  \tilde\nu_e(\braket{f}{g})\) for \(f, g \in A(\Bisp_p)\), \(p\in P\)
  follows immediately from Lemma~\ref{lem:CPcovconduni}.
  Conversely, assume \(\tilde\nu_p(f)^* \tilde\nu_p(g) =
  \tilde\nu_e(\braket{f}{g})\) for all \(f, g \in A(\Bisp_p)\), \(p \in P\).
  Since~\(\varphi\) is a ring homomorphism,
  Lemma~\ref{lem:CPcovconduni} implies
  \begin{align*}
    \varphi(\kappa_p(f)^*) \cdot \tilde\nu_p(g)
      &= \varphi(\kappa_p(f)^* \kappa_p(g))
       = \varphi(\kappa_e(\braket{f}{g}))
       = \tilde\nu_e(\braket{f}{g})
       = \tilde\nu_p(f)^* \tilde\nu_p(g).
  \end{align*}
  Therefore, \((\varphi(\kappa_p(f)^*) - \tilde\nu_p(f)^*) \cdot \tilde\nu_p(g) = 0\)
  for all \(g \in A(\Bisp_p)\).
  Since \(\tilde\nu_p(A(\Bisp_p))D\) is dense in~\(D\) for any
  convariant representation, this implies \(\varphi(\kappa_p(f)^*) =
  \tilde\nu_p(f)^* = \varphi(\kappa_p(f))^*\).
  This implies that~\(\varphi\) is compatible with the involution on
  all the generators \(\delta_U\) and~\(\delta_U^*\).
  Thus, \(\varphi\) is a \Star{}homomorphism.
\end{proof}

\begin{theorem}
  \label{the:Cstar_covariance_ring}
  Let \(\X = (P,\Gr, \Bisp_p, \mu_{p,q})\) be a diagram of proper
  groupoid correspondences.
  Let~\(\OOO\) be the covariance ring of the corresponding diagram of
  proper bimodules \(\A*\X\).
  Equip~\(\OOO\) with the canonical involution and let~\(\Cst(\OOO)\)
  be its \(\Cst\)\nb-completion.
  This is the absolute Cuntz--Pimsner algebra of the proper product
  system~\(\Cst*\X\).
\end{theorem}

\begin{proof}
  By definition, \(\Cst*\X\) is the pseudofunctor from~\(P\) to the
  \(\Cst\)\nb-correspondence bicategory that we get by
  composing~\(\X\) with the pseudofunctor~\(\Cst\) from the groupoid to
  the \(\Cst\)\nb-correspondence bicategory.
  Such a pseudofunctor is identified with a product system
  in~\cite{Albandik-Meyer:Colimits}.
  It consists of proper \(\Cst\)\nb-correspondences because~\(\X\)
  consists of proper groupoid correspondences.
  Therefore, the absolute Cuntz--Pimsner algebra~\(\mathcal{Q}\) of
  this product system is defined.
  Being absolute means that we impose the Cuntz--Pimsner relation on
  all elements of \(\Cst(\Gr)\), which all act by compact operators on
  \(\Cst(\Bisp_p)\).
  As a consequence, a nondegenerate \Star{}homomorphism
  from~\(\mathcal{Q}\) to a \(\Cst\)\nb-algebra~\(D\) is a family of
  Cuntz--Pimsner covariant Toeplitz representations \(S_p\colon
  \Cst(\Bisp_p) \to D\) for all \(p\in P\) that also satisfy
  \(S_p(f)S_q(g) = S_{p q}(\mu_{p,q}(f\otimes g))\) for all \(f\in
  \Cst(\Bisp_p)\), \(g\in \Cst(\Bisp_q)\).
  By \cite{Albandik-Meyer:Product}*{Proposition~2.5}, the
  Cuntz--Pimsner covariance condition is equivalent to \(\Cst(\Bisp_p)
  \cdot D\) being dense in~\(D\).
  It suffices to check the conditions for Toeplitz representations on
  the dense subspaces \(A(\Bisp_p) \subseteq \Cst(\Bisp_p)\) and
  \(A(\Gr) \subseteq \Cst(\Gr)\).
  Therefore, \((S_p)\) is the same as a covariant representation of
  \(\A*\X\) in \(A(\Gr) D A(\Gr)\) that satisfies the extra condition
  \(\tilde\nu_p(f)^* \tilde\nu_p(g) = \tilde\nu_e(\braket{f}{g})\) for
  all \(f, g \in A(\Bisp_p)\), \(p \in P\), which is characteristic for
  Toeplitz representations.
  Now Proposition~\ref{prop:covrepCstar} provides a natural bijection
  between nondegenerate \Star{}homomorphisms \(\mathcal{Q} \to D\) and
  \(\OOO\to A(\Gr) D A(\Gr)\) or, equivalently, \(\Cst(\OOO)\to D\).
  This implies \(\Cst(\OOO) \cong \mathcal{Q}\).
\end{proof}

\begin{corollary}
  In the situation of Theorem~\textup{\ref{the:Cstar_covariance_ring}}, assume
  also that~\(P\) is an Ore monoid.
  Then the absolute Cuntz--Pimsner algebra of~\(\Cst*\X\) is
  isomorphic to the \(\Cst\)\nb-algebra of the groupoid model
  \(\Cst(\GMH)\).
\end{corollary}

\begin{proof}
  This follows from Theorem~\ref{the:Cstar_covariance_ring} and the
  isomorphism \(\Cst(\OOO) \cong \Cst(\GMH)\) that we deduced above
  if~\(P\) is an Ore monoid.
\end{proof}

\section{Examples: Higher-rank graphs and self-similarities}
\label{sec:eg}

In this section, we relate our theory to some previous work, in the
special cases where the underlying groupoids are just spaces viewed as
groupoids with only identity arrows or discrete groups.
In the first case, the resulting \(\Cst\)\nb-algebras and groupoid
models have already been studied in \cites{Albandik-Meyer:Product,
  Antunes-Ko-Meyer:Groupoid_correspondences, Meyer:Diagrams_models}.
They generalise Kumjian--Pask algebras of higher-rank graphs.
In the second case, we look at a special case of groupoid
correspondences on a group studied previously by
Stammeier~\cite{Stammeier:Irreversible} using a different language.

\subsection{Topological correspondences between spaces}
\label{sec:top_correspondences}

Let \(\Gr\) and~\(\Gr[H]\) be Hausdorff, locally compact spaces with
ample bases, viewed as ample groupoids with only identity arrows.
In this case, a groupoid correspondence \(\Bisp\colon \Gr\leftarrow
\Gr[H]\) is the same as a Hausdorff, locally compact space~\(\Bisp\)
with an ample base together with a continuous map \(\rg\colon \Bisp\to
\Gr\) and a local homeomorphism \(\s\colon \Bisp\to\Gr[H]\).
The correspondence is proper if and only if~\(\rg\) is a proper map.
(The space~\(\Bisp\) must be Hausdorff in order for the right
\(\Gr[H]\)\nb-action on it to be free and proper.)

Let~\(P\) be a monoid.
A \(P\)\nb-shaped diagram of proper groupoid correspondences is
the same as an action of~\(P\) on a space~\(\Gr\) by proper
topological correspondences as defined in
\cite{Albandik-Meyer:Product}*{Definition~4.4}; in this article, we
assume~\(\Gr\) to have an ample base in order to define Steinberg
algebras.
It is shown in~\cite{Albandik-Meyer:Product} that such an action
generates a product system over~\(P\), which gives rise to a
Cuntz--Pimsner algebra.
Our description of the covariance ring in
Section~\ref{sec:cov_ring_construction} is an algebraic analogue of
the description of the Cuntz--Pimsner algebra of the product system
in~\cite{Albandik-Meyer:Product}.
The spaces of maps \(\Hom_{-,F_1}(F_{p_1}, F_{p_2})F_1 \cong
F_{p_2}\otimes_{F_1} F_{p_1}^*\) used in
Section~\ref{sec:cov_ring_construction} replace the compact operators
between Hilbert modules used in~\cite{Albandik-Meyer:Product}.
We also replace the inductive limit \(\Cst\)\nb-algebras
in~\cite{Albandik-Meyer:Product} by a purely algebraic inductive
limit.

When \(P=\N^k\), then such a proper diagram is the same as a
row-finite topological rank-\(k\) graph.
If~\(\Gr\) is a set with the discrete topology (and still only identity
arrows), then we get the usual higher-rank graphs of
Kumjian--Pask~\cite{Kumjian-Pask:Higher_rank}.
These are described in a different language, using a small
category~\(\Lambda\) and a functor \(d\colon \Lambda\to\N^k\) with a
certain factorisation property.
More generally, for an arbitrary monoid, this construction was
generalised by Brown and Yetter~\cite{Brown-Yetter:Conduche}, where it
is realised that the relevant factorisation property is that of a discrete
Conduch\'e fibration.
As noticed in~\cite{Antunes-Ko-Meyer:Groupoid_correspondences}, a
discrete Conduch\'e fibration over a monoid~\(P\) is the same as a
\(P\)\nb-shaped diagram of groupoid correspondences where~\(\Gr\) is a
set made a topological groupoid with the discrete topology and only
identity arrows.
In one direction, the translation replaces~\((\Lambda,d)\) by \(\Gr=
d^{-1}(0)\) and \(\Bisp_p\defeq d^{-1}(p) \subseteq \Lambda\) for all
\(p\in \N^k\), with the multiplication maps coming from the
multiplication in~\(\Lambda\).
In the other direction, \(\Lambda = \bigsqcup_{p\in P} \Bisp_p\) with
object space~\(\Gr\), range and source as given on the
components~\(\Bisp_p\), and the multiplication coming from the maps
\(\Bisp_p \circ \Bisp_q \to \Bisp_{p q}\) and with the canonical
functor \(\Lambda \to P\).
The discrete Conduch\'e fibration property says exactly that the
multiplication in~\(\Lambda\) defines bijective maps \(\Bisp_p \circ
\Bisp_q \congto \Bisp_{p q}\).

In addition, the groupoid correspondences~\(\Bisp_p\) are proper if
and only if the maps  \(\rg\colon \Bisp_p \to \Gr^0\) are
finite-to-one.
In this case, the higher-rank graph is called \emph{row-finite}.
We need a stronger property: we call the discrete Conduch\'e fibration
\emph{regular} if the maps \(\rg\colon \Bisp_p \to \Gr^0\) are
finite-to-one and surjective for all \(p\in P\).

\begin{proposition}
  \label{pro:Kumjian-Pask}
  Let \(d\colon \Lambda\to P\) be a regular discrete Conduch\'e fibration,
  turned into a diagram of proper groupoid correspondences as above.
  Then the presentation of the covariance ring of the diagram in
  Corollary~\textup{\ref{cor:presentation_covariance_ring_groupoid}} is the
  same as the presentation of the  \(\Cst\)\nb-algebra of the discrete
  Conduch\'e fibration in
  \cite{Brown-Yetter:Conduche}*{Definition~2.7}.
\end{proposition}

\begin{proof}
  First, we consider the relation \(\delta_x \delta_y =  \delta_{x
    y}\) if \(\s(x)=\rg(y)\) and \(\delta_x\delta_y=0\) if
  \(\s(x)\neq\rg(y)\).
  If \(x,y\in\Bisp_1=\Gr = \Gr^0\), this simply says that the
  elements~\(\delta_x\) for \(x\in\Gr\) are orthogonal idempotents.
  We have seen in the proof of
  Theorem~\ref{the:presentation_covariance_ring_groupoid} that
  \(\delta_x = \delta_x^*\) for all \(x\in \Gr^0\).
  So these elements are orthogonal projections.
  In addition, \(\delta_{\rg(x)} \delta_x \delta_{\s(x)} = \delta_x\)
  and \(\delta_{\s(x)} \delta^*_x \delta_{\rg(x)} = \delta_x\)
  holds for all \(x\in \Bisp_p\), \(p\in\Gr\).

  The relation \(\delta_{x_1}^* \delta_{x_2} =
  \delta_{\braket{x_1}{x_2}}\) simplifies because \(\braket{x_1}{x_2}\)
  for \(x_1,x_2\in \Bisp_p\) is only defined if \(x_1 = x_2\), and then
  it is \(\s(x_1)=\s(x_2)\).
  So this relation says \(\delta_{x_1}^* \delta_{x_2} =0\) for
  \(x_1\neq x_2\) and \(\delta_x^* \delta_x = \delta_{\s(x)}\) for all
  \(x\in \bigsqcup\Bisp_p\).
  Together with the relations \(\delta_{\rg(x)} \delta_x
  \delta_{\s(x)} = \delta_x\) and \(\delta_{\s(x)} \delta^*_x
  \delta_{\rg(x)} = \delta_x\), this implies that~\(\delta_x\) is a
  partial isometry.
  The last relation in
  Corollary~\ref{cor:presentation_covariance_ring_groupoid} simplifies
  to
  \begin{equation}
    \label{eq:CP-relation_graph}
    \delta_x = \sum_{\setgiven{y\in\Bisp_p}{\rg(y)=x}} \delta_y\delta_y^*
  \end{equation}
  because the right \(\Gr\)\nb-orbits are just singletons.
  With these points clarified, it becomes easy to check that the
  relations in
  Corollary~\ref{cor:presentation_covariance_ring_groupoid} and those
  in \cite{Brown-Yetter:Conduche}*{Definition~2.7} are equivalent.
\end{proof}

If \(P=\N^k\), then the presentation above generalises the definition
of the Kumjian--Pask algebra of a higher-rank graph in
\cite{Kumjian-Pask:Higher_rank}*{Definition~1.5}.
If the discrete Conduch\'e fibration is row-finite but not regular,
then our definition makes sense.
It does \emph{not} give the usual algebra, however, because we impose
the relation~\eqref{eq:CP-relation_graph} even if
\(\rg^{-1}(x)=\emptyset\), so that the relation degenerates to
\(\delta_x = 0\).

It has always been known that the \(\Cst\)\nb-algebras of
higher-rank graphs are groupoid \(\Cst\)\nb-algebras.
For a regular higher-rank graph, the groupoid model that we study here
is the same as the path groupoid introduced
in~\cite{Kumjian-Pask:Higher_rank}.
When we replace~\(\N^k\) by a monoid~\(P\) that satisfies Ore
conditions, then the covariance ring is the Steinberg algebra of the
groupoid model by our main theorem.
The analogous result for the covariance \(\Cst\)\nb-algebra, which is
the Cuntz--Pimsner algebra of the associated product system,
has been shown both in \cite{Albandik-Meyer:Product} and
in~\cite{Brown-Yetter:Conduche}.
The result in~\cite{Albandik-Meyer:Product} also covers the case
when~\(\Gr\) becomes a locally compact space, which generalises
higher-rank topological graphs.
The groupoid model that we study here is the same as in
\cites{Albandik-Meyer:Product, Brown-Yetter:Conduche}.
Here we need to assume the maps \(\rg\colon \Bisp_p \to \Gr^0\) to be
surjective and proper in order for the covariance ring or
\(\Cst\)\nb-algebra to be defined and the correct object.

We see no need to discuss the groupoid model in detail in
this article because this has already been done previously.
The new aspect in this article compared to
\cites{Albandik-Meyer:Product,
  Antunes-Ko-Meyer:Groupoid_correspondences} is that we can now
realise the Steinberg algebra of the groupoid model as a covariance
algebra as well, and not just the groupoid \(\Cst\)\nb-algebra.

\subsection{Some higher-rank self-similar groups}

Now let~\(\Gr\) be a discrete group.
As noted in \cites{Albandik:Thesis, Meyer:Diagrams_models}, a proper
groupoid correspondence \(\Gr\leftarrow \Gr\) is the same as a self-similar
action of~\(\Gr\), without the assumption that the induced action on
the rooted tree is faithful.
The resulting groupoid model and \(\Cst\)\nb-algebra are the ones
defined by Nekrashevych~\cite{Nekrashevych:Cstar_selfsimilar}.
Our main result also interprets the Steinberg algebra of this groupoid
model as a covariance ring.
The presentation of the covariance ring in
Corollary~\ref{cor:presentation_covariance_ring_groupoid} is the same
as for Nekrashevych's \(\Cst\)\nb-algebra.

We shall not discuss this rank-\(1\) example any further here.
Instead, we consider a class of higher-rank self-similar groups
introduced in a different language by
Stammeier~\cite{Stammeier:Irreversible}.
Let~\(P\) be a monoid and let~\(P\) act on a group~\(\Gr\) by
injective endomorphisms.
That is, we are given injective group homomorphisms \(\vartheta_p\colon
\Gr \to \Gr\) for all \(p\in P\), subject to the conditions \(\vartheta_p
\vartheta_q = \vartheta_{p q}\) for all \(p,q\in P\) and \(\vartheta_1 =
\id_{\Gr}\).
The setup above is more general than the \emph{irreversible algebraic
  dynamical systems} considered in
\cite{Stammeier:Irreversible}*{Definition~1.5}.
Besides a certain independence condition for coprime \(p,q\in P\),
Stammeier assumes~\(P\) to be a countably generated, free Abelian
monoid.
In particular, Stammeier assumes~\(P\) to be Abelian, so that he only
considers Ore monoids.
The fact that Stammeier's \(\Cst\)\nb-algebra is the
\(\Cst\)\nb-algebra of a diagram of groupoid correspondences was
already worked out in the Master's thesis of Daniel Jentsch.

Let~\(\Bisp_p\) be~\(\Gr\) as a set with the left and right
\(\Gr\)\nb-actions \(g_1\bullet x\bullet g_2 \defeq g_1 x
\vartheta_p(g_2)\) for all \(g_1,x,g_2\in \Gr\).
This is a groupoid correspondence because~\(\vartheta_p\) is injective.
The orbit space~\(\Bisp_p/\Gr\) is the coset space
\(\Gr/\vartheta_p(\Gr)\).
Therefore, \(\Bisp_p\) is proper as a groupoid correspondence if and
only if \([\Gr: \vartheta_p(\Gr)]<\infty\).
When this happens for all \(p\in P\), Stammeier speaks of an
irreversible algebraic dynamical systems of \emph{finite type}.
We assume this from now on.

We define the multiplication maps \(\mu_{p,q} \colon \Bisp_p \circ
\Bisp_q \to \Bisp_{p q}\) by \(\mu_{p,q}(x_1,x_2) \defeq x_1
\vartheta_p(x_2)\).

\begin{lemma}
  \label{lem:irreversible_algebraic_diagram}
  The data above is a diagram of proper groupoid correspondences.
\end{lemma}

\begin{proof}
  The map~\(\mu_{p,q}\) is \(\Gr\)\nb-equivariant for the left and
  right \(\Gr\)\nb-actions on~\(\Bisp_p\) and satisfies
  \(\mu_{p,q}(x_1 \bullet g,x_2) = x_1 \vartheta_p(g) \vartheta_g(x_2) =
  \mu_{p,q}(x_1, g \bullet x_2)\).
  Since \(\vartheta_1 = \id_{\Gr}\), the correspondence~\(\Bisp_1\) is
  the identity correspondence on~\(\Gr\) and the maps \(\mu_{1,q}\)
  and~\(\mu_{p,1}\) are the canonical multiplication maps.
  The assumption \(\vartheta_p \circ \vartheta_q = \vartheta_{p q}\) is
  equivalent to the associativity of the multiplication maps.
\end{proof}

Since~\(\Gr\) is discrete,
Corollary~\ref{cor:presentation_covariance_ring_groupoid} applies and
shows that the covariance ring of the diagram above has the following
presentation:

\begin{corollary}
  \label{cor:Stammeier_presentation}
  The covariance ring of the diagram of rings and proper bimodules
  associated to \((P,\Gr,\Bisp_p,\mu_{p,q})\) above is generated by
  elements \(\delta_{g,p}\) and~\(\delta^*_{g,p}\) for \(g\in \Gr\),
  \(p\in P\), subject to the relations
  \begin{enumerate}
  \item \(\delta_{g,p} \delta_{h,q} = \delta_{g\vartheta_p(h),p q}\) and
    \(\delta^*_{g,p} \delta^*_{h,q} = \delta^*_{h\vartheta_q(g),q p}\) for
    all \(g,h\in \Gr\), \(p,q\in P\);
  \item  \(\delta_{g,p}^* \delta_{h,p} = \delta_{k}\)  for \(g,h\in
    \Gr\), \(p\in P\) if there is \(k\in\Gr\) with \(g \vartheta_p(k) =
    h\), and \(\delta_{g,p}^* \delta_{h,p} = 0\) otherwise;
  \item if \(p\in P\) and \(g_1,\dotsc,g_n\) is a system of
    representatives for the set of cosets \(\Gr/\vartheta_p(\Gr)\), then
    \(\delta_{1,1} = \sum_{j=1}^n \delta_{g_j,p}\delta^*_{g_j,p}\).
  \end{enumerate}
\end{corollary}

To compare this presentation with the one used by Stammeier
in~\cite{Stammeier:Irreversible}, we let \(u_g \defeq \delta_{g,1}\)
and \(u_g^* \defeq \delta^*_{g,1}\) for \(g\in\Gr\), \(1\in P\), and
\(s_p \defeq \delta_{1,p}\) and \(s_p^* \defeq \delta_{1,p}^*\) for
\(1\in \Gr\), \(p\in P\).
The relations also imply \(u_{g^{-1}} = u_g^*\) for \(g\in\Gr\)
because both are inverse to~\(u_g\).
So we may discard the generators~\(u_g^*\).

\begin{lemma}
  \label{lem:compare_presentation_Stammeier}
  Assume that we are given an irreversible algebraic dynamical system
  of finite type.
  The relations in the above corollary are equivalent to the relations
  \textup{(CNP1)--(CNP3)} in
  \textup{\cite{Stammeier:Irreversible}*{Definition~3.1}}.
\end{lemma}

\begin{proof}
  Stammeier's generators \(u_g\) and~\(s_p\) produce our generators as
  \(\delta_{g,p} = u_g s_p\) and \(\delta^*_{g,p} = s_p^* u_{g^{-1}}\)
  for all \(g\in\Gr\), \(p\in P\).
  The first relation in Corollary~\ref{cor:Stammeier_presentation}
  is equivalent to \(u_g u_h = u_{g h}\), \(s_p s_q = s_{p q}\),
  \(s_p u_g = u_{\vartheta_p(g)} s_p\), \(s_p^* s_q^* = s^*_{q p}\),
  and \(u_g s_p^* = s_p^* u_{\vartheta_p(g)}\) for all \(g,h\in\Gr\),
  \(p,q\in P\).
  In the presence of the first relation, the second relation in
  Corollary~\ref{cor:Stammeier_presentation} is equivalent to \(s_p^*
  s_p = u_1 = 1\) and \(s_p^* u_g s_p = 0\) if \(g \notin
  \vartheta_p(G)\) for all \(p\in P\), \(g\in \Gr\).
  The third relation says that \(1 = \sum_{j=1}^n u_{g_j} s_p s_p^*
  u_{g_j}^{-1}\) for all \(p\in P\), where \(g_1,\dotsc,g_n\) are a
  system of representatives for the cosets in \(\Gr/\vartheta_p(\Gr)\).
  Stammeier also has a relation about \(s_p^* u_g s_q\) for different
  \(p,q\in P\).
  Here he uses that~\(P\) has more structure to simplify this
  expression.
  We use the relation \(1 = \sum_{j=1}^n u_{h_j} s_{p q} s_{p q}^*
  u_{h_j}^*\) for the element \(p q \in P\) and a set of coset
  representatives \((h_j)\) for \(G/\vartheta_p\vartheta_q(G)\) to rewrite
  \[
    s_p^* u_g s_q
    = \sum_{j=1}^n s_p^* u_g u_{h_j} s_{p q} s_{p q}^* u_{h_j^{-1}}
    s_q
    = \sum_{j=1}^n s_p^* u_{g h_j} s_p s_q s_p^*
    s_q^* u_{h_j^{-1}} s_q.
  \]
  Our relations imply \(s_p^* u_{g h_j} s_p=0\) unless \(g h_j \in
  \vartheta_p(\Gr)\) and \(s_q^* u_{h_j^{-1}} s_q=0\) unless \(h_j^{-1}
  \in \vartheta_q(\Gr)\).
  Thus all summands vanish unless \(g = g h_j h_j^{-1} \in
  \vartheta_p(\Gr) \vartheta_q(\Gr)\).
  So we get \(s_p^* u_g s_q=0\) in this case, as in (CNP2) in
  \cite{Stammeier:Irreversible}*{Definition~3.1}.
  Assume now that \(g = \vartheta_p(g_1) \vartheta_q(g_2)\).
  Let \(t\) be a common lower bound for~\(p,q\), so that \(p
  = t p'\) and \(q = t q'\).
  Then the relations above imply
  \[
    s_p^* u_g s_q
    = u_{g_1} s_p^* s_q u_{g_2}
    = u_{g_1} s_{p'}^* s_t^* s_t s_{q'} u_{g_2}
    = u_{g_1} s_{p'}^* s_{q'} u_{g_2},
  \]
  which is (CNP2) in
  \cite{Stammeier:Irreversible}*{Definition~3.1} when~\(t\) is a
  greatest common lower bound.
  So our relations imply~(CNP2) in full generality if~\(P\) is as
  in~\cite{Stammeier:Irreversible}.
\end{proof}

As a consequence, our covariance ring is an algebraic analogue of the
\(\Cst\)\nb-algebra constructed by Stammeier
in~\cite{Stammeier:Irreversible} provided the irreversible algebraic
dynamical system is of finite type.
If~\(P\) is an Ore monoid, then the covariance ring is the Steinberg
algebra of a groupoid by our main theorem.
In addition, we have described the Steinberg algebra rather concretely
in Section~\ref{sec:cov_ring_construction}.

Let us sketch how the groupoid model looks like when~\(P\) is an Ore
monoid.
This is not used in~\cite{Stammeier:Irreversible}, because Stammeier prefers
other techniques to access the structure of his \(\Cst\)\nb-algebra.
Nevertheless, the object space~\(\Omega\), the most complicated
ingredient in the groupoid model, is used by Stammeier as well.
Namely, it is the spectrum of the commutative \(\Cst\)\nb-subalgebra
generated by the projections \(E_{g,p} \defeq u_g s_p s_p^* u_g^* =
\delta_{g,p}\delta^*_{g,p}\) for all \(g\in G\), \(p\in P\).
According to our recipe, the object space~\(\Omega\) is the projective
limit of the spaces \(\Bisp_p/\Gr = \Gr/\vartheta_p(\Gr)\).
Here the structure maps of the projective system \(\Bisp_{p q}/\Gr \to
\Bisp_p/\Gr\) are the canonical maps \(\Gr/\vartheta_{p q}(\Gr) \to
\Gr/\vartheta_p(\Gr)\) induced by the identity map on~\(\Gr\).
This is well-defined because \(\vartheta_{p q}(\Gr) =
\vartheta_p(\vartheta_q(\Gr)) \subseteq \vartheta_p(\Gr)\).
Let \(\pi_p\colon \Omega \to \Omega_p=\Gr/\vartheta_p(\Gr)\) be the
canonical map from the projective limit.
Then the sets \(\pi_p^{-1}(g\vartheta_p(\Gr))\subseteq \Omega\) for
\(g\in \Gr/\vartheta_p(\Gr)\) and \(p\in P\) form a base of compact open
subsets for the topology on~\(\Omega\).
This base is not closed under set differences.
Nevertheless, the functions \(\pi_p^*(\charmap{g\vartheta_p(\Gr)})\)
generate \(A_R(\Omega)\) as an \(R\)\nb-module because the sets
\(g\vartheta_p(\Gr)\) for \(g\in \Gr/\vartheta_p(\Gr)\) are already
disjoint.

Let~\(T\) be the semigroup with~\(0\) generated by the elements
\(\delta_{g,p} = u_g s_p\) and \(\delta_{g,p}^* = s_p^* u_{g^{-1}}\)
in the covariance ring.
It is a general feature of the groupoid model construction that the
isomorphism between the covariance ring and the Steinberg algebra of
the groupoid model maps each generator of~\(T\) to the characteristic
function of a compact open slice of the groupoid model.
Therefore, \(T\) is an inverse semigroup that is contained in the
inverse semigroup of slices of the groupoid model.
In fact, this inverse semigroup is such that the groupoid model is the
transformation group of~\(T\) acting on~\(\Omega\).
An argument as in the proof of
Lemma~\ref{lem:compare_presentation_Stammeier} shows that
\(\delta_{g,p}^*\delta_{h,q}\) may be rewritten to move all stars
to the right.
Therefore, any nonzero element of~\(T\) may be written as
\(\delta_{g,p}\delta_{h,q}^*\) for some \(g,h\in \Gr\), \(p,q\in
P\).
Therefore, the idempotent elements in~\(T\) are exactly the
elements~\(E_{g,p}\) above.

\end{document}